\documentclass[9pt,reqno,oneside,a4paper]{amsart}

\usepackage{subfiles}
\usepackage[english]{babel}
\usepackage[utf8]{inputenc}
\usepackage{amsmath}
\usepackage{amsthm}
\usepackage{amssymb}
\usepackage{amsfonts}
\usepackage{graphicx}
\usepackage[shortlabels]{enumitem}

\usepackage[top=30truemm,bottom=30truemm,left=30truemm,right=30truemm]{geometry}

\usepackage{adjustbox}
\usepackage{listings}
\usepackage{color}
\usepackage{upgreek}
\usepackage{mathtools}
\usepackage{color}
\usepackage{empheq}
\usepackage{bm}
\usepackage{csquotes}
\usepackage{tikz}
\usetikzlibrary{matrix}
\usepackage{tikz-cd}
\usepackage{setspace}
\usepackage[norelsize]{algorithm2e}
\usepackage{adjustbox}
\usepackage{mathrsfs}
\usepackage{stmaryrd}
\usetikzlibrary{decorations.pathmorphing,shapes}
\usepackage{rotating}
\usepackage{pdflscape}
\usepackage{subcaption}
\usepackage{afterpage}
\usepackage[all,cmtip]{xy}
\usepackage{tabu}
\usepackage{stackengine}
\usepackage{caption}
\usetikzlibrary{arrows,calc}
\usepackage{hyperref}
\usepackage{scalerel}
\usepackage{centernot}
\usepackage{float}
\usepackage{tabularx}
\usepackage{empheq}
\usepackage{multirow}
\usetikzlibrary{calc}
\usetikzlibrary{arrows.meta}
\usetikzlibrary{chains}

\makeatletter
\DeclareRobustCommand{\rvdots}{%
  \vbox{
    \baselineskip4\p@\lineskiplimit\z@
    \kern-\p@
    \hbox{.}\hbox{.}\hbox{.}
  }}
\makeatother

\allowdisplaybreaks

\newcommand{\tikzAngleOfLine}{\tikz@AngleOfLine}
\def\tikz@AngleOfLine(#1)(#2)#3{%
\pgfmathanglebetweenpoints{%
\pgfpointanchor{#1}{center}}{%
\pgfpointanchor{#2}{center}}
\pgfmathsetmacro{#3}{\pgfmathresult}%
}

\makeatletter
\patchcmd{\@setaddresses}{\indent}{\noindent}{}{}
\patchcmd{\@setaddresses}{\indent}{\noindent}{}{}
\patchcmd{\@setaddresses}{\indent}{\noindent}{}{}
\patchcmd{\@setaddresses}{\indent}{\noindent}{}{}
\makeatother

\newcommand{\gap}{\hspace{1pt}}
\newcommand{\ZZ}{\mathbb{Z}}

\newcommand{\NN}{\mathbb{N}}

\newcommand{\sfw}{\mathsf{w}}
\newcommand{\scrD}{\mathscr{D}}

\DeclareMathOperator{\GKdim}{GKdim}
\newcommand{\leftgr}{\text{\normalfont{-grmod}}}
\newcommand{\leftGr}{{\text{\normalfont-GrMod}}}
\newcommand{\Gr}{\text{\normalfont{GrMod-}}}
\newcommand{\gr}{\text{\normalfont{grmod-}}}

\DeclareMathOperator{\im}{im}

\DeclareMathOperator{\gldim}{gl.\hspace{-1pt}dim}

\DeclareMathOperator{\idim}{i.\hspace{-1pt}dim}

\DeclareMathOperator{\End}{End}

\DeclareMathOperator{\Spec}{Spec}

\DeclareMathOperator{\Hom}{Hom}

\DeclareMathOperator{\Ext}{Ext}

\DeclareMathOperator{\Tor}{Tor}

\DeclareMathOperator{\GL}{GL}
\DeclareMathOperator{\SL}{SL}
\DeclareMathOperator{\Tr}{Tr}

\DeclareMathOperator{\hdet}{hdet}

\DeclareMathOperator{\sspan}{span}

\DeclareMathOperator{\Autgr}{Aut_{gr}}
\DeclareMathOperator{\Aut}{Aut}

\newcommand{\id}{\text{\normalfont id}}

\newcommand{\hash}{\hspace{1pt} \# \hspace{1pt}}

\numberwithin{equation}{section}

\theoremstyle{definition}

\newtheorem{defn}[equation]{Definition}

\newtheorem{example}[equation]{Example}

\theoremstyle{plain}
\newtheorem{thm}[equation]{Theorem}
\newtheorem{prop}[equation]{Proposition}
\newtheorem{lem}[equation]{Lemma}
\newtheorem{cor}[equation]{Corollary}

\newtheorem{hypothesis}[equation]{Hypothesis}

\theoremstyle{remark}
\newtheorem{rem}[equation]{Remark}

\newcommand\sbullet[1][.5]{\mathbin{\ThisStyle{\vcenter{\hbox{%
  \scalebox{#1}{$\SavedStyle\bullet$}}}}}%
}

\newcommand\restr[2]{{
  \left.\kern-\nulldelimiterspace 
  #1 
  \right|_{#2} 
  }}

\newcounter{sarrow}

\newcounter{darrow}

\makeatletter
\renewcommand*\env@matrix[1][\arraystretch]{%
  \edef\arraystretch{#1}%
  \hskip -\arraycolsep
  \let\@ifnextchar\new@ifnextchar
  \array{*\c@MaxMatrixCols c}}
\makeatother

\newboolean{printBibInSubfiles}
\setboolean{printBibInSubfiles}{true}
\def\bib{\ifthenelse{\boolean{printBibInSubfiles}}
           { \bibliographystyle{amsalpha} \bibliography{thesisbib} }
       {}
}

\makeatletter
\tikzset{
  column sep/.code=\def\pgfmatrixcolumnsep{\pgf@matrix@xscale*(#1)},
  row sep/.code   =\def\pgfmatrixrowsep{\pgf@matrix@yscale*(#1)},
  matrix xscale/.code=%
    \pgfmathsetmacro\pgf@matrix@xscale{\pgf@matrix@xscale*(#1)},
  matrix yscale/.code=%
    \pgfmathsetmacro\pgf@matrix@yscale{\pgf@matrix@yscale*(#1)},
  matrix scale/.style={/tikz/matrix xscale={#1},/tikz/matrix yscale={#1}}}
\def\pgf@matrix@xscale{1}
\def\pgf@matrix@yscale{1}
\makeatother

\definecolor{mygray}{gray}{0.8}

\def\dotfill#1{\cleaders\hbox to #1{.}\hfill}
\makeatletter
\def\myrulefill{\leavevmode\leaders\hrule height .7ex width 1ex depth -0.6ex\hfill\kern\z@}
\makeatother

\title{Superpotentials and Quiver Algebras for Semisimple Hopf Actions}
\author{Simon Crawford}
\address{The University of Manchester, Alan Turing Building, Oxford Road, Manchester, M13 9PL, United Kingdom}
\email{simon.crawford@manchester.ac.uk}

\date{\today}

\begin{document}
\begin{abstract}
We consider the action of a semisimple Hopf algebra $H$ on an $m$-Koszul Artin--Schelter regular algebra $A$. Such an algebra $A$ is a derivation-quotient algebra for some twisted superpotential $\sfw$, and we show that the homological determinant of the action of $H$ on $A$ can be easily calculated using $\sfw$. Using this, we show that the smash product $A \hash H$ is also a derivation-quotient algebra, and use this to explicitly determine a quiver algebra $\Lambda$ to which $A \hash H$ is Morita equivalent, generalising a result of Bocklandt--Schedler--Wemyss. We also show how $\Lambda$ can be used to determine whether the Auslander map is an isomorphism. We compute a number of examples, and show how several results for the quantum Kleinian singularities studied by Chan--Kirkman--Walton--Zhang follow using our techniques.
\end{abstract}
\maketitle

\section{Introduction}

In representation theory, it is a common technique to express a $\Bbbk$-algebra $A$ as (the path algebra of) a quiver with relations in order to study properties of $A$. This approach has proved most fruitful in the study of finite-dimensional algebras since, from the point of view of representation theory, every finite-dimensional $\Bbbk$-algebra can be expressed in this way. Moreover, there has been considerable success in the past few decades in expressing many interesting families of infinite-dimensional algebras as quivers with relations, and using these descriptions to deduce representation-theoretic or geometric properties of associated algebras. \\
\indent As an example, if $G$ is a finite subgroup of $\SL(2,\Bbbk)$ acting naturally on a polynomial ring $R \coloneqq \Bbbk[u,v]$, then the invariant ring $R^G$ is called a Kleinian singularity, and these rings are of considerable interest in ring theory, representation theory, and geometry. In \cite{reitenvdb}, it was shown that a closely-related algebra called the skew group algebra, denoted $R \hash G$, is Morita equivalent to a certain quiver with relations, the preprojective algebra $\Pi(Q)$ of an extended Dynkin quiver. This result forms a component of the Auslander--McKay correspondence, and can be used to deduce representation-theoretic properties of $R^G$ and $R \hash G$, and geometric properties of the singular variety $\Spec R^G$ and its minimal resolution. For example, the maximal Cohen-Macaulay $R^G$-modules can be studied using $\Pi(Q)$, and it is possible to construct the minimal resolution of $\Spec R^G$ by applying quiver GIT to $\Pi(Q)$. This result of Reiten--Van den Bergh was later extended by Bocklandt--Schedler--Wemyss in \cite{bsw}, where it was shown that if $G$ is a finite subgroup of $\GL(n,\Bbbk)$ acting on a polynomial ring $R \coloneqq \Bbbk[x_1, \dots, x_n]$, then the skew group algebra $R \hash G$ is Morita equivalent to a quiver with relations. \\
\indent In recent years, there has been a strong interest in generalising results from commutative invariant theory, such as those comprising the Auslander--McKay correspondence, to a noncommutative setting; see \cite{kirkmansurvey} for a survey of recent results. A common approach in noncommutative invariant theory is to replace the polynomial ring $R$ by an Artin--Schelter (AS) regular algebra $A$, and the finite group $G \leqslant \GL(n,\Bbbk)$ by a semisimple Hopf algebra $H$. It is then possible to define an invariant ring $A^H$, as well as an analogue of the skew group algebra, called the smash product, denoted $A \hash H$. One can then study properties of $A^H$ and $A \hash H$, with $A^H$ playing the role of the coordinate ring of a noncommutative (often singular) variety. Seeking to better understand the invariant rings that arise in this way, the main goal of this paper is to prove a generalisation of the result of Bocklandt--Schedler--Wemyss to this setting. \\
\indent Our first result provides a key stepping stone in this direction. We recall that, in the commutative setting, the determinant of the elements of $G$ often controls properties of $R^G$ and $R \hash G$. For example, Watanabe's Theorem says that if $G$ is a finite subgroup of $\SL(n,\Bbbk)$ (i.e.\ every element of $G$ has trivial determinant) then the invariant ring $R^G$ is Gorenstein. The determinant has a noncommutative analogue, called the \emph{homological determinant}, which has been shown to control similar properties of $A^H$ and $A \hash H$; see \cite{jorgensen, kkzGor, ckwzi}, for example. Despite its ubiquity in noncommutative invariant theory, the homological determinant is notoriously difficult to compute. Our first result gives a method to compute the homological determinant for the action of a Hopf algebra on a large family of AS regular algebras, which extends \cite[Theorem 3.3]{mkoszul} and \cite[Theorem 2.1]{ckwzbin}. 

\begin{thm}[{Theorem \ref{hdetthm}}] \label{hdetintrothm}
Suppose that $A$ is an $m$-Koszul AS regular algebra, generated in degree 1, and that $H$ is a semisimple Hopf algebra acting inner-faithfully and homogeneously on $A$, so that $A$ is a left $H$-module algebra. Associated to $A$ is a twisted superpotential $\sfw \in V^{\otimes \ell}$ for some $\ell$, where $V = A_1$. The homological determinant $\hdet_A : H \to \Bbbk$ of the action of $H$ on $A$ satisfies
\begin{align*}
h \cdot \sfw = \hdet_A(h) \sfw,
\end{align*}
for all $h \in H$.
\end{thm}

\indent In \cite{ckwzbin}, Chan--Kirkman--Walton--Zhang studied actions of semisimple Hopf algebras $H$ on $2$-dimensional AS regular algebras such that the $H$-action on $A$ had \emph{trivial} homological determinant, and provided a complete classification of such actions. This condition on the homological determinant is analogous to the fact that, for Kleinian singularities, we require $G$ to be a subgroup of $\SL(2,\Bbbk)$ rather than $\GL(2,\Bbbk)$. In light of this, in the sequel \cite{ckwzi}, the authors called the resulting invariant rings $A^H$ \emph{quantum Kleinian singularities}, and showed that they had many properties in common with (commutative) Kleinian singularities. In particular, they began to develop aspects of the Auslander--McKay correspondence for quantum Kleinian singularities. \\
\indent It would therefore be desirable to have a version of the result of Reiten--Van den Bergh for quantum Kleinian singularities; better still, a generalisation of the result of Bocklandt--Schedler--Wemyss to the setting of Hopf algebra actions would aid our understanding of noncommutative invariant rings in higher dimensions, which is currently relatively unexplored. The first of these goals was achieved in the author's PhD thesis, \cite[Theorem 7.2.1]{simon}, provided that $H = \Bbbk G$ was a group algebra. The main result of this paper achieves the second of these goals, which also deals with the cases that were omitted from \cite{simon}.

\begin{thm}[Theorem \ref{pathalgthm}] \label{intropathalg}
Assume that $A$ and $H$ satisfy the hypotheses of Theorem \ref{hdetintrothm}. Then $A \hash H$ is Morita equivalent to an algebra $\Lambda$, which is the path algebra of a quiver with relations. Moreover, there is a vertex $0$ such that $A^H \cong e_0 \Lambda e_0$, where $e_0$ is the vertex idempotent corresponding to vertex $0$. The relations in $\Lambda$ are obtained from a twisted superpotential and can be written down explicitly, and the twist depends on the homological determinant of the action of $H$ on $A$.
\end{thm}

\indent We remark that the fact that one is able to express $A \hash H$ as a quiver with relations is not new; this follows by combining \cite[Theorem 4.1]{rrz} with \cite[Lemma 3.4]{reyesrog}. However, this approach is very inexplicit; it is nontrivial to determine both the quiver and the relations in this way. On the other hand, the proof of Theorem \ref{intropathalg} shows that both the underlying quiver of $\Lambda$ and the defining relations can be easily calculated from the $H$-action on $A$. \\
\indent As a precursor to Theorem \ref{intropathalg}, we first need show that the smash product $A \hash H$ is a \emph{derivation-quotient algebra}, a technical notion which is defined in Section \ref{prelimsec}. The majority of the effort in this paper is expended in proving the following result, from which Theorem \ref{intropathalg} follows relatively quickly:

\begin{thm}[Theorem \ref{smashisderquot}]
Assume that $A$ and $H$ satisfy the hypotheses of Theorem \ref{hdetintrothm}. Then $A \hash H$ is a derivation-quotient algebra for some twisted superpotential, where the twist depends on the homological determinant of the action of $H$ on $A$.
\end{thm}

A version of this result was established in \cite[Theorem 4.12]{wuskew}, with the additional restrictions that $H$ is a group algebra and that the homological determinant of the $H$-action on $A$ is trivial. Removing their assumption on the homological determinant is relatively straightforward; on the other hand, replacing the group algebra $\Bbbk G$ by an arbitrary semisimple Hopf algebra $H$ required new techniques. \\
\indent The remainder of the paper is devoted to applications of Theorem \ref{intropathalg} and examples. We outline one such application now. In noncommutative invariant theory, the \emph{Auslander map}, defined in Section \ref{prelimsec}, plays an important role in the representation theory of $A^H$. If this map is an isomorphism, then \cite[Theorems A and C]{ckwzii} shows, in particular, that there is a bijection between irreducible representations of $H$ and maximal Cohen-Macaulay $A^H$-modules, up to a degree shift. For this reason, it is important to know when this map is an isomorphism; a result of Bao--He--Zhang \cite[Theorem 0.3]{bhz} provides a computationally useful criterion to determine when this is the case. Using Theorem \ref{intropathalg}, their result can be reformulated as a statement in terms of $\Lambda$, as follows:

\begin{thm}[Corollary \ref{austhmcor}]
Assume that $A$ and $H$ satisfy the hypotheses of Theorem \ref{hdetintrothm}, and additionally assume that $A$ is GK-Cohen-Macaulay. Let $\Lambda$ be the algebra from Theorem \ref{intropathalg}. Then the Auslander map is an isomorphism if and only if $\GKdim \Lambda/\langle e_0 \rangle \leqslant \GKdim A - 2$.
\end{thm}

Theorem \ref{intropathalg} also has applications to the study of maximal Cohen-Macaulay modules over the invariant rings $A^H$; see Section \ref{mcmsec}. These results will be used in forthcoming work to study the Auslander--Reiten theory of two-dimensional noncommutative singularities. \\
\indent In the final section, we compute a number of examples, with a particular emphasis on the quantum Kleinian singularities of Chan--Kirkman--Walton--Zhang. In \cite{ckwzi}, the authors expended a great deal of effort showing that the Auslander map corresponding to a quantum Kleinian singularity $A^H$ is an isomorphism and that, in a number of cases, $A^H$ is isomorphic to a commutative Kleinian singularity. Using our techniques, we are able to deduce both of these results quickly; our proof also provides a conceptual reason as to why these results should be true.

\begin{thm}
Assume that $A^H$ is a quantum Kleinian singularity, in the sense of \cite{ckwzi} (see Table \ref{qkstable} in Section \ref{examplessec} for the classification). Then the following hold:
\begin{enumerate}[{\normalfont (1)},topsep=0pt,itemsep=0pt,leftmargin=*]
\item The Auslander map corresponding to $A^H$ is an isomorphism.
\item In cases \emph{{(d)}} ($n$ even), \emph{{(e)}}, and \emph{{(f)}}, the invariant ring $A^H$ is a commutative Kleinian singularity.
\end{enumerate}
\end{thm}

\noindent\textbf{Acknowledgements.} The author is a Heilbronn fellow at the University of Manchester. Portions of this work were completed at the University of Waterloo while the author was a postdoctoral fellow, and at the University of Washington while the author was in receipt of the Cecil King Travel Scholarship. The author is grateful for their financial support.

\section{Preliminaries} \label{prelimsec}

\subsection{Notations and conventions}
Throughout $\Bbbk$ will denote an algebraically closed field of characteristic 0. Let $R$ be an $\NN$-graded ring. We write $R\leftGr$ (respectively, $\Gr R$) for the category of $\ZZ$-graded left (respectively, right) $R$-modules with graded (i.e.\ degree-preserving) morphisms, and $R \leftgr$ (respectively, $\gr R$) for the full subcategory of finitely generated objects. We will usually work with left modules, and most definitions will be stated using this convention. Given $M \in R \leftGr$, we define $M[i]$ to be the graded module which is isomorphic to $M$ as an ungraded module, but which satisfies $M[i]_n = M_{i+n}$. \\
\indent If $M,N \in R \leftGr$, then we write $\Hom_{R \leftGr}(M,N)$ for the space of graded morphisms, and use analogous notation when $M$ and $N$ are finitely generated. If $M$ and $N$ are finitely generated, we can make an identification $\Hom_R(M,N) = \bigoplus_{i \in \ZZ} \Hom_{R \leftgr}(M,N[i])$, which gives $\Hom_R(M,N)$ a natural grading. Elements of the graded vector space $\Hom_R(M,N)$ will be referred to as homomorphisms. From this, $\Ext^i_R(M,N)$ inherits a natural grading for all $i \geqslant 0$. We write $M^* \coloneqq \Hom_\Bbbk(M,\Bbbk)$. If $M$ and $N$ are $(R,R)$-bimodules, then we write $\Hom({}_R M, {}_R N)$ (respectively, $\Hom(M_R,N_R)$) for the space of left (respectively, right) $R$-module homomorphisms between $M$ and $N$; we may also use this notation in other cases where Hom spaces may be ambiguous. Morphisms will be composed right-to-left, unless otherwise stated. \\
\indent We write $\idim M$ for the injective dimension of a module, and $\gldim R$ for the global dimension of a ring, tacitly assuming that the left and right global dimensions coincide. Unadorned tensors of objects will be over the field $\Bbbk$, i.e.\ $\otimes = \otimes_\Bbbk$. In general, tensors of elements will be unadorned.

\subsection{AS regular algebras and $m$-Koszul algebras}
Suppose that $A$ is a $\Bbbk$-algebra. We say that $A$ is \emph{connected graded} if it is $\NN$-graded with $A_0 = \Bbbk$. We say that $M \in A \leftGr$ is \emph{locally finite} if $\dim_\Bbbk M_i < \infty$ for all $i \in \mathbb{Z}$. \\
\indent If $A$ is a connected graded $\Bbbk$-algebra which is locally finite, then we may define the \emph{Gelfand-Kirillov (GK) dimension} of $M \in A \leftgr$ (which also allows $M = {}_A A$) as
\begin{align*}
\GKdim M \coloneqq \limsup_{n \to \infty} \log_n(\dim_\Bbbk M_n).
\end{align*}
The GK dimension can be defined in more general settings but, under our assumptions, the above definition is equivalent to the usual one by \cite[Proposition 6.6]{lenagan}. The GK dimension serves as a sensible dimension function for noncommutative rings; for example, if $A$ is commutative, then it agrees with the Krull dimension of $A$. \\
\indent We now define the algebras which will serve as noncommutative analogues of commutative polynomial rings:

\begin{defn} \label{ASregdef}
Let $A$ be a connected graded $\Bbbk$-algebra, and also write $\Bbbk = A/A_{\geqslant 1}$ for the trivial module. We say that $A$ is \emph{Artin--Schelter Gorenstein} (or AS Gorenstein) \emph{of dimension $d$} if:
\begin{enumerate}[{\normalfont (1)},topsep=1pt,itemsep=1pt,leftmargin=35pt]
\item $\idim {}_A A = \idim A_A = d < \infty$, and 
\item $\Ext^i(_{A} \Bbbk,{}_A A) \cong \left \{
\begin{array}{cl}
0 & \text{if } i \neq d \\
\Bbbk[\ell]_A & \text{if } i = d
\end{array}
\right . $ \hspace{2pt} as graded right $A$-modules, for some integer $\ell$, and a symmetric condition holds for $\Ext^i(\Bbbk_{A},A_A)$, with the same integer $\ell$. We call $\ell$ the \emph{Gorenstein parameter} of $A$.
\end{enumerate}
If, moreover
\begin{enumerate}[{\normalfont (1)},topsep=1pt,itemsep=1pt,leftmargin=35pt]
\item[(3)] $\gldim A = d$, and
\item[(4)] $A$ has finite GK dimension,
\end{enumerate}
then we say that $A$ is \emph{Artin--Schelter regular} (or AS regular) \emph{of dimension $d$}.
\end{defn}

Unless otherwise stated, if $A$ is an AS regular algebra then we will assume that it is generated in degree 1. AS regular algebras are often thought of as noncommutative analogues of polynomial rings, and they have good ring-theoretic properties; in particular, all known examples are noetherian domains, and it is conjectured that this is always the case. In particular, if $A$ is a commutative AS regular algebra, then it is a polynomial ring. In all known examples the GK dimension and global dimension of an AS regular algebra coincide. \\
\indent When $A$ is AS regular of dimension $2$ and generated in degree $1$ then, up to isomorphism, it is one of the following algebras:
\begin{align*}
\Bbbk_q[u,v] \coloneqq \frac{\Bbbk \langle u,v \rangle}{\langle vu - quv \rangle}\quad (q \in \Bbbk^\times), \qquad \Bbbk_J[u,v] \coloneqq \frac{\Bbbk \langle u,v \rangle}{\langle vu - uv - u^2 \rangle}.
\end{align*}
These are called the \emph{quantum plane} and \emph{Jordan plane}, respectively. We will also work with the algebras
\begin{align*}
\frac{\Bbbk \langle u,v \rangle}{\langle v^2-u^2 \rangle}, \qquad \frac{\Bbbk \langle u,v \rangle}{\langle v^2+u^2 \rangle},
\end{align*}
which are both isomorphic to $\Bbbk_{-1}[u,v]$. \\
\indent It is straightforward to show that we can always write an AS regular algebra $A$ in the form
\begin{align}
A = \frac{T_\Bbbk(V)}{\langle \mathcal{R} \rangle}, \label{mhomogeneous}
\end{align}
where $V = A_1$ is a finite-dimensional vector space, and $\mathcal{R}$ is a set of relations in $V^{\otimes m}$ for some $m$. An algebra with a presentation of the form (\ref{mhomogeneous}) is said to be \emph{$m$-homogeneous}. \\
\indent Since $V$ is finite-dimensional, we have an identification
\begin{align}
\varphi : (V^*)^{\otimes k} \to (V^{\otimes k})^*, \quad \varphi(f_1 \otimes \dots \otimes f_k)(v_1 \otimes \dots \otimes v_k) = f_k(v_1) f_{k-1}(v_2) \dots f_1 (v_k). \label{tensordual}
\end{align}
We will usually suppress the map $\varphi$. We remark that there are competing conventions (see, for example, \cite[Footnote 3]{mkoszul}) when identifying $(V^*)^{\otimes k}$ with $(V^{\otimes k})^*$. However, for our purposes the above identification is most natural, since if $V$ is a left $H$-module for some Hopf algebra $H$, it ensures that the evaluation map $(V^*)^{\otimes k} \otimes V^{\otimes k} \to \Bbbk$ is a morphism of $H$-modules. With this identification, we set 
\begin{align*}
\mathcal{R}^\perp = \{ f \in (V^*)^{\otimes k} \mid f(\mathcal{R}) = 0 \},
\end{align*}
and then define the $m$\emph{-homogeneous dual} of $A$ to be
\begin{align*}
A^! \coloneqq \frac{T_\Bbbk(V^*)}{\langle \mathcal{R}^\perp \rangle}.
\end{align*}
This algebra is also connected graded. \\
\indent Now fix a basis $\{ v_1, \dots, v_r \}$ of $V=A_1$, and let $\{ \phi_1, \dots, \phi_r \}$ be the corresponding dual basis of $V^* = A_1^!$. Define $e = \sum_{i=1}^n v_i \otimes \phi_i$, which is independent of our choice of basis. Let $P_j = A \otimes (A_j^!)^*$, which is a free left $A$-module. Then ``right multiplication by $e$'' gives a map
\begin{align*}
\mbox{} \cdot e : P_j \to P_{j-1}, \qquad a \otimes f \mapsto \sum_{i=1}^n a v_i \otimes f \phi_i, \quad \text{where } f \phi_i : A^!_{j-1} \to \Bbbk, \hspace{5pt} f \phi_i(x) = f(\phi_i x).
\end{align*}
Since $A$ is $m$-homogeneous, $(\cdot e)^m : P_j \to P_{j-m}$ is the zero map. Therefore, there is a complex of left $A$-modules
\begin{align*}
P^\bullet : \quad \dots \xrightarrow{(\cdot e)^{m-1}} P_{2m+1} \xrightarrow{\hspace{7pt} \cdot e \hspace{7pt}} P_{2m} \xrightarrow{(\cdot e)^{m-1}} P_{m+1} \xrightarrow{\hspace{7pt} \cdot e \hspace{7pt}} P_{m} \xrightarrow{(\cdot e)^{m-1}} P_{1} \xrightarrow{\hspace{7pt} \cdot e \hspace{7pt}} P_{0} \to \Bbbk \to 0,
\end{align*}
called the $m$-\emph{Koszul complex} for $A$. It will be convenient to define a map $\rho$ as follows:
\begin{align*}
\rho: \mathbb{N}_0 \to \mathbb{N}_0, \quad \rho(i) = \left \{
\begin{array}{cc}
    \tfrac{mi}{2} & \text{if } i \text{ is even}, \\[3pt]
    \tfrac{m(i-1)}{2} + 1 & \text{if } i \text{ is odd.}
\end{array}
\right.
\end{align*}
More explicitly, $\rho$ maps the sequence of integers $0,1,2,3,4,5, \dots$ to $0,1,m,m+1,2m,2m+1, \dots $. In particular, the $i$th projective module appearing in the Koszul complex $P^\bullet$ is $P_{\rho(i)}$.

\begin{defn}[\cite{berger1}]
If $P^\bullet$ is exact, then we call $A$ an $m$\emph{-Koszul algebra}. If this is the case, then $P^\bullet$ is a minimal projective resolution of ${}_A \Bbbk$.
\end{defn}

When $m=2$, this recovers the usual notion of a Koszul algebra. We remark that, in \cite{berger1}, the author defines an $m$\emph{-Koszul algebra} to be an $m$-homogeneous algebra for which $\Tor_i^A(\Bbbk,\Bbbk)$ is concentrated in a single degree for $i \geqslant 3$, and then showed that this is equivalent to the above definition in \cite[Theorem 2.11]{berger1}. \\
\indent If we further assume that $A$ is AS regular, then we can write down $P^\bullet$ more precisely. Suppose that $A$ has global dimension $d$ and Gorenstein parameter $\ell$. In this case, we have $\rho(d) = \ell$, and the left-hand part of the $m$-Koszul complex has the following form \cite[Proposition 2.9]{mkoszul}:
\begin{align*}
P^\bullet : \quad 0 \to P_\ell \xrightarrow{\hspace{7pt} \cdot e \hspace{7pt}} P_{\ell-1} \xrightarrow{(\cdot e)^{m-1}} \dots
\end{align*}

\subsection{Superpotentials and derivation-quotient algebras}
By a result of Dubois-Violette, every $m$-Koszul AS regular algebra is a \emph{derivation-quotient algebra}. We recall some definitions before stating this result precisely. \\
\indent Throughout this subsection, let $V$ be a finite-dimensional $\Bbbk$-vector space, with basis $\{v_1, \dots, v_r \}$. For any positive integer $\ell$, there is a linear map 
\begin{align*}
\theta: V^{\otimes \ell} \to V^{\otimes \ell}, \qquad \theta( u_1 \otimes u_2 \otimes \dots \otimes u_{\ell-1} \otimes u_{\ell}) = u_2 \otimes \dots \otimes u_{\ell} \otimes u_1.
\end{align*}

\begin{defn} \label{superpotentialdefn1}
Let $\sfw \in V^{\otimes \ell}$. 
\begin{enumerate}[{\normalfont (1)},leftmargin=*,topsep=0pt,itemsep=0pt]
\item We say that $\sfw$ is a \emph{superpotential} if $\theta(\sfw) = \sfw$.
\item Let $\sigma \in \GL(V)$. We say that $\sfw$ is a $\sigma$\emph{-twisted superpotential} if $(\id^{\otimes (\ell-1)} \otimes \sigma) \theta(\sfw) = \sfw$. Moreover, we say that $\sfw$ is a \emph{twisted superpotential} if it is a $\sigma$-twisted superpotential for some $\sigma \in \GL(V)$.
\end{enumerate}
\end{defn}

\begin{rem}
Our terminology differs from both \cite{bsw} and \cite{mkoszul}; the former includes a coefficient of $(-1)^{\ell+1}$ in the definition of a (twisted) superpotential, while \cite{mkoszul} would call the above a $\sigma^{-1}$-twisted superpotential.
\end{rem}

\begin{defn} \label{derquodef1}
Let $\sfw \in V^{\otimes \ell}$. We define
\begin{align*}
\partial^i \sfw \coloneqq \big\{ \psi_1 \otimes \psi_2 \otimes \dots \psi_i \otimes \id^{\ell-i}(\sfw) \hspace{3pt} \big| \hspace{3pt} \psi_1, \dots, \psi_i \in V^* \big \}.
\end{align*}
We then define the \emph{derivation-quotient algebra of $\sfw$ of order $i$} to be
\begin{align*}
\scrD(\sfw,i) \coloneqq \frac{T_{\Bbbk}(V)}{\langle \partial^i \sfw \rangle}.
\end{align*}
\end{defn}

\indent In practice, the relations in $\scrD(\sfw,i)$ are obtained from $\sfw$ by \emph{formal differentiation} on the left with respect to all length $i$ expressions in the $v_j$. This terminology appears to be standard, but a more accurate term might be ``formal left deletion''. \\
\indent The following result establishes the connection between $m$-Koszul AS regular algebras and derivation-quotient algebras:

\begin{thm}[{\cite[Theorem 11]{dubois}}]
Suppose that $A = T_\Bbbk(V)/\langle \mathcal{R} \rangle$ is an $m$-Koszul AS regular algebras of Gorenstein parameter $\ell$. Then there exists a superpotential $\sfw \in V^{\otimes \ell}$ such that
\begin{align*}
A \cong \scrD(\sfw, \ell-m).
\end{align*}
\end{thm}

\begin{example} \label{downupisderquot}
Consider the noetherian graded down-up algebras of Benkart-Roby, \cite{benkart}:
\begin{align*}
A(\alpha, \beta) = \frac{\Bbbk \langle u,v \rangle}{\left\langle 
\begin{array}{c}
v^2u = \alpha vuv + \beta u v^2 \\
vu^2 = \alpha uvu + \beta u^2 v
\end{array}\right \rangle}, \qquad \text{where } \alpha \in \Bbbk, \gap \beta \in \Bbbk^\times.
\end{align*}
By \cite[4.4]{kirkmandu}, these algebras are AS regular of global dimension 3, have Gorenstein parameter 4, and are 3-Koszul. Therefore, if we set $V = \sspan \{u,v\}$, then there exists a (twisted) superpotential $\sfw \in V^{\otimes 4}$ such that $A(\alpha, \beta) = \scrD(\sfw,1)$. Indeed, if we set (where here we omit the tensors)
\begin{gather*}
\sfw = uv^2u - \alpha uvuv - \beta u^2v^2 - \beta^{-1} v^2u^2 + \alpha \beta^{-1} vuvu + v u^2 v, \\
\sigma: V \to V, \qquad u \mapsto -\beta^{-1} u, \quad v \mapsto -\beta v
\end{gather*}
then it is straightforward to check that $\sfw$ is a $\sigma$-twisted superpotential and that $A(\alpha,\beta) = \scrD(\sfw,1)$. 
\end{example}

In \cite{bsw}, the authors generalised the above constructions to the setting of bimodules over a semisimple algebra $H$. We now recall some material from \cite[Section 2]{bsw}. \\
\indent Throughout this subsection, let $H$ be a finite-dimensional semisimple $\Bbbk$-algebra and let $M$ be an $(H,H)$-bimodule. We will be concerned with three different possible duals of $M$:
\begin{itemize}[leftmargin=25pt,topsep=0pt,itemsep=0pt]
\item $M^* \coloneqq \Hom_\Bbbk(M,\Bbbk)$, the space of $\Bbbk$-linear morphisms from $M$ to $\Bbbk$, which is an $(H,H)$-bimodule via $(h \gap\psi\gap k)(m) = \psi(kmh)$;
\item $M^{*R} \coloneqq \Hom(M_H, H_H)$, the space of right $H$-module morphisms from $M$ to $H$, which is an $(H,H)$-bimodule via $(h \gap\psi\gap k)(m) = h \gap\psi (km)$; and
\item $M^{*L} \coloneqq \Hom({}_H M, {}_H H)$, the space of left $H$-module morphisms from $M$ to $H$, which is an $(H,H)$-bimodule via $(h \gap\psi\gap k)(m) = \psi(mh)k$.
\end{itemize}
These duals give rise to three contravariant functors, $(-)^*$, $(-)^{*R}$ and $(-)^{*L}$, from the category of $(H,H)$-bimodules to itself. These functors are not canonically isomorphic, but can be identified by specifying a \emph{trace function} $\Tr : H \to \Bbbk$ which is nondegenerate in the sense that the associated form $H \otimes H \to \Bbbk$ defined by $h \otimes k \mapsto \Tr(hk)$ is bilinear and nondegenerate. Using this, one can define natural isomorphisms $R$ and $L$ from the $\Bbbk$-dual to the other two duals by requiring
\begin{align}
\Tr((R\psi)(m)) = \psi(m) = \Tr((L\psi)(m)) \label{tracefunction}
\end{align}
for all $\psi \in M^*$ and all $m \in M$. Henceforth we work with a fixed trace function $\Tr$. We then obtain $H$-bimodule morphisms 
\begin{gather*}
\llbracket -,- \rrbracket : M^* \otimes_H M \to H, \qquad {} \llbracket\psi,m\rrbracket = (R \psi)(m), \\
\llbracket-,-\rrbracket : M \otimes_H M^* \to H, \qquad \llbracket m,\psi\rrbracket = (L \psi)(m).
\end{gather*}
(The same notation is used for both pairings, but which we are using will be clear from context.)  For $i \leqslant j$, the first of these maps can be extended as follows:
\begin{gather*}
\llbracket-,-\rrbracket : (M^*)^{\otimes_H i} \otimes_H M^{\otimes_H j} \to M^{\otimes_H (j-i)}, \\
\llbracket\psi_1 \otimes \dots \otimes \psi_i, m_1 \otimes \dots m_j \rrbracket = \llbracket\psi_1, \llbracket\psi_2,  \dots, \llbracket\psi_{i-1} ,\llbracket\psi_i,m_1\rrbracket m_2\rrbracket \dots m_{i-1}\rrbracket m_i\rrbracket \hspace{3pt} m_{i+1} \otimes \dots \otimes m_j.
\end{gather*}
We can extend the second map in a similar fashion. In both cases, if $i = j$ then the codomain is $H$. \\
\indent Let $\nu$ be a graded $\Bbbk$-algebra automorphism of the tensor algebra $T_H(M)$ (so in particular $\nu$ restricts to automorphisms of $H$ and $M$) such that the trace is invariant under $\nu$. This gives rise to an automorphism $\nu^*$ of $M^*$ by pulling back: $\nu^*(\psi) = \psi \circ \nu$. Later, it will be convenient to write $\nu^*(\psi) = \psi^{\nu^*}$, so henceforth we use this notation. \\
\indent We are now able to define (twisted) superpotentials in this new setting.

\begin{defn} \label{superpotentialdefn2}
Let $\sfw \in M^{\otimes_H \ell}$.
\begin{enumerate}[{\normalfont (1)},leftmargin=*,topsep=0pt,itemsep=0pt]
\item We say that $\sfw$ is a \emph{weak potential} if it commutes with the action of $H$:
\begin{align*}
h \cdot \sfw = \sfw \cdot h \qquad \text{for all } h \in H.
\end{align*}
A weak potential is called a \emph{superpotential} if 
\begin{align*}
\llbracket\psi, \sfw\rrbracket = \llbracket\sfw, \psi\rrbracket \qquad \text{for all } \psi \in M^*.
\end{align*}
\item We say that $\sfw$ is a \emph{($\nu$-)twisted weak potential} if
\begin{align*}
h \cdot \sfw = \sfw \cdot \nu(h) \qquad \text{for all } h \in H.
\end{align*}
A twisted weak potential is called a \emph{($\nu$-)twisted superpotential} if
\begin{align*}
\llbracket\psi^{\nu^*}, \sfw \rrbracket = \llbracket\sfw, \psi\rrbracket \qquad \text{for all } \psi \in M^*.
\end{align*}
\end{enumerate}
\end{defn}

If $\sfw \in M^{\otimes_H \ell}$ is a $\nu$-twisted weak potential, then for every non-negative integer $i \leqslant \ell$ there is a bimodule morphism
\begin{align*}
\partial_\sfw^i : (M^*)^{\otimes_H i} \otimes_H {}_\nu H \to M^{\otimes_H (\ell-i)}, \quad (\psi_1 \otimes \dots \otimes \psi_i) \otimes h =  \llbracket\psi_1 \otimes \dots \otimes \psi_i, \sfw h\rrbracket,
\end{align*}
where ${}_\nu H$ denotes the $(H,H)$-bimodule with right action given by multiplication, and left action given by $h \cdot k \coloneqq \nu(h) k$.

\begin{defn} \label{derquodef2}
Let $\sfw \in M^{\otimes_H \ell}$ be a twisted weak potential. The \emph{derivation-quotient algebra of $\sfw$ of order $i$} is
\begin{align*}
\scrD(\sfw,i) \coloneqq \frac{T_H(M)}{\langle \im \partial_\sfw^i \rangle}.
\end{align*}
\end{defn}

In the case when $H = \Bbbk$, Definitions \ref{superpotentialdefn2} and \ref{derquodef2} are equivalent to Definitions \ref{superpotentialdefn1} and \ref{derquodef1}, respectively. \\
\indent Much of the above theory will be applied to path algebras of quivers, so we recall some definitions. A \emph{quiver} is a directed multigraph, and we will always assume that our quivers are \emph{finite}, in the sense that they have finitely many vertices and edges. We will usually assume that our quivers have vertex set $\{0, 1, \dots, n\}$. We can equip $Q$ with \emph{head} and \emph{tail} maps, which map an arrow $\alpha : i \to j$ to the vertex $j$ and the vertex $i$, respectively. A \emph{path (of length $\ell$} in $Q$ is a sequence of arrows $p = \alpha_1 \dots \alpha_\ell$ such that $h(\alpha_i) = t(\alpha_{i+1})$ for $1 \leqslant i < \ell$ (in particular, we compose paths from left to right). We can extend the head and tail maps to paths in the obvious way. \\
\indent Given a finite quiver, we can form a $\Bbbk$-algebra $\Bbbk Q$ called the \emph{path algebra} of $Q$ as follows. As a vector space, $\Bbbk Q$ has a basis consisting of paths in the quiver (including the \emph{stationary paths} $e_i$ where we simply remain at vertex $i$), and multiplication of paths is given by concatenation, where defined: 
\begin{align*}
p \cdot q \coloneqq \left \{
\begin{array}{cl}
pq & \text{if } h(p) = t(q), \\
0 & \text{otherwise},
\end{array}
\right.
\end{align*}
and then extended linearly to all of $\Bbbk Q$. The elements $e_i$ are pairwise orthogonal idempotents, and the unit element in $\Bbbk Q$ is $1 = e_0 + e_1 + \dots + e_n$. The path algebra has a natural grading given by path length. \\
\indent If $\Bbbk Q$ is a path algebra, a \emph{relation} $\rho$ in $\Bbbk Q$ is an element of $(\Bbbk Q)_m$ for some $m \geqslant 2$, where every path in $\rho$ has the same head and tail, i.e.\ $\rho \in e_i (\Bbbk Q)_m e_j$ for some vertices $i$ and $j$. If $I$ is a two-sided ideal of $\Bbbk Q$ generated by relations, then we call $\Bbbk Q/I$ a \emph{path algebra with relations} or a \emph{quiver with relations}. \\
\indent A particularly important family of quivers with relations are preprojective algebras:

    

\begin{defn} \label{preprojdef}
Let $Q$ be a quiver without loops, and define the \emph{double} of $\overline{Q}$ by adding an arrow $\overline{\alpha} : j \to i$ if there is an arrow $\alpha: i \to j$ in $Q$. The \emph{preprojective algebra} of $Q$ is then the quiver with relations
\begin{align*}
\Pi(Q) \coloneqq \Bbbk \overline{Q} / \Big \langle \sum_{\alpha \in Q} \alpha \overline{\alpha} - \overline{\alpha} \alpha \Big \rangle.
\end{align*}
Observe that, by pre- and post-multiplying the defining relation by $e_i$, for each vertex $i$ there is a relation
\begin{align*}
\sum_{\alpha : t(\alpha)=i} \alpha \overline{\alpha} \hspace{5pt} - \hspace{-3pt} \sum_{\alpha : h(\alpha)=i} \overline{\alpha} \alpha.
\end{align*}
\end{defn}

We remark that if $\Gamma$ is a graph and $Q$ and $Q'$ are quivers obtained from $\Gamma$ by assigning some orientation to the arrows, then $\Pi(Q) \cong \Pi(Q')$. \\
\indent Of particular importance are the \emph{extended Dynkin graphs}, which are shown in Figure \ref{dynkindiagrams}. These consist of two infinite families, $\widetilde{\mathbb{A}}_n$ (for $n \geqslant 1$) and $\widetilde{\mathbb{D}}_n$ (for $n \geqslant 4$), and three exceptional examples, $\widetilde{\mathbb{E}}_6$, $\widetilde{\mathbb{E}}_7$, and $\widetilde{\mathbb{E}}_8$. By removing the starred vertex, we obtain a \emph{Dynkin graph}; for example, removing the starred vertex from an $\widetilde{\mathbb{A}}_n$ extended Dynkin graph yields an $\mathbb{A}_n$ Dynkin graph, and similarly for the other cases. Given an (extended) Dynkin graph, we can turn it into a quiver $Q$ by assigning some orientation to the edges, and then we can form the preprojective algebra $\Pi(Q)$ as above. \\
\indent Figure \ref{dynkindiagrams} also includes two other graphs, namely $\widetilde{\mathbb{L}}_n$ (for $n \geqslant 1$) and $\widetilde{\mathbb{DL}}_n$ (for $n \geqslant 2$). As before, we can remove the starred vertex to obtain a new graph, which is called an $\mathbb{L}_n$ graph. These graphs can be turned into quivers $Q$ by assigning some orientation to the edges. With some care, one can then define the preprojective algebra $\Pi(Q)$ of $Q$; however, Definition \ref{preprojdef} requires $Q$ to not have loops. For details on how to define the preprojective algebra in these cases, see \cite{malkin,simon}.

\begin{figure}[h]
\begin{tabular}{ p{6.2cm}  p{6.2cm} }
\begin{tikzpicture}[-,thick,scale=0.75]


\def \vwiggle {0.8415};
\node at (0,0.5+\vwiggle) {};
\node at (0,0.5-\vwiggle) {};

\def \wiggle {0.67};

\node at ({-1-\wiggle},0.5) {$\widetilde{\mathbb{A}}_n:$};
\node at ({6.5135-\wiggle},0.5) {\phantom{}};

\node[scale=0.8,circle] (1) at (0,0) {$\bullet$};
\node[scale=0.8,circle] (2) at (1,0) {$\bullet$};
\node[scale=0.8,circle] (3) at (2,0) {};
\node[scale=0.8,circle] (4) at (3,0) {};
\node[scale=0.8,circle] (5) at (4,0) {$\bullet$};
\node[scale=0.8,circle] (6) at (5,0) {$\bullet$};
\node[scale=0.8,circle] (0) at (2.5,1) {$\star$};

\draw (1) to (2);
\draw (2) to (3);
\draw (4) to (5);
\draw (5) to (6);
\draw (1.45) to (0);
\draw (0) to (6.135);

\draw[-,dash pattern={on 1pt off 2pt}] (3) to (4);

\end{tikzpicture}
&
\begin{tikzpicture}[-,thick,scale=0.75]

\def \wiggle {0.1};

\node at (5.8505-\wiggle,0) {\phantom{0}};
\node at (-1.707-\wiggle,0) {$\widetilde{\mathbb{D}}_n:$};

\node[scale=0.8,circle] (1) at (0,0) {$\bullet$};
\node[scale=0.8,circle] (2) at (1,0) {$\bullet$};
\node[scale=0.8,circle] (3) at (2,0) {};
\node[scale=0.8,circle] (4) at (3,0) {};
\node[scale=0.8,circle] (5) at (4,0) {$\bullet$};
\node[scale=0.8,circle] (6) at (5,0) {$\bullet$};

\node[scale=0.8,circle] (7) at (-0.707,-0.707) {$\star$};
\node[scale=0.8,circle] (8) at (-0.707,0.707) {$\bullet$};
\node[scale=0.8,circle] (9) at (5+0.707,-0.707) {$\bullet$};
\node[scale=0.8,circle] (10) at (5+0.707,0.707) {$\bullet$};

\draw (1) to (2);
\draw (2) to (3);
\draw (4) to (5);
\draw (5) to (6);
\draw (7) to (1);
\draw (8) to (1);
\draw (9) to (6);
\draw (10) to (6);

\draw[-,dash pattern={on 1pt off 2pt}] (3) to (4);

\end{tikzpicture}
\\
\begin{tikzpicture}[-,thick,scale=0.75]

\def \wiggle {1.21};

\node at (6.439-\wiggle,0) {\phantom{0}};

\node at (-1-\wiggle,1) {$\widetilde{\mathbb{E}}_6:$};

\node[scale=0.8,circle] (1) at (0,0) {$\star$};
\node[scale=0.8,circle] (2) at (1,0) {$\bullet$};
\node[scale=0.8,circle] (3) at (2,0) {$\bullet$};
\node[scale=0.8,circle] (4) at (3,0) {$\bullet$};
\node[scale=0.8,circle] (5) at (4,0) {$\bullet$};
\node[scale=0.8,circle] (6) at (2,1) {$\bullet$};
\node[scale=0.8,circle] (7) at (2,2) {$\bullet$};

\draw (1) to (2);
\draw (2) to (3);
\draw (3) to (4);
\draw (4) to (5);
\draw (6) to (3);
\draw (6) to (7);

\end{tikzpicture}
&
\begin{tikzpicture}[-,thick,scale=0.75]

\def \vwiggle {1.1349};
\node at (0,0.5+\vwiggle) {};
\node at (0,0.5-\vwiggle) {};

\def \wiggle {0.33};
\node at (6.5915-\wiggle,0) {\phantom{0}};

\node at (-1-\wiggle,0.5) {$\widetilde{\mathbb{E}}_7:$};

\node[scale=0.8,circle] (1) at (0,0) {$\star$};
\node[scale=0.8,circle] (2) at (1,0) {$\bullet$};
\node[scale=0.8,circle] (3) at (2,0) {$\bullet$};
\node[scale=0.8,circle] (4) at (3,0) {$\bullet$};
\node[scale=0.8,circle] (5) at (4,0) {$\bullet$};
\node[scale=0.8,circle] (6) at (5,0) {$\bullet$};
\node[scale=0.8,circle] (7) at (6,0) {$\bullet$};
\node[scale=0.8,circle] (8) at (3,1) {$\bullet$};

\draw (1) to (2);
\draw (2) to (3);
\draw (3) to (4);
\draw (4) to (5);
\draw (5) to (6);
\draw (6) to (7);
\draw (8) to (4);

\end{tikzpicture}
\\
\multicolumn{2}{c}
{
\begin{tikzpicture}[-,thick,scale=0.75]

\node at (-1,0.5) {$\widetilde{\mathbb{E}}_8:$};

\node[scale=0.8,circle] (1) at (0,0) {$\star$};
\node[scale=0.8,circle] (2) at (1,0) {$\bullet$};
\node[scale=0.8,circle] (3) at (2,0) {$\bullet$};
\node[scale=0.8,circle] (4) at (3,0) {$\bullet$};
\node[scale=0.8,circle] (5) at (4,0) {$\bullet$};
\node[scale=0.8,circle] (6) at (5,0) {$\bullet$};
\node[scale=0.8,circle] (7) at (6,0) {$\bullet$};
\node[scale=0.8,circle] (8) at (7,0) {$\bullet$};
\node[scale=0.8,circle] (9) at (5,1) {$\bullet$};

\draw (1) to (2);
\draw (2) to (3);
\draw (3) to (4);
\draw (4) to (5);
\draw (5) to (6);
\draw (6) to (7);
\draw (7) to (8);
\draw (6) to (9);

\end{tikzpicture}
}
\\[6pt]
\begin{tikzpicture}[-,thick,scale=0.75]

\def \vwiggle {0.84167};
\node at (0,\vwiggle) {};
\node at (0,-\vwiggle) {};

\node at (-1.707,0) {$\widetilde{\mathbb{L}}_n:$};

\node[scale=0.8,circle] (1) at (0,0) {$\star$};
\node[scale=0.8,circle] (2) at (1,0) {$\bullet$};
\node[scale=0.8,circle] (3) at (2,0) {};
\node[scale=0.8,circle] (4) at (3,0) {};
\node[scale=0.8,circle] (5) at (4,0) {$\bullet$};
\node[scale=0.8,circle] (6) at (5,0) {$\bullet$};

\draw (1) to (2);
\draw (2) to (3);
\draw (4) to (5);
\draw (5) to (6);

\node [circle](a) at (5,0) {$\phantom{\bullet}$};
\node [circle,minimum size=0.707cm](b) at ([{shift=(0:0.4)}]6){};
\coordinate  (c) at (intersection 2 of a and b);
\coordinate  (d) at (intersection 1 of a and b);
 \tikzAngleOfLine(b)(d){\AngleStart}
  \tikzAngleOfLine(b)(c){\AngleEnd}
 \draw%
   let \p1 = ($ (b) - (d) $), \n2 = {veclen(\x1,\y1)}
   in   
     
      (d) arc (\AngleStart-360:\AngleEnd:\n2); 

\node [circle](a) at (0,0) {$\phantom{\bullet}$};
\node [circle,minimum size=0.707cm](b) at ([{shift=(0:-0.4)}]1){};
\coordinate  (c) at (intersection 2 of a and b);
\coordinate  (d) at (intersection 1 of a and b);
 \tikzAngleOfLine(b)(d){\AngleStart}
  \tikzAngleOfLine(b)(c){\AngleEnd}
 \draw%
   let \p1 = ($ (b) - (d) $), \n2 = {veclen(\x1,\y1)}
   in   
     
      (d) arc (\AngleStart:\AngleEnd:\n2); 

\draw[-,dash pattern={on 1pt off 2pt}] (3) to (4);

\end{tikzpicture}
&
\begin{tikzpicture}[-,thick,scale=0.75]

\node at (-1.707,0) {$\widetilde{\mathbb{DL}}_n:$};

\node[scale=0.8,circle] (1) at (0,0) {$\bullet$};
\node[scale=0.8,circle] (2) at (1,0) {$\bullet$};
\node[scale=0.8,circle] (3) at (2,0) {};
\node[scale=0.8,circle] (4) at (3,0) {};
\node[scale=0.8,circle] (5) at (4,0) {$\bullet$};
\node[scale=0.8,circle] (6) at (5,0) {$\bullet$};

\node[scale=0.8,circle] (7) at (-0.707,-0.707) {$\star$};
\node[scale=0.8,circle] (8) at (-0.707,0.707) {$\bullet$};

\draw (1) to (2);
\draw (2) to (3);
\draw (4) to (5);
\draw (5) to (6);
\draw (7) to (1);
\draw (8) to (1);

\node [circle](a) at (5,0) {$\phantom{\bullet}$};
\node [circle,minimum size=0.707cm](b) at ([{shift=(0:0.4)}]6){};
\coordinate  (c) at (intersection 2 of a and b);
\coordinate  (d) at (intersection 1 of a and b);
 \tikzAngleOfLine(b)(d){\AngleStart}
  \tikzAngleOfLine(b)(c){\AngleEnd}
 \draw%
   let \p1 = ($ (b) - (d) $), \n2 = {veclen(\x1,\y1)}
   in   
     
      (d) arc (\AngleStart-360:\AngleEnd:\n2); 

\draw[-,dash pattern={on 1pt off 2pt}] (3) to (4);

\end{tikzpicture}

\end{tabular}
\caption{The extended Dynkin graphs, and two other Euclidean graphs. Each graph has $n$ black vertices, and a starred extending vertex.} \label{dynkindiagrams} 
\end{figure}
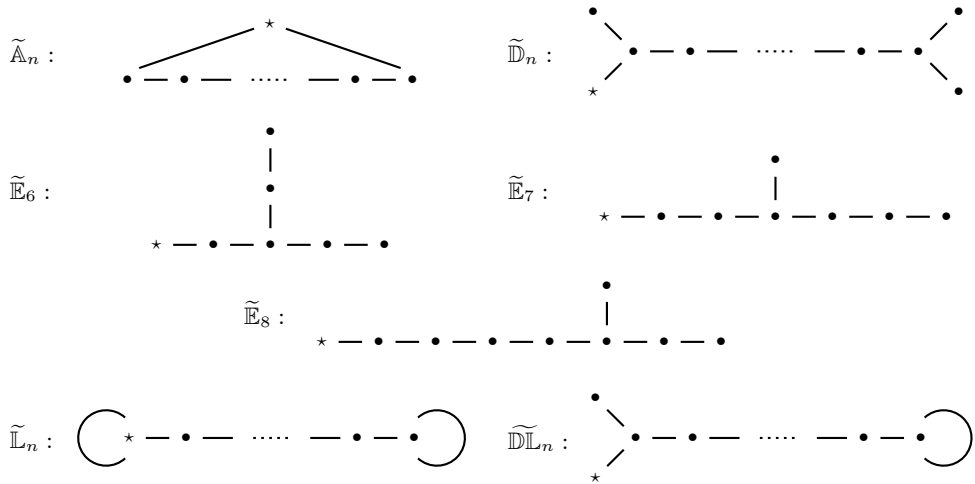

\subsection{Hopf algebras and noncommutative invariant theory}
We now discuss Hopf algebras and their actions on rings, as well as some aspects of noncommutative invariant theory. Given a Hopf algebra, we write $\Delta$ for the coproduct, $\varepsilon$ for the counit, and $S$ for the antipode. We write the coproduct in $H$ using sumless Sweedler notation so, for $h \in H$,
\begin{align*}
\Delta(h) = h_{(1)} \otimes h_{(2)}.
\end{align*}
We will chiefly be concerned with semisimple Hopf algebras, which are necessarily finite-dimensional and where the antipode satisfies $S^2 = \id_H$. In particular, this last condition implies (in fact, is equivalent to)
\begin{align*}
\varepsilon(h) 1_H = h_{(2)} S(h_{(1)}) =  S(h_{(2)}) h_{(1)}
\end{align*}
for all $h \in H$. We will require this fact at numerous points.

\begin{defn}
Let $H$ be a Hopf algebra and $A$ a $\Bbbk$-algebra. We say that \emph{$A$ is an $H$-module algebra} if $A$ is a left $H$-module which satisfies
\begin{align*}
h \cdot (ab) = (h_{(1)} \cdot a)(h_{(2)} \cdot b) \qquad \text{and} \qquad h \cdot 1_A = \varepsilon(h) 1_A
\end{align*}
for all $h \in H$ and $a,b \in A$. In this case, the \emph{invariant ring} of the action of $H$ on $A$ is
\begin{align*}
A^H = \{ a \in A \mid h \cdot a = \varepsilon(h) a \text{ for all } h \in H \}.
\end{align*}
When $A$ is an $H$-module algebra, we can form the \emph{smash product} $A \hash H$ which, as an abelian group, is $A \otimes H$, and where the multiplication is given by 
\begin{align*}
(a \hash h)(b \hash k) = a (h_{(1)} \cdot b) \hash h_{(2)} k,
\end{align*}
and extended linearly. When $H = \Bbbk G$ is a group algebra, we write $A \hash G$ instead of $A \hash \Bbbk G$. The algebra $A \hash G$ is called a \emph{skew group algebra}.
\end{defn}

\indent Using our current perspective, the classical setting can be described as follows: we are interested in actions of the Hopf algebra $H = \Bbbk G$, where $G \subseteq \Autgr(A)$, on the AS regular algebra $A = \Bbbk [x_1, \dots, x_n]$. We observe that, in this case, $H$ is finite-dimensional and semisimple. Moreover, since $G \subseteq \Autgr(A)$, the action of $H$ is degree-preserving and faithful. We will be concerned with actions of arbitrary Hopf algebras on AS regular algebras which satisfy these properties, although we replace the last condition with the following:

\begin{defn}
Let $V$ be a left $H$-module. We say that the action of $H$ on $V$ is \emph{inner-faithful} if $IV \neq 0$ for every nonzero Hopf ideal $I$ of $H$.
\end{defn}

This condition ensures that the action of $H$ does not factor through the action of one of its proper quotients. If $H = \Bbbk G$, then the action of $H$ on a module $V$ is inner-faithful if and only if it is faithful.

\begin{hypothesis} \label{mainhypothesis}
Let $A$ be an $m$-Koszul AS regular algebra of Gorenstein parameter $\ell$, of global dimension and GK dimension $d$, and which is generated in degree $1$. In particular, $A \cong \scrD(\sfw,\ell-m)$ for some twisted superpotential $\sfw$. Let $H$ be a semisimple (hence finite dimensional) Hopf algebra acting on $A$. Throughout, we assume the following hypotheses on the pair $(A,H)$:
\begin{enumerate}[{\normalfont (1)},topsep=1pt,itemsep=1pt,leftmargin=35pt]
\item $A$ is a left $H$-module algebra;
\item The action of $H$ on $A$ is degree-preserving, in the sense that each $A_i$ is an $H$-module; and
\item The action of $H$ on $A$ is inner-faithful.
\end{enumerate}
\end{hypothesis}

\indent In classical invariant theory, many properties of the invariant ring $\Bbbk[x_1, \dots, x_n]^G$ are controlled by the determinant of the elements in $G$. For example, a result of Watanabe says that, if every element of $G$ has determinant 1, then this invariant ring is Gorenstein. In \cite{jorgensen}, the authors defined the \emph{homological determinant} of the action of a finite group on an AS Gorenstein algebra $A$, and showed that it had a number of good properties. This construction was later generalised to actions of Hopf algebras on AS Gorenstein algebras in \cite{kkzGor}. We give a definition in the particular case where $A$ is AS regular.

\begin{defn}[{cf. \cite[Definition 3.3, Lemma 5.10]{kkzGor}}]
Suppose that $(A,H)$ satisfy Hypothesis \ref{mainhypothesis}. If $P^\bullet$ is a graded projective resolution of $\Bbbk$ viewed as a left $A \hash H$-module, then $P^\bullet$ may also be viewed as an $H$-equivariant graded projective resolution of $\Bbbk$ as a left $A$-module. Using this, $\Ext_A^d(\Bbbk,\Bbbk)$ becomes a left $H$-module. Since $A$ is AS regular, $\Ext_A^d(\Bbbk,\Bbbk)$ is one-dimensional, so there is an algebra homomorphism $\eta : H \to \Bbbk$ such that $h \cdot x = \eta(h) x$ for all $0 \neq x \in \Ext_A^d(\Bbbk,\Bbbk)$. The composite map $\eta \circ S : H \to \Bbbk$ is called the \emph{homological determinant} of the action of $H$ on $A$.
\end{defn}

We write $\hdet_A : H \to \Bbbk$ for the homological determinant map. We say that the homological determinant is \emph{trivial} if $\hdet_A = \varepsilon$, the counit of $H$. If $A$ is a $d$-dimensional polynomial ring and $H = \Bbbk G$ for a finite subgroup $G \leqslant \Autgr(A) = \GL(d,\Bbbk)$, then $\hdet_A(g) = \det(g)$ for all $g \in G$. By \cite[Theorem 3.3]{jorgensen} (for the group case) and \cite[Theorem 3.6]{kkzGor} (for the Hopf case), if the homological determinant of the $H$-action on an AS regular algebra is trivial, then $A^H$ is AS Gorenstein. 

\indent A map that is of particular importance in noncommutative invariant theory is the \emph{Auslander map}, which we now define:

\begin{defn} \label{ausmap}
Suppose that $A$ is a left $H$-module algebra. Then there is a natural map of graded $\Bbbk$-algebras
\begin{align*}
\gamma : A \hash H \to \End(A_{A^H}), \quad \gamma(a \hash h)(b) = a (h \cdot b),
\end{align*}
called the \emph{Auslander map}.
\end{defn}

In some situations the Auslander map is an isomorphism, and there is considerable interest in determining when this is the case. In the classical setting of a finite group acting on a polynomial ring, a complete answer is given in \cite{purity}, where it is shown that this map is an isomorphism if and only if $G$ contains no non-trivial reflections. In the noncommutative setting, it is known that if $A$ is AS regular of dimension 2 and the $H$-action on $A$ has trivial homological determinant, then the Auslander map is an isomorphism \cite{ckwzi}. For other recent progress, see \cite{ncauslanderthm, won, crawford19}, for example. \\
\indent The main tool that has been used to show that the Auslander map is an isomorphism is the following result:

\begin{thm}[{\cite[Theorem 0.3]{bhz}}] \label{introbhz}
Suppose that the pair $(A,H)$ satisfies Hypothesis \ref{mainhypothesis}, and that $A$ is GK-Cohen-Macaulay in the sense of \cite[Definition 1.4]{bhz}. Let $t \in H$ be a nonzero \emph{integral} in $H$, i.e.\ $t$ satisfies $h t = th = \varepsilon(h) t$ for all $h \in H$. Then the Auslander map is an isomorphism for the pair $(A,H)$ if and only if 
\begin{align*}
\GKdim \frac{A \hash H}{\langle 1 \hash t \rangle } \leqslant \GKdim A - 2.
\end{align*}
\end{thm}

The above criterion is computationally useful, since it is usually much easier to determine the GK dimension of $(A \hash H)/\langle 1 \hash t \rangle$ rather than determine whether the Auslander map is an isomorphism. We remark that all known AS regular algebras are GK-Cohen-Macaulay, and it is conjectured that every AS regular algebra has this property. \\
\indent When the Auslander map is an isomorphism, we are able to deduce strong representation-theoretic results relating the algebras $H$, $A \hash H$, and $A^H$. In the following, a graded module $M$ is said to be \emph{initial} if $M_{<0} = 0$ and $M$ is generated in degree $0$, and the definition of a \emph{maximal Cohen-Macaulay module} can be found in \cite[Definition 3.5]{ckwzii}.

\begin{thm}[{\cite[Theorem A, Theorem C]{ckwzii}}] \label{mcmmodules}
Suppose that the pair $(A,H)$ satisfies Hypothesis \ref{mainhypothesis}. Suppose also that the Auslander map $\gamma : A \hash H \to \End_{A^H}(A)$ is an isomorphism. Then there exist bijections between isomorphism classes of:
\begin{enumerate}[{\normalfont (1)},leftmargin=30pt,topsep=0pt,itemsep=0pt]
\item irreducible left $H$-modules;
\item indecomposable direct summands of $A$, viewed as a left $A^H$-module; and
\item indecomposable finitely generated, projective, initial left $A \hash H$ modules.
\end{enumerate}
If $A$ has GK dimension $2$, then the above are also in bijection with:
\begin{enumerate}[{\normalfont (1)},leftmargin=30pt,topsep=0pt,itemsep=0pt]
\item[\normalfont{(4)}] indecomposable maximal Cohen-Macaulay left $A^H$-modules, up to a degree shift.
\end{enumerate}
The correspondence $(1) \to (3)$ is given by $V \mapsto A \otimes V$, and the correspondence $(3) \to (4)$ is given by $P \mapsto P^H$.
\end{thm}

\indent If $H$ is a semisimple Hopf algebra, then the category of left $H$-modules is a \emph{fusion category}: a $\Bbbk$-linear monoidal category which is rigid (in the sense that we can define duals of objects) and semisimple with finitely many isoclasses of irreducible objects. In particular, the tensor product of two modules decomposes uniquely as a direct sum of irreducible modules, up to isomorphism, which allows us to make the following definition:

\begin{defn}
Let $H$ be a semisimple Hopf algebra. Let $\{ V_0, V_1, \dots, V_n \}$ be a complete list of representatives for isoclasses of irreducible $H$-modules, where $V_0$ is the trivial representation (i.e.\ $V_0 = \Bbbk v$ where $h \cdot v = \varepsilon(h) v$ for all $h \in H$). Fix a representation $V$ of $H$. The \emph{(left) McKay quiver} associated to $V$ is the quiver with vertex set $\{0, 1, \dots, n\}$, and $m_{ij}$ arrows from vertex $i$ to vertex $j$, where 
\begin{align*}
V \otimes V_j \cong \bigoplus_{i=0}^n V_i^{m_{ij}}.
\end{align*}
Note that $m_{ij}$ satisfies
\begin{align*}
m_{ij} = \dim_\Bbbk \Hom_H(V_i, V \otimes V_j).
\end{align*}
One can define the \emph{right McKay quiver} by replacing $V \otimes V_j$ with $V_j \otimes V$. We will only work with left McKay quivers, and will omit the word ``left''. \\
\indent If $H$ acts on an AS regular algebra $A$, then $V \coloneqq A_1$ is a representation of $H$, and we refer to the McKay quiver of $V$ as the \emph{McKay quiver of the action of $H$ on $A$} or \emph{of the pair $(A,H)$}.
\end{defn}

In the classical setting of group actions, one does not need to distinguish between the ``left'' and ``right'' McKay quivers since the category of representations of a group is symmetric, i.e.\ $V \otimes W \cong W \otimes V$ for all representations $V$ and $W$. In general, the category of representations of a semisimple Hopf algebra is not symmetric; for example, when $H$ is the dual of a finite nonabelian group. \\
\indent The following result provides another characterisation of the inner-faithful condition in terms of the McKay quiver:

\begin{lem}
Suppose that $H$ is a semisimple Hopf algebra, and let $V$ be a left $H$-module. Then the following are equivalent:
\begin{enumerate}[{\normalfont (1)},topsep=0pt,itemsep=0pt,leftmargin=*]
\item $V$ is an inner-faithful $H$-module;
\item Each irreducible module $V_i$ appears as a direct summand of $V^{\otimes n_i}$ for some $n_i \geqslant 1$; and
\item The McKay quiver $Q$ of $V$ is strongly connected, in the sense that there is a path from one vertex to any other.
\end{enumerate}
\end{lem}
\begin{proof}
\indent $(1) \Leftrightarrow (2)$: This is shown in \cite[Theorem 1.4]{three}. \\
\indent Before establishing the equivalence of (2) and (3), we note that there is a path from vertex $i$ to vertex $j$ in $Q$ of length $k$ if and only if $V_i$ is a direct summand of $V^{\otimes k} \otimes V_j$. \\
\indent $(2) \Rightarrow (3)$: Fix vertices $i$ and $j$ in $Q$, and consider the irreducible $H$-module $V_i \otimes V_j^*$. By assumption, there exists $k \geqslant 1$ such that $V^{\otimes k} \cong (V_i \otimes V_j^*) \oplus U$ for some $H$-module $U$. Tensoring with $V_j$, and noting that $V_j^* \otimes V_j \cong V_0 \oplus W$ for some $H$-module $W$, we obtain
\begin{align*}
V^{\otimes k} \otimes V_j \cong (V_i \otimes V_j^* \otimes V_j) \oplus (U \otimes V_j) \cong (V_i \otimes (V_0 \oplus W)) \oplus (U \otimes V_j) = V_i \oplus (V_i \otimes W) \oplus (U \otimes V_j).
\end{align*}
In particular, $V_i$ is a summand of $V^{\otimes k} \otimes V_j$, so there is a path of length $k$ from vertex $i$ to vertex $j$. \\
\indent $(3) \Rightarrow (2)$: Suppose that $Q$ is strongly connected, and fix an irreducible module $V_i$. By assumption, there is a path from vertex $i$ to vertex $0$ in $Q$, of length $k$, say. In particular, $V^{\otimes k} \otimes V_0 \cong V_i \oplus U$ for some $H$-module $U$. However, $V^{\otimes k} \otimes V_0 \cong V^{\otimes k}$, and so $V_i$ is a summand of $V^{\otimes k}$.
\end{proof}

Since our standing assumption is that the action of a Hopf algebra on an AS regular algebra is inner-faithful, every McKay quiver that we will be concerned with will be strongly connected. \\
\indent Suppose now that $A = \Bbbk[u,v]$ and $H = \Bbbk G$, where $G$ is a finite subgroup of $\SL(2,\Bbbk)$. It is well-known that, up to conjugation, $G$ belongs to one of two infinite families or is one of three exceptional examples. A very brief overview of the classification is as follows:
\begin{center}
\begin{tabular}{c|c|c}
Group $G$  & $|G|$ & McKay quiver  \\ \hline && \\[-8pt]
Cyclic, parametrised by $n \geqslant 1$ & $n+1$ & $\widetilde{\mathbb{A}}_n$ \\
Binary dihedral, parametrised by $n \geqslant 4$ & $4(n-2)$ & $\widetilde{\mathbb{D}}_n$ \\
Binary tetrahedral & $24$ & $\widetilde{\mathbb{E}}_6$ \\
Binary octahedral & $48$ & $\widetilde{\mathbb{E}}_7$ \\
Binary icosahedral  & $120$ & $\widetilde{\mathbb{E}}_8$ 
\end{tabular}
\end{center}
More precisely, the McKay quiver is the double of the listed extended Dynkin graph. The invariant rings $\Bbbk[u,v]^G$ are called \emph{Kleinian singularities}. If $G$ is cyclic, for example, then we say that $\Bbbk[u,v]^G$ is a Type $\mathbb{A}$ (Kleinian) singularity, and similarly for the other cases.

\section{The homological determinant of a Hopf action on an $m$-Koszul AS regular algebra}
In this section, we give a simple formula for the homological determinant of the action of a Hopf algebra on an $m$-Koszul AS regular algebra. The proof strategy mimics that of \cite[Theorem 3.3]{mkoszul}. \\
\indent Throughout this section, let $A$ be an $m$-Koszul AS regular algebra of global dimension $d$ and with Gorenstein parameter $\ell$. Therefore, writing $V = A_1$, there exists a twisted superpotential $\sfw \in V^{\otimes \ell}$ such that, if we set $\mathcal{R} = \partial^{\ell-m}(\sfw)$, then
\begin{align*}
A \cong \scrD(\sfw, \ell-m) = \frac{T_\Bbbk(V)}{\langle \mathcal{R} \rangle}. 
\end{align*}

\begin{lem}
Suppose that $H$ is a semisimple Hopf algebra acting homogeneously on $A$. Then the natural action of $H$ on $T_\Bbbk(V^*)$ descends to an action on $A^!$.
\end{lem}
\begin{proof}
It suffices to show that $\mathcal{R}^\perp$ is closed under the $H$-action. Let $f \in \mathcal{R}^\perp$, say $f = \sum_i \alpha_i f_{i_1} \otimes \dots \otimes f_{i_m}$. This means that $\mathcal{R}$ vanishes under $f$, where we remind that reader that we identify $(V^*)^{\otimes m}$ with $(V^{\otimes m})^*$ as in (\ref{tensordual}). Then, if $r = \sum_j \beta_j v_{j_1} \otimes \dots \otimes v_{j_m}$ is any relation in $\mathcal{R}$, we have
\begin{align*}
(h \cdot f)(r) &= \sum_{i,j} \alpha_i \beta_j \big((h_{(1)} f_{i_1}) \otimes \dots \otimes (h_{(m)} f_{i_m})\big) \big( v_{j_1} \otimes \dots \otimes v_{j_m} \big) \\ 
&= \sum_{i,j} \alpha_i \beta_j (h_{(m)} f_{i_m})(v_{j_1}) \dots (h_{(1)} f_{i_1})(v_{j_m}) \\
&= \sum_{i,j} \alpha_i \beta_j f_{i_m}\big(S(h_{(m)}) v_{j_1}\big) \dots f_{i_1} \big(S(h_{(1)}) v_{j_m}\big) \\
&= \sum_{i,j} \alpha_i \beta_j \big(f_{i_1} \otimes \dots \otimes f_{i_m}\big) \big( (S(h_{(m)}) v_{j_1}) \otimes \dots \otimes (S(h_{(1)}) v_{j_m}) \big) \\
&= \sum_{i,j} \alpha_i \beta_j \big(f_{i_1} \otimes \dots \otimes f_{i_m}\big) \big( S(h) \cdot (v_{j_1} \otimes \dots \otimes v_{j_m}) \big) \\
&= f(S(h) \cdot r) \\
&= 0,
\end{align*}
where the last equality follows since $R$ is closed under the action from $H$.
\end{proof}

By definition, the Koszul complex $P^\bullet$ is a projective resolution of the left trivial module $\Bbbk = A/A_{\geqslant 1}$. It is easy to see that each map in this resolution is $H$-equivariant, and therefore $\Ext_A^i(\Bbbk,\Bbbk)$ is an $H$-module. In fact, the differential in the complex $\Hom_A(P^\bullet, \Bbbk)$ is zero \cite[p.\ 77]{berger2}, so $\Ext_A^i(\Bbbk,\Bbbk) = \Hom_A(P_{\rho(i)},\Bbbk)$. We then have a chain of equalities and $H$-module isomorphisms:
\begin{align*}
\Ext_A^i(\Bbbk,\Bbbk) &= \Hom_A(P_{\rho(i)},\Bbbk) = \Hom_A(A \otimes (A^!_{\rho(i)})^*, \Bbbk) \\ 
&\cong \Hom_\Bbbk((A^!_{\rho(i)})^*, \Hom_A(A,\Bbbk)) \cong \Hom_\Bbbk((A^!_{\rho(i)})^*, \Bbbk) \\
&\cong A^!_{\rho(i)}.
\end{align*}
In particular, $\Ext_A^d(\Bbbk,\Bbbk) \cong A^!_\ell$. Now, by \cite[Lemma 2.7]{mkoszul}, if we define
\begin{align*}
\mathsf{W} = \bigcap_{s+m+t = \ell} V^{\otimes s} \otimes R \otimes V^{\otimes t},
\end{align*}
then there is a natural identification $A_\ell^! \cong \mathsf{W}^*$. Tracing through the isomorphism, we find that the induced $H$-action on $\mathsf{W}^*$ is simply the usual action of $H$ on $\mathsf{W}^*$ viewed as a subspace (in fact, an $H$-submodule) of $(V^{\otimes \ell})^* \cong (V^*)^{\otimes \ell}$. Moreover $\sfw \in \mathsf{W}$, and this space is one-dimensional, i.e.\ $\mathsf{W} = \Bbbk \sfw$ \cite[Lemma 2.12]{mkoszul}. \\ 
\indent We now have all of the ingredients required to prove the main result of this section:

\begin{thm} \label{hdetthm}
Suppose that the pair $(A,H)$ satisfies Hypothesis \ref{mainhypothesis}. Then $\mathsf{W} = \Bbbk \sfw$ is an $H$-submodule of $V^{\otimes \ell}$, and the homological determinant of the action of $H$ on $A$ satisfies
\begin{align*}
h \cdot \sfw = \hdet_A(h) \sfw.
\end{align*}
\end{thm}
\begin{proof}
Since $\mathsf{W}^*$ is a one-dimensional representation of $H$, so too is its dual $\mathsf{W}^{**} \cong \mathsf{W}$, establishing the first claim. \\
\indent Now fix $h \in H$. Since $\mathsf{W}$ is one-dimensional representation of $H$, we have $h \cdot \sfw = \lambda \sfw$ for some $\lambda \in \Bbbk$. Then, if $f \in W^*$, we have
\begin{align*}
(S(h) \cdot f)(w) = f(S^2(h) \cdot w) = f(h \cdot w) = \lambda f(w),
\end{align*}
since $S^2 = \id_H$, and so $S(h) \cdot f = \lambda f$. Let $\psi : \Ext_A^d(\Bbbk,\Bbbk) \to \mathsf{W}^*$ be the $H$-module isomorphism established before the statement of the theorem. By definition, the $H$-action on $\Ext_A^d(\Bbbk,\Bbbk)$ satisfies $h \cdot x = \eta(h) x$ for all $0 \neq x \in \Ext_A^d(\Bbbk,\Bbbk)$, where $\eta \circ S = \hdet_A$. Now if $0 \neq f \in \mathsf{W}^*$, then there exists $0 \neq x \in \Ext_A^d(\Bbbk,\Bbbk)$ with $\psi(x) = f$, and we have
\begin{align*}
\lambda f = S(h) \cdot f = S(h) \cdot \psi(x) = \psi(S(h) \cdot x) = \psi\big((\eta \circ S\big)(h) \cdot x) = (\eta \circ S)(h) \psi(x) = \hdet_A(h) f.
\end{align*}
Therefore $\lambda = \hdet_A(h)$, i.e.\ $h \cdot \sfw = \hdet_A(h) \sfw$.
\end{proof}

If $H = \Bbbk G$, then this result recovers \cite[Theorem 3.3]{mkoszul}. In the special case where $H = (\Bbbk G)^*$ is the dual of a group algebra (equivalently, $A$ is $G$-graded) it is particularly easy to calculate the homological determinant.

\begin{cor} \label{hdetdual}
Let $A \cong \scrD(\sfw,\ell-m)$ be as above, let $G$ be a finite group, and suppose that the Hopf algebra $H = (\Bbbk G)^*$ acts homogeneously on $A$; equivalently, $A$ is $G$-graded. Let $\{ f_g \mid g \in G \}$ be the basis for $H$ which is dual to the standard basis of $\Bbbk G$. Then $\sfw$ is $G$-homogeneous, and the homological determinant of the action of $H$ on $A$ satisfies
\begin{align*}
\hdet_A(f_g) = \left\{ \begin{array}{ll}
1 & \text{if } \deg_G(\sfw) = g, \\
0 & \text{otherwise}.
\end{array} \right.
\end{align*}
\end{cor}
\begin{proof}
Since $H \cong \Bbbk^{|G|}$ as an algebra, and the basis consists of pairwise orthogonal idempotents, the one-dimensional representations of $H$ are of the form
\begin{align*}
\chi_g : H \to \Bbbk, \quad \chi_g(f_h) = \delta_{gh}.
\end{align*}
Since $\mathsf{W} = \Bbbk \sfw$ is a one-dimensional $H$-module, there exists $g \in G$ with $f_g \cdot \sfw = \sfw$, and $f_h \cdot \sfw = 0$ for $h \neq g$. By the definition of the coproduct in $(\Bbbk G)^*$, $f_g \cdot \sfw$ is simply the $g$-component of $\sfw$ under the $G$-grading, but also $f_g \cdot \sfw = \sfw$, and so $\sfw$ is $G$-homogeneous of $G$-degree $g$. It now follows that $\hdet_A$ has the claimed form by Theorem \ref{hdetthm}.
\end{proof}

\section{Hopf smash products are derivation-quotient algebras} \label{derquotsec}
In this section, we show that if the pair $(A,H)$ satisfies Hypothesis \ref{mainhypothesis} (so, in particular, $A$ is a derivation-quotient algebra), then the smash product $A \hash H$ is also a derivation-quotient algebra, and we give a precise description of the corresponding superpotential. A version of this result for actions of finite groups on polynomial rings was established in \cite[Theorem 3.2]{bsw}. This was later generalised to actions of finite groups on AS regular algebras in \cite[Theorem 4.16]{wuskew}, with the additional restriction that the action has trivial homological determinant. Our result strengthens this by removing the hypothesis on the homological determinant, as well as by replacing the finite group by a semisimple Hopf algebra.  \\
\indent The main difficulty in generalising \cite[Theorem 4.16]{wuskew} to the Hopf algebra setting is the lack of control one has over the form of the coproduct. When $H = \Bbbk G$ is a group algebra, it has a distinguished basis given by the elements of $G$, and the coproduct has a simple form. As a result, in \cite{wuskew} the authors were able to perform explicit calculations to deduce their result. On the other hand, an arbitrary semisimple Hopf algebra has no distinguished basis in which the coproduct has a nice form. Our approach is to fix a basis of $H$ by appealing to its Artin--Wedderburn decomposition, which provides some control over the form of the coproduct. \\
\indent Since $A$ is an $m$-Koszul AS regular algebra, we assume that $A = T_\Bbbk(V)/\langle \partial^{\ell-m} \sfw \rangle = \scrD(\sfw,\ell-m)$ where $\ell$ is the Gorenstein parameter of $A$, $V = A_1$ is a finite-dimensional vector space $V$ with basis $\{ v_1, \dots, v_r \}$, and
\begin{align*}
\sfw \coloneqq \sum_q \alpha_q v_{q_1} \otimes \dots \otimes v_{q_\ell}
\end{align*}
is a $\sigma$-twisted superpotential; that is, there exists $\sigma \in \Aut(V)$ such that
\begin{align*}
\llbracket\phi^{\sigma^*},\sfw \rrbracket = \llbracket \sfw,\phi \rrbracket 
\end{align*}
for all $\phi \in V^*$. In particular, $V$ is a left $H$-module, and we can turn $V \otimes H$ into an $H$-bimodule where the left $H$-action is diagonal (via the coproduct) and the right $H$-action acts only on the right tensorand by right multiplication. Then the map
\begin{gather*}
\Psi : T_\Bbbk(V) \hash H \to T_H(V \otimes H), \quad \Psi((u_1 \otimes \dots \otimes u_n) \hash h) = (u_1 \otimes 1) \otimes_H \dots \otimes_H (u_{n-1} \otimes 1) \otimes_H (u_n \otimes h) 
\end{gather*}
is an algebra isomorphism. Here, we have written $\otimes_H$ to help emphasise over which rings some of these tensors are formed. We will not use this notation going forwards, and will instead try to distinguish various tensors by bracketing terms in the same way as above. Our claim is that $\Psi(\sfw \hash 1)$ is a twisted superpotential, for an appropriate twist $\nu \in \Aut(V \otimes H)$, and that there is an isomorphism 
\begin{align*}
A \hash H \cong \scrD(\Psi(\sfw \hash 1),\ell-m),
\end{align*}
where $\ell$ and $m$ are as above.\\
\indent Suppose that $H$ has finite-dimensional irreducible modules $V_0, \dots, V_n$, where $V_0$ is the trivial module (that is, it corresponds to the counit $\varepsilon$), and fix an isomorphism with the Artin--Wedderburn decomposition of $H$,
\begin{align}
H \cong \bigoplus_{k=0}^n \text{Mat}_{\dim V_k}(\Bbbk). \label{awdecomp}
\end{align}
By transferring the Hopf algebra structure of $H$ along this isomorphism, the right hand side becomes a Hopf algebra; by an abuse of notation, we will think of $H$ as being equal to this Hopf algebra. We write $e_{ij}^{(k)}$ for the $(i,j)$th matrix unit in the $k$th component of $H$. \\
\indent To show that $\Psi(\sfw \hash 1)$ is a twisted superpotential, we need to show that it satisfies
\begin{align*}
h \cdot \Psi(\sfw \hash 1) = \Psi(\sfw \hash 1) \cdot \nu(h)
\end{align*}
for all $h \in H$ and for a suitable twist $\nu$, and that
\begin{align*}
\llbracket\phi^{\nu^*}, \Psi(\sfw \hash 1) \rrbracket = \llbracket \Psi(\sfw \hash 1), \phi \rrbracket
\end{align*}
for all $\phi \in (V \otimes H)^*$ (see Definition \ref{superpotentialdefn2}). For the latter of these, we will need to consider the maps $R \phi^{\nu^*} \in \Hom((V \otimes H)_H,H_H)$ and $L \phi \in \Hom({}_H (V \otimes H),{}_H H)$ which satisfy the identity (\ref{tracefunction}). To define the natural isomorphisms $R$ and $L$, we need to make a choice of a nondegenerate trace function $\Tr$ for $H$. Since $H$ is a semisimple Hopf algebra, it has a one-dimensional space of (left and right) integrals for $H$, i.e.\ maps $T \in H^*$ with the property 
\begin{align*}
\phi * T = \phi(1_H) T = T * \phi
\end{align*}
for all $\phi \in H^*$, where $*$ is the convolution product in $H^*$. By \cite[2.1.3 Theorem]{montgomery2}, provided that $T \neq 0$, the associated bilinear form $h \otimes k \mapsto T(hk)$ is nondegenerate. 
Therefore we can, and will, take our trace function to be the unique integral $\Tr$ for $H^*$ with the property that $\Tr(1_H) = 1$. With this choice, we have the following result:

\begin{lem} \label{tracelem}
The trace $\Tr$ defined above satisfies 
\begin{align*}
\Tr(e_{ij}^{(k)}) = \frac{\dim V_k}{\dim H} \delta_{ij}.
\end{align*}
\end{lem}
\begin{proof}
This follows from \cite[Proposition 2.13 (2)]{montgomeryrep}.
\end{proof}

\indent If $f: H \to \Bbbk$ is any algebra homomorphism, then we can define the \emph{left (respectively, right) winding automorphism of $H$ associated to $f$}, denoted ${}_f \Xi$, (respectively, $\Xi_f$) by
\begin{align*}
{}_f \Xi : H \to H, \quad {}_f \Xi(h) = f(h_{(1)}) h_{(2)}, \\
\Xi_f : H \to H, \quad \Xi_f(h) = h_{(1)} f(h_{(2)}).
\end{align*}
It is straightforward to check that ${}_{f \circ S} \Xi$ is inverse to ${}_f \Xi$ and $\Xi_{f \circ S}$ is inverse $\Xi_f$, where $S$ is the antipode of $H$, so both of these maps are algebra automorphisms of $H$. We record the following key lemma.

\begin{lem} \label{integrallemma}
Let $T$ be a (left and right) integral for $H$. Then, for all $x,y \in H$:
\begin{enumerate}[{\normalfont (1)},leftmargin=*,topsep=0pt,itemsep=0pt]
\item $x_{(1)} T(x_{(2)}) = T(x) 1_H = T(x_{(1)}) x_{(2)}$;
\item $x_{(1)} T(x_{(2)} y) = T(x y_{(2)}) S(y_{(1)})$;
\item If $f: H \to \Bbbk$ is any algebra homomorphism, then $T({}_f \Xi(x)) = T(x) = T(\Xi_f (x))$.
\end{enumerate}
\begin{proof}
\begin{enumerate}[{\normalfont (1)},wide=0pt,topsep=0pt,itemsep=0pt]
\item This is \cite[Remark 5.1.2]{dascalescu}.
\item We recall the standard notation $\rightharpoonup$ for the left action of $H$ on $H^*$ given by $(h \rightharpoonup \phi)(k) \coloneqq \phi(kh)$. To establish the result, it suffices to show that both sides of the claimed equality are equal after applying an arbitrary $\phi \in H^*$:
\begin{align*}
\phi\big(x_{(1)} T(x_{(2)} y)\big) &= \phi(x_{(1)}) T(x_{(2)} y) = \phi(x_{(1)}) T\big(x_{(2)} \varepsilon(y_{(1)}) y_{(2)}\big) = \phi\big(x_{(1)} \varepsilon(y_{(1)})\big) T(x_{(2)} y_{(2)}) \\ 
&= \phi\big(x_{(1)} y_{(2)} S(y_{(1)})\big) T(x_{(2)} y_{(3)})
= (S(y_{(1)}) \rightharpoonup \phi)(x_{(1)} y_{(2)}) T(x_{(2)} y_{(3)}) \\
&= ((S(y_{(1)}) \rightharpoonup \phi) * T)(x y_{(2)}) = (S(y_{(1)}) \rightharpoonup \phi)(1) T(x y_{(2)}) = \phi(S(y_{(1)})) T(x y_{(2)}) \\
&= \phi\big( T(x y_{(2)}) S(y_{(1)}) \big),
\end{align*}
where the equality when moving from the first line to the second line requires the semisimplicity of $H$.
\item Direct calculation gives:
\begin{align*}
T({}_f \Xi(x)) = T\big( f(x_{(1)}) x_{(2)} \big) = f(x_{(1)}) T(x_{(2)}) = f\big( x_{(1)} T(x_{(2)}) \big) = f(T(x) 1_H) = T(x) f(1_H) = T(x),
\end{align*}
using part (1). The other equality is similar. \qedhere
\end{enumerate}
\end{proof}
\end{lem}

\indent To show that $\Psi(\sfw \hash 1)$ is a twisted superpotential, we first need to show that it is a twisted weak potential for some twist $\nu$. Henceforth, let $\Xi$ denote the the left winding automorphism of $H$ associated to the algebra morphism $\hdet : H \to \Bbbk$, i.e.\
\begin{align*}
\Xi(h) = \hdet(h_{(1)}) h_{(2)}.
\end{align*}
Let $\nu$ be the graded $\Bbbk$-algebra automorphism of $T_H(V \otimes H)$ defined by 
\begin{align*}
\restr{\nu}{H} = \Xi, \qquad \nu(v \otimes h) = \sigma(v) \otimes \Xi(h),
\end{align*}
where we recall that $\sfw$ is a $\sigma$-twisted superpotential. By Lemma \ref{integrallemma} (3), the trace $\Tr$ is invariant under $\nu$, as required.

\begin{lem} \label{twistedlem}
With the above setup, $\Psi(\sfw \hash 1)$ is a twisted weak potential, where the twist is given by $\nu$.
\end{lem}
\begin{proof}
We need to show that 
\begin{align*}
h \Psi(\sfw \hash 1) = \Psi(\sfw \hash 1) \nu(h)
\end{align*}
for all $h \in H$. As before, write $\sfw = \sum_q \alpha_{q} v_{q_1} \otimes \dots \otimes v_{q_\ell}$. By Theorem \ref{hdetthm}, we have
\begin{align*}
\hdet(h) \sfw = h \cdot \sfw = \sum_q \alpha_{q} (h_{(1)} \cdot v_{i_1}) \otimes \dots \otimes (h_{(\ell)} \cdot v_{i_\ell}).
\end{align*}
Therefore,
\begin{align*}
h \Psi(\sfw \hash 1) &= h \sum_q \alpha_{q} (v_{q_1} \otimes 1) \otimes (v_{q_2} \otimes 1) \otimes \dots \otimes (v_{q_\ell} \otimes 1) \\
&= \sum_q \alpha_{q} \big(h \cdot (v_{q_1} \otimes 1)\big) \otimes (v_{q_2} \otimes 1) \otimes \dots \otimes (v_{q_\ell} \otimes 1) \\
&= \sum_q \alpha_{q} (h_{(1)} \cdot v_{q_1} \otimes h_{(2)}) \otimes (v_{q_2} \otimes 1) \otimes \dots \otimes (v_{q_\ell} \otimes 1) \\
&= \sum_q \alpha_{q} (h_{(1)} \cdot v_{q_1} \otimes 1) \cdot h_{(2)} \otimes (v_{q_2} \otimes 1) \otimes \dots \otimes (v_{q_\ell} \otimes 1) \\
&= \sum_q \alpha_{q} (h_{(1)} \cdot v_{q_1} \otimes 1) \otimes h_{(2)} \cdot (v_{q_2} \otimes 1) \otimes \dots \otimes (v_{q_\ell} \otimes 1) \\
& \hspace{20pt} \vdots \\
&= \sum_q \alpha_{q} (h_{(1)} \cdot v_{q_1} \otimes 1) \otimes (h_{(2)} \cdot v_{q_2} \otimes 1) \otimes \dots \otimes (h_{(\ell)} v_{q_\ell} \otimes 1) \cdot h_{(\ell+1)} \\
&= \Psi\Big( \sum_q \alpha_q (h_{(1)} \cdot v_{q_1}) \otimes \dots \otimes (h_{(\ell)} \cdot v_{q_\ell}) \hash 1 \Big) h_{(\ell+1)} \\
&= \Psi\big((h_{(1)} \cdot \sfw ) \hash 1\big) h_{(2)} \\
&= \Psi(\sfw \hash 1) \hdet(h_{(1)}) h_{(2)} \\
&= \Psi(\sfw \hash 1) \nu(h). \qedhere
\end{align*}
\end{proof}

The next step is to show that $\Psi(\sfw \hash 1)$ is a twisted superpotential; that is,
\begin{align*}
\llbracket \phi^{\nu^*}, \Psi(\sfw \hash 1) \rrbracket = \llbracket \Psi(\sfw \hash 1), \phi \rrbracket
\end{align*}
for all $\phi \in (V \otimes H)^*$. It suffices to show that this is true for all elements in a basis of $(V \otimes H)^*$. Recalling that $V$ has basis $\{ v_1, \dots, v_r\}$ and that $H$, by (\ref{awdecomp}), has a basis given by
\begin{align*}
\{ e_{ij}^{(k)} \mid 0 \leqslant k \leqslant n, \hspace{2pt} 1 \leqslant i,j \leqslant \dim V_k \},
\end{align*}
we obtain a basis $\{ v_p \otimes e_{ij}^{(k)} \}$ of $V \otimes H$. Taking $\{ \phi_{v_p \otimes e_{ij}^{(k)}} \}$ to be the corresponding dual basis of $(V \otimes H)^*$, it therefore suffices to show 
\begin{align*}
\Big\llbracket \phi_{v_p \otimes e_{ij}^{(k)}}^{\nu^*}, \Psi(\sfw \hash 1) \Big\rrbracket = \Big \llbracket \Psi(\sfw \hash 1), \phi_{v_p \otimes e_{ij}^{(k)}} \Big \rrbracket,
\end{align*}
for all $p,i,j,k$. \\
\indent As a first step, we need to determine formulae for the maps $R \phi_{v_p \otimes e_{ij}^{(k)}}^{\nu^*}$ and $L \phi_{v_p \otimes e_{ij}^{(k)}}$. For the rest of this section, if $\{ u_1, \dots, u_r \}$ is a basis for a vector space $U$ and $\{\phi_{u_1}, \dots,\phi_{u_r}\}$ is the corresponding dual basis, we will sometimes write $[u]_{u_i} \coloneqq \phi_{u_i}(u)$. (Roughly speaking, we use the $\phi$ notation when we are interested in properties of the maps, and the square bracket notation when we want to perform explicit calculations.)

\begin{prop} \label{RLprop}
With the setup as above, we have
\begin{enumerate}[{\normalfont (1)},leftmargin=*,topsep=0pt,itemsep=0pt]
\item $\displaystyle{R \phi_{v_p \otimes e_{ij}^{(k)}}^{\nu^*}(v_q \otimes 1) = \frac{\dim H}{\dim V_k} [\sigma(v_q)]_{v_p} \Xi^{-1}(e_{ji}^{(k)})}$;
\item $\displaystyle{L \phi_{v_p \otimes e_{ij}^{(k)}}(v_q \otimes 1) = \frac{\dim H}{\dim V_k} [S((e_{ji}^{(k)})_{(1)}) v_q]_{v_p} (e_{ji}^{(k)})_{(2)}}$.
\end{enumerate}
\end{prop}

\begin{rem} \label{Xiinverseremark}
Observe that the right hand side of both expressions depends on $e_{ji}^{(k)}$, rather than $e_{ij}^{(k)}$. We also note that the map $\Xi^{-1}$ is given by $\Xi^{-1}(h) = \hdet(S(h_{(1)})) h_{(2)}$.
\end{rem}

\begin{proof}[Proof of Proposition \ref{RLprop}]
\begin{enumerate}[{\normalfont (1)},wide=0pt,topsep=0pt,itemsep=0pt]
\item Since $\Tr$ is non-degenerate, it suffices to show that we have an equality
\begin{align}
\Tr \Big( R \phi_{v_p \otimes e_{ij}^{(k)}}^{\nu^*}(v_q \otimes 1) e_{rs}^{(t)} \Big) = \Tr \left( \frac{\dim H}{\dim V_k} [\sigma(v_q)]_{v_p} \Xi^{-1}(e_{ji}^{(k)}) e_{rs}^{(t)} \right) \label{Rclaimedequality}
\end{align}
for all $r,s,t$. The left hand side of (\ref{Rclaimedequality}) simplifies as
\begin{align*}
\Tr \Big( R \phi_{v_p \otimes e_{ij}^{(k)}}^{\nu^*}(v_q \otimes 1) e_{rs}^{(t)} \Big) 
&= \Tr \Big( R \phi_{v_p \otimes e_{ij}^{(k)}}^{\nu^*}(v_q \otimes e_{rs}^{(t)}) \Big) \\
&= \phi_{v_p \otimes e_{ij}^{(k)}}^{\nu^*} ( v_q \otimes e_{rs}^{(t)}) \\
&= \phi_{v_p \otimes e_{ij}^{(k)}} (\sigma(v_q) \otimes \Xi(e_{rs}^{(t)})) \\
&= [\sigma(v_q)]_{v_p} [\Xi(e_{rs}^{(t)})]_{e_{ij}^{(k)}}.
\end{align*}
If we instead consider the right hand side of (\ref{Rclaimedequality}), we obtain
\begin{align}
\Tr \left( \frac{\dim H}{\dim V_k} [\sigma(v_q)]_{v_p} \Xi^{-1}(e_{ji}^{(k)}) e_{rs}^{(t)} \right)
&= \frac{\dim H}{\dim V_k} [\sigma(v_q)]_{v_p} \Tr\big( \Xi^{-1}(e_{ji}^{(k)}) e_{rs}^{(t)} \big) \nonumber \\
&= \frac{\dim H}{\dim V_k} [\sigma(v_q)]_{v_p} \Tr\big( \Xi^{-1}\big(e_{ji}^{(k)} \Xi(e_{rs}^{(t)})\big) \big) \nonumber \\
&= \frac{\dim H}{\dim V_k} [\sigma(v_q)]_{v_p} \Tr\big( e_{ji}^{(k)} \Xi(e_{rs}^{(t)}) \big) \label{Requality1}\\
&= \frac{\dim H}{\dim V_k} [\sigma(v_q)]_{v_p}  \frac{\dim V_k}{\dim H} [\Xi(e_{rs}^{(t)})]_{e_{ij}^{(k)}} \label{Requality2} \\
&= [\sigma(v_q)]_{v_p} [\Xi(e_{rs}^{(t)})]_{e_{ij}^{(k)}} \nonumber,
\end{align}
where we use Lemma \ref{integrallemma} (3) to establish (\ref{Requality1}), and Lemma \ref{tracelem} at (\ref{Requality2}). This shows that (\ref{Rclaimedequality}) holds, and so the result follows.
\item As with (1), it suffices to show that we have an equality 
\begin{align}
\Tr \Big(e_{rs}^{(t)} L \phi_{v_p \otimes e_{ij}^{(k)}}(v_q \otimes 1) \Big) = \Tr \left(  e_{rs}^{(t)} \frac{\dim H}{\dim V_k} [S((e_{ji}^{(k)})_{(1)}) v_q]_{v_p} (e_{ji}^{(k)})_{(2)} \right) \label{Lclaimedequality}
\end{align}
for all $r,s,t$. We first simplify the left hand side of this expression:
\begin{align*}
\Tr \Big(e_{rs}^{(t)} L \phi_{v_p \otimes e_{ij}^{(k)}}(v_q \otimes 1) \Big) 
&= \Tr \Big(L \phi_{v_p \otimes e_{ij}^{(k)}}\Big(e_{rs}^{(t)}(v_q \otimes 1) \Big) \Big) \\
&= \Tr \Big(L \phi_{v_p \otimes e_{ij}^{(k)}}\Big( (e_{rs}^{(t)})_{(1)} v_q \otimes (e_{rs}^{(t)})_{(2)} \Big) \Big) \\
&= \phi_{v_p \otimes e_{ij}^{(k)}}\Big( (e_{rs}^{(t)})_{(1)} v_q \otimes (e_{rs}^{(t)})_{(2)} \Big) \\
&= [(e_{rs}^{(t)})_{(1)} v_q]_{v_p} [(e_{rs}^{(t)})_{(2)}]_{e_{ij}^{(k)}}.
\end{align*}
On the other hand, if we set $x = e_{rs}^{(t)}$ and $y = e_{ji}^{(k)}$, the right hand side of (\ref{Lclaimedequality}) becomes
\begin{align}
\Tr \left(  e_{rs}^{(t)} \frac{\dim H}{\dim V_k} [S((e_{ji}^{(k)})_{(1)}) v_q]_{v_p} (e_{ji}^{(k)})_{(2)} \right)
&= \Tr \left(  x \frac{\dim H}{\dim V_k} [S(y_{(1)}) v_q]_{v_p} y_{(2)} \right) \nonumber \\
&= \frac{\dim H}{\dim V_k} [S(y_{(1)}) v_q]_{v_p} \Tr (x y_{(2)}) \nonumber \\
&= \frac{\dim H}{\dim V_k} [\Tr (x y_{(2)}) S(y_{(1)}) v_q]_{v_p} \nonumber \\
&= \frac{\dim H}{\dim V_k} [x_{(1)} \Tr (x_{(2)} y) v_q]_{v_p} \label{Lequality1} \\
&= \frac{\dim H}{\dim V_k} [x_{(1)} v_q]_{v_p} \Tr (x_{(2)} y) \nonumber \\
&= \frac{\dim H}{\dim V_k} [(e_{rs}^{(t)})_{(1)} v_q]_{v_p} \Tr\big( (e_{rs}^{(t)})_{(2)} e_{ji}^{(k)}\big) \nonumber\\
&= \frac{\dim H}{\dim V_k} [(e_{rs}^{(t)})_{(1)} v_q]_{v_p} \frac{\dim V_k}{\dim H} [(e_{rs}^{(t)})_{(2)}]_{e_{ij}^{(k)}} \label{Lequality2} \\
&= [(e_{rs}^{(t)})_{(1)} v_q]_{v_p} [(e_{rs}^{(t)})_{(2)}]_{e_{ij}^{(k)}}, \nonumber
\end{align}
where we use Lemma \ref{integrallemma} (2) at (\ref{Lequality1}), and Lemma \ref{tracelem} at (\ref{Lequality2}). Therefore (\ref{Lclaimedequality}) holds, as required. \qedhere
\end{enumerate}
\end{proof}

Before using the above result to show that $\Psi(\sfw \hash 1)$ is a twisted superpotential, we make an observation regarding a calculation in $T_H(V \otimes H)$. If $h \in H$ and $v \in V$, then
\begin{align}
h_{(2)}(S(h_{(1)})v \otimes 1) &= (h_{(2)} S(h_{(1)})v) \otimes h_{(3)} = \varepsilon(h_{(1)})v \otimes h_{(2)} = v \otimes \varepsilon(h_{(1)})h_{(2)} =(v \otimes 1) h, \label{needHSS}
\end{align}
where the second equality crucially requires $S^2 = \id$, which is the case since we assume $H$ to be semisimple. \\
\indent With this observation in hand, we are now in a position to show that $\Psi(\sfw \hash 1)$ is a twisted superpotential.

\begin{prop}
$\Psi(\sfw \hash 1) \in T_H(V \otimes H)$ is a $\nu$-twisted superpotential, where the twist is given by $\nu = \sigma \otimes \Xi$; that is,
\begin{align*}
\llbracket \phi^{\nu^*},\Psi(\sfw \hash 1) \rrbracket = \llbracket \Psi(\sfw \hash 1), \phi \rrbracket 
\end{align*}
for all $\phi \in (V \otimes H)^*$.
\end{prop}
\begin{proof}
Clearly it suffices to verify this equality as $\phi$ varies over elements in our chosen basis for $(V \otimes H)^*$. We write $\{ \phi_{v_1}, \dots, \phi_{v_r}\}$ for the basis dual to our chosen basis for $V$. We remind the reader that
\begin{align*}
\sfw = \sum_{q} \alpha_q v_{q_1} \otimes \dots \otimes v_{q_\ell}.
\end{align*}
Now, for any $p,i,j,k$ we have
\begin{align}
\Big \llbracket & \phi_{v_p \otimes e_{ij}^{(k)}}^{\nu^*}, \Psi(\sfw \hash 1) \Big \rrbracket \nonumber \\
&= \sum_q \alpha_q \Big(R \phi_{v_p \otimes e_{ij}^{(k)}}^{\nu^*}\Big)(v_{q_1} \otimes 1) \cdot  (v_{q_2} \otimes 1) \otimes \dots \otimes (v_{q_\ell} \otimes 1) \nonumber \\
&= \sum_q \alpha_q \frac{\dim H}{\dim V_k} [\sigma(v_{q_1})]_{v_p} \Xi^{-1}(e_{ji}^{(k)}) \cdot  (v_{q_2} \otimes 1) \otimes \dots \otimes (v_{q_\ell} \otimes 1) \nonumber \\
&= \frac{\dim H}{\dim V_k} \Xi^{-1}(e_{ji}^{(k)}) \sum_q \alpha_q [\sigma(v_{q_1})]_{v_p} (v_{q_2} \otimes 1) \otimes \dots \otimes (v_{q_\ell} \otimes 1) \nonumber \\
&= \frac{\dim H}{\dim V_k} \Xi^{-1}(e_{ji}^{(k)}) \Psi \bigg( \sum_q \alpha_q [\sigma(v_{q_1})]_{v_p} (v_{q_2} \otimes \dots \otimes v_{q_\ell}) \hash 1 \bigg) \nonumber \\
&= \frac{\dim H}{\dim V_k} \Xi^{-1}(e_{ji}^{(k)}) \Psi \big( \llbracket \phi_{v_p}^{\sigma^*},\sfw \rrbracket \hash 1 \big) \nonumber \\
&= \frac{\dim H}{\dim V_k} \hdet\big(S((e_{ji}^{(k)})_{(1)})\big) (e_{ji}^{(k)})_{(2)} \Psi \big( \llbracket \sfw,\phi_{v_p} \rrbracket \hash 1 \big) \nonumber \\
&= \frac{\dim H}{\dim V_k} (e_{ji}^{(k)})_{(2)} \Psi \Big( \big \llbracket \hdet\big(S((e_{ji}^{(k)})_{(1)})\big) \sfw,\phi_{v_p}\big \rrbracket \hash 1 \Big) \nonumber \\
&= \frac{\dim H}{\dim V_k} (e_{ji}^{(k)})_{(2)} \Psi \Big( \big \llbracket S((e_{ji}^{(k)})_{(1)}) \sfw,\phi_{v_p}\big \rrbracket \hash 1 \Big) \label{usinghdetlem} \\
&= \frac{\dim H}{\dim V_k} (e_{ji}^{(k)})_{(\ell+1)} \Psi \Bigg( \Big \llbracket \sum_q \alpha_{q} S((e_{ji}^{(k)})_{(\ell)}) v_{q_1} \otimes S((e_{ji}^{(k)})_{(\ell-1)}) v_{q_2} \otimes \dots \otimes S((e_{ji}^{(k)})_{(1)}) v_{q_\ell}, \phi_{v_p}\Big \rrbracket \hash 1 \Bigg) \nonumber \\
&= \frac{\dim H}{\dim V_k} (e_{ji}^{(k)})_{(\ell+1)} \sum_q \alpha_{q} \Big( S((e_{ji}^{(k)})_{(\ell)}) v_{q_1} \otimes 1 \Big) \otimes \dots \otimes \Big( S((e_{ji}^{(k)})_{(2)}) v_{q_{\ell-1}} \otimes 1 \Big) \big[ S((e_{ji}^{(k)})_{(1)}) v_{q_\ell} \big]_{v_p} \nonumber \\
&= \frac{\dim H}{\dim V_k} \sum_q \alpha_{q} \Big( (e_{ji}^{(k)})_{(\ell+1)} \big( S((e_{ji}^{(k)})_{(\ell)}) v_{q_1} \otimes 1 \big) \Big) \otimes \dots \otimes \Big( S((e_{ji}^{(k)})_{(2)}) v_{q_{\ell-1}} \otimes 1 \Big) \big[ S((e_{ji}^{(k)})_{(1)}) v_{q_\ell} \big]_{v_p} \nonumber \\
&= \frac{\dim H}{\dim V_k} \sum_q \alpha_{q} \Big( \big( v_{q_1} \otimes 1 \big) (e_{ji}^{(k)})_{(\ell)} \Big) \otimes \dots \otimes \Big( S((e_{ji}^{(k)})_{(2)}) v_{q_{\ell-1}} \otimes 1 \Big) \big[ S((e_{ji}^{(k)})_{(1)}) v_{q_\ell} \big]_{v_p} \nonumber \\
&= \frac{\dim H}{\dim V_k} \sum_q \alpha_{q} \big( v_{q_1} \otimes 1 \big) \otimes \Big( \big(  v_{q_2} \otimes 1 \big) (e_{ji}^{(k)})_{(\ell-1)} \Big) \otimes \dots \otimes \Big( S((e_{ji}^{(k)})_{(2)}) v_{q_{\ell-1}} \otimes 1 \Big) \big[ S((e_{ji}^{(k)})_{(1)}) v_{q_\ell} \big]_{v_p} \nonumber \\
& \hspace{10pt} \vdots \nonumber \\
&= \frac{\dim H}{\dim V_k} \sum_q \alpha_{q} \big( v_{q_1} \otimes 1 \big) \otimes \dots \otimes \big( v_{q_{\ell-1}} \otimes 1 \big) (e_{ji}^{(k)})_{(2)} \big[ S((e_{ji}^{(k)})_{(1)}) v_{q_\ell} \big]_{v_p} \nonumber \\
&= \sum_q \alpha_{q} \big( v_{q_1} \otimes 1 \big) \otimes \dots \otimes \big( v_{q_{\ell-1}} \otimes 1 \big) \cdot \frac{\dim H}{\dim V_k} \big[ S((e_{ji}^{(k)})_{(1)}) v_{q_\ell} \big]_{v_p} (e_{ji}^{(k)})_{(2)} \nonumber \\
&= \sum_q \alpha_{q} \big( v_{q_1} \otimes 1 \big) \otimes \dots \otimes \big( v_{q_{\ell-1}} \otimes 1 \big) \cdot \Big(L \phi_{v_p \otimes e_{ij}^{(k)}}\Big)(v_{q_\ell} \otimes 1) \nonumber \\
&= \llbracket \Psi(\sfw \hash 1), \phi_{v_p \otimes e_{ij}^{(k)}} \rrbracket, \nonumber
\end{align}
where we make use of the formulae obtained in Proposition \ref{RLprop}, where the equality at (\ref{usinghdetlem}) follows from Theorem \ref{hdetthm}, and where we have made use of the observation (\ref{needHSS}) multiple times. This establishes the claimed identity on a basis for $(V \otimes H)^*$, and hence for all $\phi$ in this space.
\end{proof}

The last step is to show that the ideal of relations defining $A \hash H$ is derived from the superpotential $\Psi(\sfw \hash 1)$. To be more precise, first note that if $I$ is an ideal in $T_\Bbbk(V)$, then 
\begin{align*}
I \hash H = \langle \Psi(I \otimes H) \rangle,
\end{align*}
and therefore
\begin{align*}
A \hash H = \frac{T_\Bbbk(V)}{\langle \im \partial^{\ell-m}_{\sfw} \rangle} \hash H \cong \frac{T_\Bbbk(V) \hash H}{\langle \im \partial^{\ell-m}_{\sfw} \rangle \hash H}
\cong \frac{T_H(V \otimes H)}{\langle \Psi(\im \partial^{\ell-m}_\sfw \hash H) \rangle}.
\end{align*}
So to show that we have an isomorphism $A \hash H \cong \scrD(\Psi(\sfw \hash 1),\ell-m)$, it remains to show that we have an equality $\Psi(\im \partial^{\ell-m}_\sfw \hash H) = \im \partial^{\ell-m}_{\Psi(\sfw \hash 1)}$.

\begin{lem}
There is an equality
\begin{align*}
\Psi(\im \partial^{\ell-m}_\sfw \hash H) = \im \partial^{\ell-m}_{\Psi(\sfw \hash 1)}.
\end{align*}
\end{lem}
\begin{proof}
($\supseteq$) Set $s = \ell-m$. Recall that $V$ and $H$ have respective $\Bbbk$-bases $\{ v_1, \dots, v_r \}$ and $\{ e_{ij}^{(k)} \}$, so that $(V \otimes H)^*$ has basis $\{ \phi_{v_p \otimes e_{ij}^{(k)}} \}$ dual to that of $V \otimes H$. To establish the claimed inclusion, it suffices to show that any element of the form 
\begin{align*}
\partial^{s}_{\Psi(\sfw \hash 1)}\Big( \phi_{v_{p_1} \otimes e_{i_1 j_1}^{(k_1)}} \otimes \dots \otimes \phi_{v_{p_s} \otimes e_{i_s j_s}^{(k_s)}} \otimes h \Big)
\end{align*}
lies in $\Psi(\im \partial^{s}_\sfw \hash H)$. \\
\indent Using the same method as in the proof of Proposition \ref{RLprop}, one can show that
\begin{align*}
(R \phi_{v_{p} \otimes e_{i j}^{(k)}})(v_q \otimes 1) = \frac{\dim H}{\dim V_k} [v_q]_{v_p} e_{ji}^{(k)}.
\end{align*}
Therefore
\begin{align*}
\partial^{s}_{\Psi(\sfw \hash 1)}\Big( &\phi_{v_{p_1} \otimes e_{i_1 j_1}^{(k_1)}} \otimes \dots \otimes \phi_{v_{p_s} \otimes e_{i_s j_s}^{(k_s)}} \otimes h \Big) \\
&= \sum_q \alpha_q (R \phi_{v_{p_s} \otimes e_{i_s j_s}^{(k_s)}})(v_{q_1} \otimes 1) \dots (R \phi_{v_{p_1} \otimes e_{i_1 j_1}^{(k_1)}})(v_{q_s} \otimes 1) (v_{q_{s+1}} \otimes 1) \otimes \dots \otimes (v_{q_\ell} \otimes 1) \\
&= \frac{(\dim H)^{s}}{\dim V_{k_1} \dots \dim V_{k_s}} \sum_q \alpha_q [v_{q_1}]_{v_{p_s}} \dots [v_{q_s}]_{v_{p_1}} \underbrace{e_{j_s i_s}^{k_s} \dots e_{j_1 i_1}^{k_1}}_{\eqqcolon x \in H} (v_{q_{s+1}} \otimes 1) \otimes \dots \otimes (v_{q_{\ell}} \otimes h) \\
&= \frac{(\dim H)^{s}}{\dim V_{k_1} \dots \dim V_{k_s}} \sum_q \alpha_q [v_{q_1}]_{v_{p_s}} \dots [v_{q_s}]_{v_{p_1}}  (x_{(1)} v_{q_{s+1}} \otimes 1) \otimes \dots \otimes (x_{(\ell-s)} v_{q_{\ell}} \otimes x_{(\ell-s+1)} h) \\
&= \frac{(\dim H)^{s}}{\dim V_{k_1} \dots \dim V_{k_s}} \Psi \bigg( \sum_q \alpha_q [v_{q_1}]_{v_{p_s}} \dots [v_{q_s}]_{v_{p_1}} (x_{(1)} v_{q_{s+1}}) \otimes \dots \otimes (x_{(\ell-s)} v_{q_{\ell}}) \hash x_{(\ell-s+1)} h \bigg) \\
&= \frac{(\dim H)^{s}}{\dim V_{k_1} \dots \dim V_{k_s}} \Psi \bigg( x_{(1)} \bigg(\sum_q \alpha_q [v_{q_1}]_{v_{p_s}} \dots [v_{q_s}]_{v_{p_1}}  v_{q_{s+1}} \otimes \dots \otimes v_{q_{\ell}} \bigg) \hash x_{(2)} h \bigg) \\
&= \frac{(\dim H)^{s}}{\dim V_{k_1} \dots \dim V_{k_s}} \Psi \Big( x_{(1)} \partial^{s}_\sfw\big(\phi_{v_{p_1}} \otimes \dots \otimes \phi_{v_{p_s}}\big) \hash x_{(2)} h \Big).
\end{align*}
Here, $x_{(1)} \partial^{s}_\sfw\big(\phi_{v_{p_1}} \otimes \dots \otimes \phi_{v_{p_s}} \otimes x_{(2)} h \big)$ lies in $\im \partial_\sfw^{s}$ since this set is closed under the action from $H$, and hence the whole expression lies in $\Psi(\im \partial^{s}_\sfw \hash H)$, as required. \\
\indent ($\subseteq$) We use the same bases of $V$ and $V^*$ as above, but we choose a different basis for $H$. Indeed, let $h_1 = 1$ and extend this to a basis $\{ h_1, h_2, \dots, h_m \}$ for $H$ with the property that $\Tr(h_i) = 0$ for $i \neq 1$; we can do this by first choosing any basis $\{ h_1=1, h_2', \dots, h_m'\}$ and then setting $h_i = h_i' - \frac{\Tr(h_i')}{\Tr(h_1)} h_1$ for $i \geqslant 2$. We now determine the values of $R \phi_{v_p \otimes h_1} = R \phi_{v_p \otimes 1}$ on elements of the form $v_q \otimes 1$. So suppose $(R \phi_{v_p \otimes 1})(v_q \otimes 1) = h$ for some $h \in H$. Then, for $1 \leqslant i \leqslant m$,
\begin{align*}
\Tr( h h_i ) &= \Tr\big((R \phi_{v_p \otimes 1})(v_q \otimes 1)h_i\big) = \Tr\big((R \phi_{v_p \otimes 1})(v_q \otimes h_i)\big) = \phi_{v_p \otimes 1}(v_q \otimes h_i) \\
&= \left \{ 
\begin{array}{ll}
[v_q]_{v_p} & \text{if } i=1,  \\
0 & \text{otherwise,}
\end{array}
\right. \\
&= \Tr\big([v_q]_{v_p} h_i\big).
\end{align*}
Since $\Tr$ is nondegenerate, it follows that $h = [v_q]_{v_p} 1_H$. \\
\indent Now, to establish this desired inclusion, by linearity it suffices to show that any element of the form $\Psi\big(\partial^{s}_\sfw(\phi_{v_{p_1}} \otimes \dots \otimes \phi_{v_{p_s}}) \hash h\big)$ lies in $\im \partial^{s}_{\Psi(\sfw \hash 1)}$; that is, we must show that
\begin{align}
\sum_q \alpha_q [v_{q_1}]_{v_{p_s}} \dots [v_{q_s}]_{v_{p_1}} (v_{q_{s+1}} \otimes 1) \otimes \dots \otimes (v_{q_{\ell-1}} \otimes 1) \otimes (v_{q_{\ell}} \otimes h) \in \im \partial^{s}_{\Psi(\sfw \hash 1)}. \label{partialinclusion}
\end{align}
Indeed, we have
\begin{align*}
\im \partial^{s}_{\Psi(\sfw \hash 1)} &\ni \partial^{s}_{\Psi(\sfw \hash 1)}\big( \phi_{v_{p_1} \otimes 1} \otimes \dots \otimes \phi_{v_{p_s} \otimes 1} \otimes h \big)  \\
&= \sum_q \alpha_q (R \phi_{v_{p_s} \otimes 1})(v_{q_1} \otimes 1) \dots (R \phi_{v_{p_1} \otimes 1})(v_{q_s} \otimes 1) (v_{q_{s+1}} \otimes 1) \otimes \dots \otimes (v_{q_{\ell-1}} \otimes 1) \otimes (v_{q_{\ell}} \otimes h) \\
&=\sum_q \alpha_q [v_{q_1}]_{v_{p_s}} \dots [v_{q_s}]_{v_{p_1}} (v_{q_{s+1}} \otimes 1) \otimes \dots \otimes (v_{q_{\ell-1}} \otimes 1) \otimes (v_{q_{\ell}} \otimes h),
\end{align*}
and so (\ref{partialinclusion}) holds.
\end{proof}

To summarise, in this section we have shown the following:

\begin{thm} \label{smashisderquot}
Suppose that the pair $(A,H)$ satisfies Hypothesis \ref{mainhypothesis}, and write $A = \scrD(\sfw,\ell-m)$ for some $\sigma$-twisted superpotential $\sfw$. Then
\begin{align*}
A \hash H \cong \scrD(\Psi(\sfw \hash 1), \ell-m),
\end{align*}
where $\Psi(\sfw \hash 1)$ is a twisted superpotential, where the twist $\nu$ is given by $\nu = \sigma \otimes {}_{\hdet} \Xi$.
\end{thm}

\section{From smash products to quivers}
In the previous section, we showed that if $H$ is a semisimple Hopf algebra acting homogeneously on an $m$-Koszul AS regular algebra, then $A \hash H$ is a derivation-quotient algebra. Moreover, we gave an explicit formula for the twisted superpotential defining $A \hash H$. We now use this to write down a path algebra with relations, $\Lambda$, to which $A \hash H$ is Morita equivalent. As explained in the introduction, the fact that this is possible follows from general results in the literature, but our result provides a much more explicit presentation; in particular, the relations in $\Lambda$ are also obtained from a twisted superpotential. 
We will also deduce a number of corollaries; for example, we will give a criterion in terms of $\Lambda$ which shows when the Auslander map is an isomorphism. \\
\indent In this section, it will sometimes be convenient to compose morphisms from left to right; when this is the case, our notation for the application of a morphism $f$ to an element $a$ will be $a^f$, and composition of a morphism $f$ followed by $g$ will be denoted $f \sbullet g$. \\
\indent Assume that the pair $(A,H)$ satisfies Hypothesis \ref{mainhypothesis} as usual. As in (\ref{awdecomp}), $H$ has some Artin--Wedderburn decomposition
\begin{align*}
H \cong \bigoplus_{k=0}^n \text{Mat}_{\dim V_k}(\Bbbk),
\end{align*}
where $V_0, \dots, V_n$ are the irreducible representations of $H$. Again, we fix such an isomorphism which allows us to give the right hand side the structure of a Hopf algebra. Using this, we obtain $n+1$ pairwise orthogonal idempotents $e_0, e_1, \dots, e_n$ in $H$ such that $V_i \cong H e_i$; more concretely, one could take $e_i \coloneqq e_{11}^{(i)}$. Then $e = \sum_{i=0}^n e_i$ is a full idempotent of $H$, and so $eHe$ is Morita equivalent to $H$. In fact, if we interpret $e$ as an element of $A \hash H$, then $e$ is still a full idempotent, and $e(A \hash H)e$ is Morita equivalent to $A \hash H$. We now explain how the results of the previous section allow us to show that $e(A \hash H)e$ can be viewed as a quiver with relations which are derived from a twisted superpotential. \\
\indent To begin with, we need to ensure that the idempotent $e$ is closed under the action of the twist $\nu$; in other words, recalling that $\restr{\nu}{H} = \Xi$ is the left winding automorphism of $H$ associated to $\hdet_A$, we require $\Xi(e) = e$. One way to achieve this is explained before Theorem 3.2 in \cite{bsw}, which we paraphrase below. \\
\indent The goal is to find idempotents $e_0, e_1, \dots, e_n \in H$ such that $e = \sum_{i=0}^n e_i$ is a full idempotent, $H e_i \cong V_i$ and, for each $i$, $e_i \Psi(\sfw \hash 1) = \Psi(\sfw \hash 1) e_j$ for some $j$. By Lemma \ref{twistedlem}, this is equivalent to requiring $e_j = \Xi(e_i)$. Now, $\Bbbk \sfw$ is a one-dimensional representation of $H$, so the functor $\Bbbk \sfw \otimes -$ induces a permutation of the $V_i$, and hence partitions the set of irreducible representations into orbits. \\
\indent Fix some irreducible representation $U$, and let $r \geqslant 1$ be minimal so that $(\Bbbk \sfw)^{\otimes r} \otimes U \cong U$. In particular, this means that the modules $(\Bbbk \sfw)^{\otimes i} \otimes U$ for $0 \leqslant i \leqslant r-1$ are pairwise nonisomorphic. Fix an isomorphism $\psi : U \to (\Bbbk \sfw)^{\otimes r} \otimes U$, and let $u \in U$ be an eigenvector for $\psi$, in the sense that $\psi(u) = \lambda \sfw^{\otimes r} \otimes u$ for some $\lambda \in \Bbbk$. Now let $f_1 \in H$ be a primitive idempotent such that the map $H f_1 \to U, \hspace{2pt} f_1 \mapsto u$ is an isomorphism. There is an isomorphism
\begin{align*}
\Hom_H(H f_1, (\Bbbk \sfw)^{\otimes r} \otimes H f_1) \cong f_1 ((\Bbbk \sfw)^{\otimes r} \otimes H) f_1,
\end{align*}
and under this isomorphism, the map
\begin{align*}
\psi' : H f_1 \to (\Bbbk \sfw)^{\otimes r} \otimes H f_1, \quad \psi'(f_1) = \lambda \sfw^{\otimes r} \otimes f_1
\end{align*}
corresponds to the element $f_1$. In particular, we find that $f_1 (\sfw^{\otimes r} \hash 1) = (\sfw^{\otimes r} \hash 1) f_1$ in $T_\Bbbk(V) \hash H$. \\
\indent Now let $f_2, f_3, \dots, f_r \in H \subseteq T_\Bbbk(V) \hash H$ be the primitive idempotents satisfying $f_i (\sfw \hash 1) = (\sfw \hash 1) f_{i+1}$ for $1 \leqslant i \leqslant r-1$. By construction, we also have $f_r (\sfw \hash 1) = (\sfw \hash 1) f_1$. In particular, by Lemma \ref{twistedlem}, $\Xi(f_i) = f_{i+1}$, where the subscripts are read modulo $r$ if necessary. Therefore the set $\{f_1, \dots, f_r \}$ is closed under the twist $\nu$. (Of course, we could have just picked $f_1$ to be any idempotent in $H$ with $H f_1 \cong U$ and then set $f_{i+1} = \Xi(f_i)$ for $1 \leqslant i \leqslant r-1$, but then there is no guarantee that $\Xi^r(f_1) = f_1$.) \\
\indent Repeating this procedure for each orbit of irreducible representations, we obtain $n+1$ elements of $H$ (which we may also view as elements of $A \hash H$) which we label $e_0, \dots, e_n$, with the following properties:
\begin{itemize}[leftmargin=25pt,topsep=0pt,itemsep=0pt]
\item The $e_i$ are pairwise orthogonal idempotents;
\item The set $\{ e_0, \dots, e_n \}$ is closed under the twist $\nu$;
\item $H e_i \cong V_i$, and $H e_0$ is the trivial representation;
\item The element $e \coloneqq e_0 + \dots + e_n$ is a full idempotent of $H$, and the same is true if we view $e$ as an element of $A \hash H$.
\item $e_0 (A \hash H) e_0 \cong A^H$, by \cite[4.3.4 Lemma]{montgomery2}.
\end{itemize}
Finally, by combining \cite[Lemma 2.2]{bsw} with our Theorem \ref{smashisderquot}, we deduce the following:

\begin{prop} \label{moritaequivprop}
Suppose that the pair $(A,H)$ satisfies Hypothesis \ref{mainhypothesis}. Then $A \hash H$ is Morita equivalent to
\begin{align}
e (A \hash H) e \cong \scrD(e \Psi(\sfw \hash 1) e, \ell - m). \label{cornerisder}
\end{align}
In particular, $A^H \cong e_0 \scrD(e \Psi(\sfw \hash 1) e, \ell - m) e_0$.
\end{prop}

\subsection{Viewing $\scrD(e \Psi(\sfw \hash 1) e, \ell - m)$ as a quiver with relations} We now show how the algebra (\ref{cornerisder}) can be viewed as the path algebra of a quiver with relations, and that these relations are derived from the superpotential $e \Psi(\sfw \hash 1) e$, which can be interpreted as a linear combination of paths in the quiver. \\
\indent We first show that we can view $e (T_\Bbbk(V) \otimes H) e \hspace{3pt}(\cong eT_H(V \otimes H)e )$ as the path algebra of a quiver. The vertices of $Q$ correspond to the idempotents $e_i$, and the arrows correspond to elements of $e (V \otimes H) e$. In particular, the number of arrows from vertex $i$ to vertex $j$ is equal to the dimension of
\begin{align}
e_i (V \otimes H) e_j \cong \Hom_H(H e_i, (V \otimes H) e_j) = \Hom(V_i, V \otimes V_j); \label{arrowsinquiver}
\end{align}
that is, the number of arrows from $i$ to $j$ is equal to the multiplicity of the irreducible module $V_i$ in $V \otimes V_j$. The set of arrows of $Q$ is given by the union of bases for the vector spaces $e_i (V \otimes H) e_j$. It follows that $Q$ is the left McKay quiver for the action of $H$ on $A$, and that $e (T_\Bbbk(V) \otimes H) e$ is isomorphic to the path algebra of $Q$. \\
\indent Using (\ref{arrowsinquiver}), we will be able to identify arrows in $Q$ (i.e.\ elements of $e(V \otimes H)e$) with morphisms of $H$-modules. Since our preference will be to compose arrows in quivers from left to right, in this section it will be convenient to compose the corresponding morphisms from left to right, using the notation at the start of this section. \\
\indent By (\ref{arrowsinquiver}), we can identify every arrow $\alpha : i \to j$ in $Q$ (which is itself an element of $e_i (V \otimes H) e_j$) with an $H$-module map $\phi_\alpha : V_i \to V \otimes V_j$. Extending this, to every path $p = \alpha_1 \dots \alpha_k \in e_i (V^{\otimes k} \otimes H) e_j \subseteq T_\Bbbk(V) \hash H$ in $Q$ we can associate an $H$-module map $\phi_p \coloneqq \phi_{\alpha_1} \sbullet (\id_V \otimes \phi_{\alpha_2}) \sbullet \dots \sbullet (\id_V^{\otimes k-1} \otimes \phi_{\alpha_k})$. Abusing notation, we will simply write $\phi_p = \phi_{\alpha_1} \sbullet \dots \sbullet \phi_{\alpha_k} \in \Hom_H(V_{t(\alpha_1)}, V^{\otimes k} \otimes V_{h(\alpha_k)})$. Conversely, it is straightforward to show that every map in $\Hom_H(V_i, V^{\otimes k} \otimes V_j)$ can be decomposed uniquely as a sum of morphisms of this form. \\
\indent We also wish to identify certain morphisms which are dual to those corresponding to the arrows in $Q$. Recall that the number of arrows from vertex $i$ to vertex $j$ is equal to the dimension of $\Hom_H(V_i, V \otimes V_j)$; call this number $m_{ij}$. In particular, the multiplicity of $V_i$ as an irreducible summand of $V \otimes V_j$ is $m_{ij}$, and hence we also have $ \dim \Hom_H(V \otimes V_j, V_i) = m_{ij}$. Therefore for each arrow $\alpha : i \to j$ we can also choose an $H$-module map $\psi_\alpha : V \otimes V_j \to V_i$. In particular, we can choose each map $\psi_\alpha$ so that it is dual to $\phi_\alpha$ in the sense that, for any arrow $\beta : i \to j$, the composition of $H$-module maps
\begin{align}
V_{i} \xrightarrow{\phi_\alpha} V \otimes V_{j} \xrightarrow{\psi_\beta}  V_{i} \label{arrowidentity}
\end{align}
is the identity on $V_i$ if $\alpha = \beta$, and the zero map otherwise. In turn, for every path $p = \alpha_1 \dots \alpha_k$ in $Q$, we obtain $H$-module maps $\psi_p \coloneqq \psi_{\alpha_k} \sbullet \dots \sbullet \psi_{\alpha_1} \in \Hom_H(V^{\otimes k} \otimes V_{h(\alpha_k)}, V_{t(\alpha_1)})$ with the property that, by (\ref{arrowidentity}), if $q$ is any other path between the same vertices, the composition
\begin{align}
V_{t(p)} \xrightarrow{\phi_p} V^{\otimes k} \otimes V_{h(p)} \xrightarrow{\psi_q} V_{t(p)} \label{compforpaths}
\end{align}
is the identity on $V_{t(p)}$ if $p = q$, and the zero map otherwise. Again, we abuse notation here by writing $\psi_{\alpha_k} \sbullet \psi_{\alpha_{k-1}} \sbullet \dots \sbullet \psi_{\alpha_1}$ rather than $(\id_V^{\otimes k-1} \otimes \psi_{\alpha_k}) \sbullet ( \id_V^{\otimes k-2} \otimes \psi_{\alpha_{k-1}}) \sbullet \dots \sbullet \psi_{\alpha_1}$. \\
\indent We now use the above to write the element $e (\sfw \hash 1) e \in e(T_\Bbbk(V) \hash H)e$ in terms of paths in the quiver. Write
\begin{align*}
e (\sfw \hash 1) e = \sum_p \lambda_p p \eqqcolon \Phi
\end{align*}
for some $\lambda_p \in \Bbbk$, where the sum runs over all paths $p$ and all but finitely many $\lambda_p$ are zero. Let $\tau$ be the permutation of the vertices satisfying 
\begin{align*}
\Bbbk \sfw \otimes V_i \cong V_{\tau(i)}.
\end{align*}
By our choice of idempotents, we can (and do) assume that this isomorphism is given by the map
\begin{align*}
\theta_i : V_{\tau(i)} \to \Bbbk \sfw \otimes V_i, \qquad \theta_i(e_{\tau(i)}) = \sfw \otimes e_i.
\end{align*}
In particular, in $T_\Bbbk(V) \hash H$ this means that $e_{\tau(i)} (\sfw \otimes 1) = (\sfw \otimes 1) e_i$, and so $\nu(e_{\tau(i)}) = e_i$, i.e.\ the permutation of the vertices induced by $\nu$ is inverse to the permutation $\tau$. \\
\indent Fix a path $p$ in $Q$ of length $\ell$ from vertex $i$ to vertex $j$ (where we recall that $\sfw \in V^{\otimes \ell}$), and consider the composition
\begin{align}
V_{\tau(j)} \xrightarrow{\theta_{\tau(j)}} \Bbbk \sfw \otimes V_j \xhookrightarrow{\iota} V^{\otimes \ell} \otimes V_j \xrightarrow{\psi_p}  V_i. \label{importantcomp}
\end{align}
Since this composition is a map of $H$-modules and the domain and codomain are irreducible, Schur's Lemma implies that it is a scalar multiple of the identity, say $\mu_p \id_{V_i}$. In particular, 
\begin{align}
\mu_p \neq 0 \quad \Leftrightarrow \quad i = \tau(j) \label{startendvertex}
\end{align}
Applying this map to $e_i = e_{\tau(j)}$:
\begin{align}
e_i \mapsto \sfw \otimes e_j = (e (\sfw \otimes 1) e) e_j = \sum_{q : h(q) = j} \lambda_q q = \sum_{q : h(q) = j} \lambda_q e_i^{\phi_q} 
\mapsto \sum_{q : h(q) = j} \lambda_q e_i^{\phi_q \sbullet \psi_p} 
= \lambda_p e_i. \label{supercoeff}
\end{align}
\noindent Therefore $\lambda_p = \mu_p$. Since $\Psi : T_\Bbbk(V) \otimes H \to T_H(V \otimes H)$ is an isomorphism which restricts to the identity on the vertices and arrows of $Q$, the same coefficients appear in the terms of $e\Psi(\sfw \hash 1)e$, which we also call $\Phi$. In particular, we have a recipe for writing down $\Phi$ in terms of paths in the quiver. We can now reformulate Proposition \ref{moritaequivprop} as follows:

\begin{thm} \label{pathalgthm}
Suppose that the pair $(A,H)$ satisfies Hypothesis \ref{mainhypothesis}. Then $A \hash H$ is Morita equivalent to $\Lambda \coloneqq \scrD(\ell-m, \Phi)$, where $Q$ is the left McKay quiver for the action of $H$ on $A$, and $\Phi$ is a twisted superpotential in the path algebra of $Q$. Explicitly, $\Lambda = e(A \hash H)e$ so, in particular, $A^H \cong e_0 \Lambda e_0$.
\end{thm}

\begin{example} \label{KSexample}
As a first example, it is well-known that if $\Bbbk[u,v]^G$ is a Kleinian singularity with associated McKay quiver $Q$, then the algebra $\Lambda$ from Theorem \ref{pathalgthm} is the preprojective algebra $\Pi(Q)$; see \cite[Corollary 4.2]{bsw}.
\end{example}

\begin{rem}
Suppose that $H$ is commutative; that is, $H$ is the dual of the group algebra of a finite group (this includes the possibility that $H = \Bbbk G$ where $G$ is a finite abelian group, since in that case $H \cong H^*$). Then $H \cong \Bbbk^{n+1}$ as $\Bbbk$-algebras, and so the full idempotent $e$ which gives rise to the Morita equivalence between $\Lambda$ and $A \hash H$ is simply the identity. Therefore $A \hash H = e(A \hash H)e \cong \Lambda$.
\end{rem}

\begin{rem} \label{altmethod}
Given an arrow $\alpha : i \to j$ in the quiver for $A \hash H$, in practice it is easier to write down an explicit formula for the map $\phi_\alpha : V_i \to V \otimes V_j$ than it is for the dual map $\psi_\alpha : V \otimes V_j \to V_i$. We now give an alternative description of the maps $\psi_\alpha$ which is more amenable to computations, and explain how they can be used to obtain the coefficients as in (\ref{supercoeff}). \\
\indent Given finite-dimensional $H$-modules $U,V,W$, it is straightforward to show that the map 
\begin{align*}
\eta : \Hom_H(U, V^* \otimes W) \to \Hom_H(V \otimes U, W), \qquad \eta(\phi)(v \otimes u) = \sum_i f_i(v) w_i, \text{ where } \phi(u) = \sum_i f_i \otimes w_i
\end{align*}
is a vector space isomorphism (here, we require $S^2 = \id$). In particular, if $\alpha : i \to j$ is an arrow, then there is a map $\xi_\alpha : V_j \to V^* \otimes V_i$ such that $\eta (\xi_\alpha) = \psi_\alpha$. The duality condition (\ref{arrowidentity}) then means that the composition  
\begin{align}
V_{i} \xrightarrow{\phi_\alpha} V \otimes V_{j} \xrightarrow{\xi_\beta} V \otimes V^* \otimes V_{i} \xrightarrow{\text{eval} \otimes \id} V_i \label{arrowidentity2}
\end{align}
is the identity when $\alpha = \beta$, and the zero map otherwise. These maps can be extended to paths in the same way as in (\ref{compforpaths}) to obtain maps $\xi_p$. Finally, if $p$ is a path of length $\ell$ in the quiver, then the composition
\begin{align}
V_{\tau(j)} \xrightarrow{\theta_{\tau(j)}} \Bbbk \sfw \otimes V_j \xhookrightarrow{\hspace{3pt}\iota\hspace{3pt}} V^{\otimes \ell} \otimes V_j \xrightarrow{\id \otimes \xi_p} V^{\otimes \ell} \otimes (V^*)^{\otimes \ell} \otimes V_i \xrightarrow{\text{eval} \otimes \id} V_i \label{Phicoeffs}
\end{align}
is equal to $\lambda_p \id_{V_i}$, where $\lambda_p$ is the coefficient of $p$ in $\Phi$. (Here, we emphasise that the evaluation map $V^{\otimes \ell} \otimes (V^*)^{\otimes \ell} \to \Bbbk$ is given by $v_1 \otimes \dots v_\ell \otimes f_1 \otimes \dots \otimes f_\ell \mapsto f_1(v_\ell) \dots f_\ell(v_1)$ to ensure that it is an $H$-module morphism, as in (\ref{tensordual}).) 
\end{rem}

\begin{rem} \label{orbit}
As noted in \cite[p.\ 1508]{bsw}, the exact form of the twisted superpotential, and hence the relations, depends highly on the choice of representatives in $e(A \hash H)e$ for the arrows in $Q$. There is an action of the graded automorphism group $\Autgr (\Bbbk Q)$ on $(\Bbbk Q)_\ell$, and all twisted superpotentials that give isomorphic derivation-quotient algebras lie in the same orbit under this action.
\end{rem}

We now discuss a number of corollaries and applications of Theorem \ref{pathalgthm}.

\subsection{Applications to the Auslander map} 
Let $(A,H)$ be a pair satisfying Hypothesis \ref{mainhypothesis}. Recall that the Auslander map is 
\begin{align*}
\gamma : A \hash H \to \End(A_{A^H}), \quad \gamma(a \hash h)(b) = a (h \cdot b).
\end{align*}
Theorem \ref{introbhz} provides a computationally useful criterion for establishing when the Auslander map is an isomorphism. However, applying this result can still be quite involved; for example, lengthy calculations are required in \cite{ckwzi,won,crawford19} before Theorem \ref{introbhz} can be applied. The following result provides a similar criterion, but allows computations to be performed in $\Lambda$ rather than $A \hash H$.

\begin{cor} \label{austhmcor}
Suppose that the pair $(A,H)$ satisfies Hypothesis \ref{mainhypothesis}, and that $A$ is GK-Cohen-Macaulay, as in Theorem \ref{introbhz}. Let $\Lambda$ be the algebra from Theorem \ref{pathalgthm}. Then the Auslander map is an isomorphism if and only if $\GKdim \Lambda/\langle e_0 \rangle \leqslant \GKdim A -2$.
\end{cor}
\begin{proof}
Let $T = A \hash H$, and let $e = e_0 + e_1 + \dots + e_n$ be the sum of idempotents from before Proposition \ref{moritaequivprop}, so that $\Lambda = eTe$. In particular, $e_0 = 1 \hash t$, where $t$ is the integral from Theorem \ref{introbhz}. The (graded) Morita equivalence between $\Lambda$ and $T$ is induced by the mutually inverse equivalences $eT \otimes_T - : T \leftgr \to \Lambda \leftgr$ and $Te \otimes_\Lambda - : \Lambda \leftgr \to T \leftgr$. Now,
\begin{align}
\frac{\Lambda}{\langle e_0 \rangle} = \frac{e \Lambda e}{e \Lambda e_0 \Lambda e} \cong \frac{eT \otimes_T T \otimes_T Te }{eT \otimes_T T e_0 T \otimes_T Te} \cong eT \otimes_T \frac{T}{\langle e_0 \rangle} \otimes_T Te. \label{moritatensors}
\end{align}
By two applications of \cite[Proposition 3.14]{mandr}, we find that $\GKdim \Lambda/\langle e_0 \rangle \leqslant \GKdim T/\langle e_0 \rangle$. Tensoring (\ref{moritatensors}) on the left and right by $Te$ and $eT$, respectively, we find that 
\begin{align*}
\frac{T}{\langle e_0 \rangle} \cong Te \otimes_\Lambda \frac{\Lambda}{\langle e_0 \rangle} \otimes_\Lambda eT,
\end{align*}
and \cite[Proposition 3.14]{mandr} now shows that $\GKdim T/\langle e_0 \rangle \leqslant \GKdim \Lambda/\langle e_0 \rangle$. Hence we have an equality, and the claim now follows by Theorem \ref{introbhz}.
\end{proof}

One advantage of this result is that it is often easier to determine the GK dimension of the algebra $\Lambda/\langle e_0 \rangle$ than that of $(A \hash H)/\langle e_0 \rangle$. For an example of this, see the last paragraph of Theorem \ref{ckwzproof}, and compare with \cite[Lemma 4.6, Proposition 4.7]{ckwzi}. \\
\indent It is possible for different pairs $(A,H)$ and $(B,K)$ satisfying Hypothesis \ref{mainhypothesis} to give rise to the same algebra $\Lambda$ (more generally, the algebras arising from Theorem \ref{pathalgthm} may only be isomorphic, but we can obtain equality by choosing different representatives for the arrows; see Remark \ref{orbit}). In particular, given a pair $(A,H)$ whose properties are unclear, it may be possible to find another pair $(B,K) = (B,\Bbbk G)$ which is easier to understand, in part because the invariant theory of finite groups is better understood than that of semisimple Hopf algebras. In this case, the properties of $A^H$ and $B^K$ are closely related:

\begin{thm} \label{samelambda}
Suppose that $(A,H)$ and $(B,K)$ both satisfy Hypothesis \ref{mainhypothesis}, and that $A$ and $B$ are GK-Cohen-Macaulay, as in Theorem \ref{introbhz}. Moreover assume that the algebras obtained by Theorem \ref{pathalgthm} to $(A,H)$ and $(B,K)$ are equal, and call this common algebra $\Lambda$. Necessarily, we must have $\GKdim A = \GKdim B$; call this common value $d$.
\begin{enumerate}[{\normalfont (1)},leftmargin=*,topsep=0pt,itemsep=0pt]
\item The following are equivalent:
\begin{enumerate}[{\normalfont (i)},leftmargin=20pt,topsep=0pt,itemsep=0pt]
\item The Auslander map for $(A,H)$ is an isomorphism;
\item The Auslander map for $(B,K)$ is an isomorphism; and
\item $\GKdim \Lambda/\langle e_0 \rangle \leqslant d-2$.
\end{enumerate}
\item $A^H \cong B^K$.
\end{enumerate}
\end{thm}
\begin{proof}
\begin{enumerate}[{\normalfont (1)},wide=0pt,topsep=0pt,itemsep=0pt]
\item By Corollary \ref{austhmcor}, (i) and (iii) are equivalent, and (ii) and (iii) are equivalent by the same reasoning.
\item $A^H \cong e_0 \Lambda e_0 \cong B^K$. \qedhere
\end{enumerate}
\end{proof}

We will apply this result a number of times in Section \ref{examplessec}.

\subsection{$\Lambda$ is a mesh algebra when $A$ is two-dimensional}
Suppose the pair $(A,H)$ satisfies Hypothesis \ref{mainhypothesis}, and moreover assume that $A$ is a two-dimensional AS regular algebra. Throughout, fix an algebra $\Lambda$ obtained from Theorem \ref{pathalgthm}. We first recall a definition from \cite{reyesrog}, with some minor modifications for notational consistency:

\begin{defn} \label{meshdef}
Let $Q$ be a finite quiver with vertex set $\{0, \dots, n \}$. Let $\mathcal{A}$ be the set of arrows of $Q$, which is therefore a basis for $U \coloneqq (\Bbbk Q)_1$. Suppose there is a permutation $\tau$ of the vertex set (and hence of the vertex idempotents) and a bijective linear map $\sigma : U \to U$ such that $\sigma (e_i U e_j) = e_{\tau(j)} U e_i$. Then
\begin{align*}
\Phi \coloneqq \sum_{\alpha \in \mathcal{A}} \sigma(\alpha) \alpha 
\end{align*}
is a $\tau^{-1}$-twisted weak potential in $\Bbbk Q$. If we let $\Phi_i = \Phi e_i = e_{\tau(i)} \Phi e_i$, then the algebra
\begin{align*}
A_2(Q,\sigma) \coloneqq \Bbbk Q/\langle \Phi \rangle = \Bbbk Q/\langle \Phi_0, \dots, \Phi_n \rangle
\end{align*}
is called a \emph{mesh algebra}, and $\Phi_i$ is called a \emph{mesh relation}.
\end{defn}

By Theorem \cite[Theorem 4.1]{rrz}, the algebra $A \hash H$ is \emph{twisted Calabi-Yau} (called skew Calabi-Yau in [loc.\ cit.]), and therefore the same is true of the Morita equivalent algebra $\Lambda$. By \cite[Proposition 7.1]{reyesrog}, it follows that $\Lambda$ is a \emph{mesh algebra}, and in our setting the permutation $\tau$ satisfies $\Bbbk \sfw \otimes V_i \cong V_{\tau(i)}$. \\
\indent When the permutation $\tau$ is the identity, under mild assumptions on the quiver $Q$, mesh algebras are quite restricted:

\begin{lem} \label{meshalgebralemma}
Assume that $\Lambda$ is a mesh algebra, where the map $\tau$ is the identity. Assume that the underlying quiver $Q$ of $\Lambda$ is the double of a quiver $Q'$, where the underlying graph of $Q'$ is a tree. Then 
\begin{align*}
\Lambda \cong \Pi(Q'),
\end{align*}
where $\Pi(Q')$ is the preprojective algebra of $Q'$ (see Definition \ref{preprojdef}).
\end{lem}
\begin{proof}
The hypotheses on $\Lambda$ tell us that arrows come in opposed pairs; label them $\alpha : i \to j$, $\overline{\alpha} : j \to i$ in some order. In particular, if there is an arrow $i \to j$, then there is exactly one arrow $j \to i$, so if we apply the map $\sigma$ from Definition \ref{meshdef} to such a pair of arrows, then
\begin{align*}
\sigma(\alpha) = \lambda_\alpha \overline{\alpha}, \quad \sigma(\overline{\alpha}) = \mu_\alpha \alpha 
\end{align*}
for some $\lambda_\alpha, \mu_\alpha \in \Bbbk^\times$. Therefore the relation at vertex $i$ is of the form
\begin{align*}
\sum_{\alpha : t(\alpha) = i} \mu_\alpha \alpha \overline{\alpha} \hspace{5pt} + \hspace{-3pt} \sum_{\alpha : h(\alpha) = i} \lambda_\alpha \overline{\alpha} \alpha.
\end{align*}
Since the underlying graph of $Q$ is a tree, a simple adaptation of the proof of \cite[Lemma 7.1.2]{simon} shows that we can scale the arrows so that $\mu_\alpha = 1$ and $\lambda_\alpha = -1$ for all $\alpha$. These are precisely the relations in a preprojective algebra so $\Lambda \cong \Pi(Q)$, as claimed.
\end{proof}


The conditions of the above lemma are satisfied in a number of examples that we will consider in Section \ref{examplessec}. In particular, we only need to know the McKay quiver of the action of $H$ on $A$ to find $\Lambda$, up to isomorphism. We note that the condition that the underlying graph of $Q$ is a tree is crucial; for example, the relations obtained in Example \ref{bgexamples} cannot be written as preprojective relations.

\subsection{Maximal Cohen-Macaulay modules in dimension 2} \label{mcmsec}
\indent Suppose that the pair $(A,H)$ satisfies Hypothesis \ref{mainhypothesis} with $A$ two-dimensional, and assume that the Auslander map for this pair is an isomorphism. Let $\Lambda$ be an algebra from Theorem \ref{pathalgthm}. We now show how the maximal Cohen-Macaulay $A^H$-modules can be constructed using $\Lambda$. \\ 
\indent By Theorem \ref{mcmmodules}, if $V$ is an irreducible $H$-module, then $P_V \coloneqq A \otimes V$ is a projective $A \hash H$-module, and $M_V \coloneqq (A \otimes V)^H$ is a maximal Cohen-Macaulay $A^H$-module. If we have enumerated the irreducible representations of $H$ as $V_0, \dots, V_n$, we then write $P_i \coloneqq P_{V_i}$ and $M_i \coloneqq M_{V_i}$. It then follows that the modules $P_i[j]$ and $M_i[j]$, where $0 \leqslant i \leqslant n$ and $j \in \ZZ$, is a complete list of isomorphism classes of indecomposable graded projective $(A \hash H)$-modules and indecomposable graded maximal Cohen-Macaulay $A^H$-modules, respectively.

\begin{lem} \label{identifyMCM}
Assume the above setup. 
\begin{enumerate}[{\normalfont (1)},leftmargin=*,topsep=0pt,itemsep=0pt]
\item The maps 
\begin{align*}
A^H \to e_0(A \hash H)e_0, \quad b \mapsto e_0(b \hash 1)e_0, \hspace{30pt} e_0(A \hash H)e_0 \to A^H, \quad e_0(a \hash h)e_0 \mapsto \varepsilon(h) e_0 \cdot a
\end{align*}
are mutually inverse isomorphisms of graded algebras.
\item There is an isomorphism of left $A^H$-modules $M_i \cong e_0 (A \hash H) e_i$.
\end{enumerate}
\end{lem}
\begin{proof}
\begin{enumerate}[{\normalfont (1)},wide=0pt,topsep=0pt,itemsep=0pt]
\item This is \cite[4.3.4 Lemma]{montgomery2}.
\item First note that, since $V_i \cong H e_i$, we can write $M_i = (A \otimes H e_i)^H$. Let $\sum_{a,h} a \otimes he_i \in M_i$. Then 
\begin{align*}
e_0 (A \hash H) e_i \ni e_0 \bigg(\sum_{a,h} a \otimes h\bigg) e_i = e_0 \sum_{a,h} a \otimes h e_i = \varepsilon(e_0) \sum_{a,h} a \otimes h e_i = \sum_{a,h} a \otimes h e_i.
\end{align*}
Conversely, let $a \hash h \in A \hash H$. Then $e_0 (a \hash h) e_i = e_0 (a \hash he_i) \in A \otimes V_i$ and, for any $k \in H$, we have
\begin{align*}
k \cdot e_0 (a \hash he_i) = (k e_0) (a \hash he_i) = \varepsilon(k) e_0 (a \hash he_i),
\end{align*}
so $e_0 (a \hash h) e_i \in (A \otimes V_i)^H = M_i$. \qedhere
\end{enumerate}
\end{proof}

Now consider the quiver algebra $\Lambda = e (A \hash H) e$, with its subalgebra $e_0 \Lambda e_0 \cong A^H$. By Lemma \ref{identifyMCM} (2), we can view the indecomposable MCM $A^H$-modules as the modules $e_0 \Lambda e_i$. In terms of the quiver, this means that the indecomposable MCM $e_0 \Lambda e_0$-module $M_i = e_0 \Lambda e_i$ has a $\Bbbk$-basis consisting of paths starting at the vertex $0$ and ending at the vertex $i$. In forthcoming work, we will use this perspective to study the Auslander--Reiten theory of $A^H$.

\section{Examples} \label{examplessec}

In this section, we compute some examples which illustrate how Theorem \ref{pathalgthm} can be applied. The main difficulty in writing down a presentation for $\Lambda$ is determining the twisted superpotential $\Phi \in \Bbbk Q$. Fortunately, in many of the cases of interest to us, $\Phi$ can be computed with very little effort. \\
\indent We begin with an example where a nontrivial semisimple Hopf algebra acts on a two-dimensional AS regular algebra.

\begin{example}
Let $H$ be the Kac-Palyutkin Hopf algebra, which has presentation
\begin{align*}
H = \frac{\Bbbk \langle x,y,z \rangle}{\left\langle 
\begin{array}{c}
x^2 = 1 = y^2, \hspace{3pt} xy = yx, \\
zx = yz, \hspace{5pt} zy = xz, \\
z^2 = \frac{1}{2}(1+x+y-xy)
\end{array}\right \rangle},
\end{align*}
where the comultiplication and counit are defined as follows,
\begin{alignat*}{2}
& \Delta(x) = x \otimes x, \quad && \varepsilon(x) = 1, \\
& \Delta(y) = y \otimes y, \quad && \varepsilon(y) = 1,  \\
& \Delta(z) = \tfrac{1}{2}(1 \otimes 1 + x \otimes 1 + 1 \otimes y - x \otimes y)(z \otimes z), \hspace{15pt} && \varepsilon(z) = 1, 
\end{alignat*}
and the antipode satisfies $S = \id_H$. This is an 8-dimensional semisimple Hopf algebra which is isomorphic to neither a group algebra nor the dual of one. By \cite[Section 2]{three}, $H$ has four one-dimensional representations and one two-dimensional representation which we now describe. Letting $\omega \in \Bbbk$ be a primitive 4th root of unity, the five irreducible representations have the following matrix representations:
\begin{alignat*}{4}
\rho_{\alpha,\beta,\gamma} &: H \to \Bbbk, \qquad && x \mapsto \alpha, \quad && y \mapsto \beta, \quad && z \mapsto \gamma, \\
\rho &: H \to \text{Mat}_{2}(\Bbbk), \qquad 
&& x \mapsto \begin{pmatrix}
-1 & 0 \\ 0 & 1
\end{pmatrix}, \quad 
&& y \mapsto \begin{pmatrix}
1 & 0 \\ 0 & -1
\end{pmatrix}, \quad
&& z \mapsto \begin{pmatrix}
0 & 1 \\ 1 & 0
\end{pmatrix},
\end{alignat*}
where $(\alpha,\beta,\gamma)$ is one of $(1,1,1)$ (the trivial representation), $(1,1,-1)$, $(-1,-1,\omega)$, $(-1,-1,-\omega)$. In order, let $V_0, V_1, V_2,$ and $V_3$ be $H$-modules such that the associated matrix representation is the corresponding $\rho_{\alpha,\beta,\gamma}$, and suppose $V_i = \sspan\{ e_i \}$. Also let $V_4 = \sspan\{ e_4, f_4 \}$ be a two-dimensional $H$-module with associated matrix representation $\rho$. \\
\indent We now want to consider actions of $H$ on two-dimensional AS regular algebras. Let $V = \sspan\{u,v\}$ be a two-dimensional $H$-module whose associated matrix representation is also $\rho$, and form the tensor algebra $T_\Bbbk(V)$. By \cite[Section 2]{three}, if we set $\sfw = u^2 + v^2$ and $\sfw' = u^2-v^2$, then the AS regular algebras $A = \scrD(\sfw,0)$ and $B = \scrD(\sfw',0)$ become $H$-module algebras which satisfy Hypothesis \ref{mainhypothesis}. We now determine the algebra $\Lambda$ from Theorem \ref{pathalgthm} in both of these cases. \\
\indent We first compute the McKay quivers. Since the McKay quiver only depends on the action of $H$ on $A_1 = B_1$, we obtain the same McKay quiver in each case. By \cite[Theorem 3.5]{three}, we have
\begin{align*}
V \otimes V_0 \cong V_0, \quad V \otimes V_1 \cong V_1,\quad  V \otimes V_2 \cong V_2, \quad V \otimes V_3 \cong V_3, \quad V \otimes V_4 \cong V_0 \oplus V_1 \oplus V_2 \oplus V_3,
\end{align*}
and so the McKay quiver $Q$ is as follows:
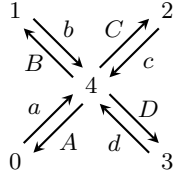
\begin{center}
\begin{tikzpicture}[->,>=stealth,thick,scale=1]

\node[circle,minimum size=0.5cm] (0) at (-1,-1) {};
\node[circle,minimum size=0.5cm] (1) at (-1,1) {};
\node[circle,minimum size=0.5cm] (2) at (1,1) {};
\node[circle,minimum size=0.5cm] (3) at (1,-1) {};
\node[circle,minimum size=0.5cm] (4) at (0,0) {};

\node at (-1,-1) {$0$};
\node at (-1,1) {$1$};
\node at (1,1) {$2$};
\node at (1,-1) {$3$};
\node at (0,0) {$4$};

\def \wiggle {20};
\draw (0.{45+\wiggle}) to (4.{225-\wiggle});
\draw (4.{225+\wiggle}) to (0.{45-\wiggle});
\draw (1.{-45+\wiggle}) to (4.{135-\wiggle});
\draw (4.{135+\wiggle}) to (1.{-45-\wiggle});
\draw (2.{225+\wiggle}) to (4.{45-\wiggle});
\draw (4.{45+\wiggle}) to (2.{225-\wiggle});
\draw (3.{135+\wiggle}) to (4.{-45-\wiggle});
\draw (4.{-45+\wiggle}) to (3.{135-\wiggle});

\node at (-0.75,-0.3) {$a$};
\node at (-0.75,0.3) {$B$};
\node at (0.75,0.3) {$c$};
\node at (0.75,-0.3) {$D$};

\node at (-0.3,-0.75) {$A$};
\node at (-0.3,0.75) {$b$};
\node at (0.3,0.75) {$C$};
\node at (0.3,-0.75) {$d$};

\end{tikzpicture}
\end{center}
By the results of Section \ref{smashisderquot}, $T_\Bbbk(V) \hash H$ is Morita equivalent to the path algebra of $Q$. \\
\indent We follow the recipe from Remark \ref{altmethod} to determine the superpotential $\Phi$ so that $A \hash H$ and $\Lambda \coloneqq \scrD(\Phi,0)$ are Morita equivalent. For each arrow $\alpha : i \to j$ we need to write down an $H$-module morphism $\xi_\alpha : V_j \to V^* \otimes V_i$. Since the antipode is the identity, the action of $H$ on the dual basis $\{u^*,v^*\}$ of $V^*$ is the same as the action on the basis $\{u,v\}$ of $V$. It is then straightforward to verify that possible choices for the maps $\xi_\alpha$ are as follows:
\begin{alignat*}{2}
\xi_a &: V_4 \to V^* \otimes V_0, \qquad && e_4 \mapsto u^* \otimes e_0, \quad f_4 \mapsto v^* \otimes e_0, \\
\xi_A &: V_0 \to V^* \otimes V_4, \qquad && e_0 \mapsto u^* \otimes e_4 + v^* \otimes f_4, \\
\xi_b &: V_4 \to V^* \otimes V_1, \qquad && e_4 \mapsto u^* \otimes e_1, \quad f_4 \mapsto -v^* \otimes e_1, \\
\xi_B &: V_1 \to V^* \otimes V_4, \qquad && e_1 \mapsto u^* \otimes e_4 - v^* \otimes f_4, \\
\xi_c &: V_4 \to V^* \otimes V_2, \qquad && e_4 \mapsto v^* \otimes e_2, \quad f_4 \mapsto \omega u^* \otimes e_2, \\
\xi_C &: V_2 \to V^* \otimes V_4, \qquad && e_2 \mapsto v^* \otimes e_4 - \omega u^* \otimes f_4, \\
\xi_d &: V_4 \to V^* \otimes V_3, \qquad && e_4 \mapsto v^* \otimes e_3, \quad f_4 \mapsto - \omega u^* \otimes e_3, \\
\xi_D &: V_3 \to V^* \otimes V_4, \qquad && e_3 \mapsto v^* \otimes e_4 + \omega u^* \otimes f_4.
\end{alignat*}
\indent It is easy to see that $x \cdot u^2 = u^2$ and $y \cdot v^2 = v^2$, and direct calculation using the coproduct shows that
\begin{align}
z \cdot u^2 = v^2, \quad z \cdot v^2 = u^2. \label{zaction}
\end{align}
Therefore $\Bbbk \sfw \cong V_0$, so the $H$-action on $A$ has trivial homological determinant. This means that the permutation $\tau$ of the vertices, defined by $\Bbbk \sfw \otimes V_i \cong V_{\tau(i)}$, is the identity, so $\Phi$ is a linear combination of paths of length $\ell = 2$ which start and end at the same vertex. To determine the coefficients of the terms in $\Phi$, we use (\ref{Phicoeffs}); for example, the coefficient of $aA$ can be calculated from the following composition:
\begin{gather*}
\arraycolsep=2pt
\begin{array}{ >{\centering\arraybackslash$} p{10pt} <{$} 
>{\centering\arraybackslash$} p{8pt} <{$} 
>{\centering\arraybackslash$} p{35pt} <{$} 
>{\centering\arraybackslash$} p{27pt} <{$} 
>{\centering\arraybackslash$} p{90pt} <{$} 
>{\centering\arraybackslash$} p{27pt} <{$} 
>{\centering\arraybackslash$} p{110pt} <{$}
>{\centering\arraybackslash$} p{20pt} <{$} 
>{\centering\arraybackslash$} p{14pt} <{$} }
V_0 & \to & V^{\otimes 2} \otimes V_0 & \xrightarrow{\id \otimes \xi_A} & V^{\otimes 2} \otimes V^* \otimes V_4 & \xrightarrow{\id \otimes \xi_a} & V^{\otimes 2} \otimes (V^*)^{\otimes 2} \otimes V_0 & \xrightarrow{\text{eval}} & V_0 \\
e_0 & \mapsto & \sfw \otimes e_0 & \xmapsto{\phantom{\id \otimes \xi_A}} & \sfw \otimes (u^* \otimes e_4 + v^* \otimes f_4) & \xmapsto{\phantom{\id \otimes \xi_a}} & \sfw \otimes (u^* \otimes u^* + v^* \otimes v^*) \otimes e_0 & \xmapsto{\phantom{\text{eval}}} & 2 e_0
\end{array}
\end{gather*}
so the coefficient of $aA$ in $\Phi$ is $2$. Similarly, to compute the coefficient of $Aa$, we consider the following:
\begin{gather*}
\arraycolsep=2pt
\begin{array}{ >{\centering\arraybackslash$} p{10pt} <{$} 
>{\centering\arraybackslash$} p{8pt} <{$} 
>{\centering\arraybackslash$} p{35pt} <{$} 
>{\centering\arraybackslash$} p{27pt} <{$} 
>{\centering\arraybackslash$} p{70pt} <{$} 
>{\centering\arraybackslash$} p{27pt} <{$} 
>{\centering\arraybackslash$} p{110pt} <{$}
>{\centering\arraybackslash$} p{20pt} <{$} 
>{\centering\arraybackslash$} p{14pt} <{$} }
V_4 & \to & V^{\otimes 2} \otimes V_0 & \xrightarrow{\id \otimes \xi_a} & V^{\otimes 2} \otimes V^* \otimes V_0 & \xrightarrow{\id \otimes \xi_A} & V^{\otimes 2} \otimes (V^*)^{\otimes 2} \otimes V_4 & \xrightarrow{\text{eval}} & V_4 \\
e_4 & \mapsto & \sfw \otimes e_4 & \xmapsto{\phantom{\id \otimes \xi_a}} & \sfw \otimes u^* \otimes e_0 & \xmapsto{\phantom{\id \otimes \xi_A}} & \sfw \otimes u^* \otimes (u^* \otimes e_4 + v^* \otimes f_4) & \xmapsto{\phantom{\text{eval}}} & e_4
\end{array}
\end{gather*}
which tells us that the coefficient of $Aa$ in $\Phi$ is $1$. Repeating this process, we obtain
\begin{align*}
\Phi = 2aA + 2bB + 2cC + 2dD + Aa + Bb + Cc + Dd.
\end{align*}
Since $A = \scrD(\sfw,0)$, the relations in $\Lambda$ are obtained by formal left differentiation with respect to paths of length $0$, i.e.\ the vertices. This means that no differentiation is required, and the relations are obtained by simply pre- and post-multiplying $\Phi$ by suitable vertex idempotents. Therefore $A \hash H$ is Morita equivalent to $\Lambda$, which is the path algebra of $Q$ with the relations (after rescaling)
\begin{align*}
aA = bB = cC = dD = Aa+Bb+Cc+Dd = 0.
\end{align*}
In particular, $\Lambda$ is a preprojective algebra of a $\widetilde{\mathbb{D}}_4$ quiver. Using Theorem \ref{samelambda}, it follows that $A^H \cong e_0 \Lambda e_0 \cong \Bbbk[u,v]^G$ where $G$ is a binary dihedral group of order $8$, which explains the observation in \cite[Theorem 2.1]{three} that $A^H$ is a commutative hypersurface. \\
\indent Now consider $\sfw' = u^2-v^2$, and $B = \scrD(\sfw',0)$; we wish to determine the twisted superpotential $\Phi'$ so that $B \hash H$ and $\Lambda' = \scrD(\Phi',0)$ are Morita equivalent. Using (\ref{zaction}) we see that $\Bbbk \sfw' \cong V_1$, so the $H$-action on $B$ has nontrivial homological determinant. It is straightforward to check directly (or by using \cite[Theorem 3.5]{three}) that the functor $\Bbbk \sfw' \otimes -$ swaps $V_0$ and $V_1$, swaps $V_2$ and $V_3$, and fixes $V_4$. Therefore, when applying (\ref{Phicoeffs}), we only need to consider paths between vertices 0 and 1, between vertices 2 and 3, and from vertex 4 to itself, all of length two. In particular, one can show in this case that
\begin{align*}
\Phi' = 2aB + 2bA - 2cD - 2dC + Aa + Bb - Cc - Dd,
\end{align*}
and so $B \hash H$ is Morita equivalent to the path algebra of $Q$ modulo the relations
\begin{align*}
aB = bA = cD = dC = Aa+Bb-Cc-Dd = 0.
\end{align*}
In particular, using Theorem \ref{samelambda} and \cite[Example 5.1]{bsw}, $B^H \cong e_0 \Lambda' e_0 \cong \Bbbk[u,v]^G$, where $G$ is a dihedral group of order $8$. Since $G$ is generated by reflections, $\Bbbk[u,v]^G$ is a polynomial ring, and hence the same is true of $B^H$, which was observed in \cite[Theorem 2.1]{three}.
\end{example}

\subsection{Calculating $\Lambda$ when $H$ is the dual of a group algebra} \label{dualofgroupsec}
Suppose that $H = (\Bbbk G)^*$ is the dual of a group algebra acting homogeneously on an $m$-Koszul AS regular algebra $A = \scrD(\sfw,i)$, where $\sfw \in V^{\otimes \ell}$. The action of $H$ is equivalent to $A$ being $G$-graded. In this situation, it is relatively straightforward to describe the quiver $Q$ of $\Lambda$, as well as the superpotential $\Phi$. We now describe this process. \\
\indent We first fix some notation. Let $\{ v_1, \dots, v_r \}$ be a basis of $V \coloneqq A_1$ and let $\{ f_g \mid g \in G\}$ be the basis of $H$ which is dual to the usual basis of $\Bbbk G$. By assumption $A$ (and hence $T_\Bbbk(V)$) is $G$-graded, and since the action of $H$ is homogeneous, we may assume that $\deg_G v_i = g_i$ for some $g_i \in G$. The Artin--Wedderburn decomposition of $H$ is simply $H = \bigoplus_{g \in G} H f_g$, so it follows that the full idempotent $e$ is simply $\sum_{g \in G} f_g$, which is just $1_H$. Therefore this set of idempotents is automatically closed under the left winding automorphism ${}_{\hdet} \Xi$. The algebra $H$ has $|G|$ (one-dimensional) irreducible representations $V_g \coloneqq H f_g$, and it is straightforward to check that $V_g \otimes V_h \cong V_{gh}$. In particular, we can (and will) label the vertices of the McKay quiver $Q$ by the elements of $G$, in contrast with our usual notation $\{0,1,\dots,n\}$ for the set of vertices of $Q$. \\
\indent We now determine the arrows of the McKay quiver $Q$. We have $V \cong \bigoplus_{i=1}^r V_{g_i}$, so if $h \in G$, then we have $V \otimes V_h \cong \bigoplus_{i=1}^r V_{g_i h}$. In particular, the McKay quiver has $r |G|$ arrows, which we denote $\alpha_{g_i,h} : g_i h \to h$. In particular, to each arrow $\alpha_{g_i,h}$ we can associate a morphism in $\Hom_H(V_{g_i h}, V \otimes V_h)$; we shall choose the map
\begin{align*}
\phi_{g_i,h} : V_{g_i h} \to V \otimes V_h, \quad \phi_{g_i,h}(f_{g_i h}) = v_i \otimes f_h,
\end{align*}
which is easily checked to be an $H$-morphism. It will also be convenient to label the arrow $\alpha_{g_i,h}$ with the basis element $v_i$; we claim that these labels can be used to immediately write down the twisted superpotential $\Phi$ defining $\Lambda$. \\
\indent First recall that, to each arrow $\alpha_{g_i,h}$, we can also associate an element of $\Hom_H(V_{h}, V^* \otimes V_{g_i h})$, and this map is dual to $\phi_{g_i,h}$ in the sense of (\ref{arrowidentity2}). Noting that $V$ has dual basis $\{ v_1^*, \dots, v_r^* \}$ which satisfies $\Bbbk v_i^* \cong V_{g_i^{-1}}$, it follows that we may choose this morphism to be the element
\begin{align*}
\xi_{g_i,h} : V_{h} \to V \otimes V_{g_i h}, \quad \xi_{g_i,h}(f_{h}) = v_i^* \otimes f_{g_i h}.
\end{align*}
Also recall that we have a permutation $\tau$ of the vertices induced by the functor $\Bbbk \sfw \otimes -$, and by Corollary \ref{hdetdual}, we have $\Bbbk \sfw \cong V_{\deg_G \sfw}$. By (\ref{startendvertex}), we know that the only paths appearing with a nonzero coefficient in $\Phi$ have length $\ell$ and are of the form $\tau(i) \to i$. In the present context, that means we only have to consider paths of length $\ell$ starting at a vertex labelled by $(\deg_G \sfw) h$ and ending at a vertex labelled by $h$, where $h$ is some group element. Suppose that $p = \beta_1 \beta_2 \dots \beta_\ell$ is such a path in $Q$, say
\begin{align*}
g_{i_1} g_{i_2} \dots g_{i_\ell} h \xrightarrow{\beta_1} g_{i_2} \dots g_{i_\ell} h \xrightarrow{\beta_2} \dots \xrightarrow{\beta_{\ell-1}} g_{i_\ell} h \xrightarrow{\beta_\ell} h, 
\end{align*}
where $g_{i_1} g_{i_2} \dots g_{i_\ell} = \deg_G\sfw$, and where the basis vector associated to $\beta_j$ is $v_{i_j}$. We claim that
\begin{align*}
[\Phi]_{\beta_1 \dots \beta_\ell} = [\sfw]_{v_{i_1} \dots v_{i_\ell}},
\end{align*}
where we recycle the notation from before Proposition \ref{RLprop} to pick out components with respect to a given basis. Indeed, by (\ref{Phicoeffs}), the following composition is equal to $[\Phi]_{\beta_1 \dots \beta_\ell} \id$:
\begin{gather*}
\arraycolsep=4pt
\begin{array}{  >{\centering\arraybackslash$} p{30pt} <{$}
>{\centering\arraybackslash$} p{8pt} <{$} 
>{\centering\arraybackslash$} p{40pt} <{$} 
>{\centering\arraybackslash$} p{20pt} <{$} 
>{\centering\arraybackslash$} p{70pt} <{$} 
>{\centering\arraybackslash$} p{25pt} <{$} 
>{\centering\arraybackslash$} p{110pt} <{$}
>{\centering\arraybackslash$} p{25pt} <{$}  }
V_{(\deg_G\sfw) h} & \to & V^{\otimes \ell} \otimes V_h & \xrightarrow{\xi_{\beta_\ell}} & V^{\otimes \ell} \otimes V^* \otimes V_{g_\ell h} & \xrightarrow{\xi_{\beta_{\ell-1}}} & V^{\otimes \ell} \otimes V^* \otimes V^* \otimes V_{g_{\ell-1} g_\ell h} & \xrightarrow{\xi_{\beta_{\ell-2}}}  \\
f_{(\deg_G\sfw) h} & \mapsto & \sfw \otimes f_h & \xmapsto{\phantom{\xi_{\beta_\ell}}} & \sfw \otimes v_{i_\ell}^* \otimes f_{g_\ell h} & \xmapsto{\phantom{\xi_{\beta_{\ell-1}}}} & \sfw \otimes v_{i_{\ell}}^* \otimes v_{i_{\ell-1}}^* \otimes f_{g_{\ell-1} g_\ell h} & \xmapsto{\phantom{\xi_{\beta_{\ell-2}}}} 
\end{array}\\[5pt]
\arraycolsep=4pt
\begin{array}{ >{\centering\arraybackslash$} p{15pt} <{$} 
>{\centering\arraybackslash$} p{17pt} <{$} 
>{\centering\arraybackslash$} p{120pt} <{$} 
>{\centering\arraybackslash$} p{20pt} <{$} 
>{\centering\arraybackslash$} p{80pt} <{$}}
\dots & \xrightarrow{\xi_{\beta_1}} & V^{\otimes \ell} \otimes (V^*)^{\otimes \ell} & \xrightarrow{\text{eval}} & V_{(\deg_G\sfw) h}, \\
\dots & \xmapsto{\phantom{\xi_{\beta_1}}} & \sfw \otimes (v_{i_\ell}^* \otimes \dots \otimes v_{i_1}^*) \otimes f_{(\deg_G\sfw) h} & \xmapsto{\phantom{\text{eval}}} & [\sfw]_{v_{i_1} \dots v_{i_\ell}} f_{(\deg_G\sfw) h}.
\end{array}
\end{gather*}
(Here, we have written $\xi_{\beta_j}$ rather than $\id_V^{\otimes \ell} \otimes \id_{V^*}^{\otimes \ell - j} \otimes \xi_{\beta_j}$.) This establishes the claim. We summarise this result as a theorem:
\begin{thm} \label{dualgroupthm}
Suppose that the pair $(A,H)$ satisfies Hypothesis \ref{mainhypothesis}, where $H = (\Bbbk G)^*$; equivalently, $A$ is $G$-graded. Let $V = A_1$ have basis $\{v_1, \dots, v_r \}$, and suppose that $\deg_G v_i = g_i$ for some $g_i \in G$. Then $A \hash H$ is isomorphic to $\Lambda \coloneqq \scrD(\Phi,\ell-m)$ for some twisted superpotential $\Phi \in (\Bbbk Q)_\ell$, where the McKay quiver $Q$ and $\Phi$ are as follows:
\begin{itemize}[leftmargin=25pt,topsep=0pt,itemsep=0pt]
\item The vertices of $Q$ are labelled by the elements of $G$, and there is an arrow of the form $g_i h \to h$ for each $1 \leqslant i \leqslant r$ and each $h \in G$; adorn such an arrow with $v_i$;
\item The twisted superpotential $\Phi \in (\Bbbk Q)_\ell$ is a linear combination of paths from the vertex labelled by $(\deg_G \sfw) h$ to the vertex labelled by $h$, for each $h \in G$. In particular, if $\lambda v_{i_1} v_{i_2} \dots v_{i_\ell}$ is a monomial (with coefficient $\lambda \in \Bbbk^\times$) appearing in $\sfw$ then, for each $h \in G$, there is a unique path $p$ in $Q$ from the vertex labelled by $(\deg_G \sfw) h$ to the vertex labelled by $h$ which traverses arrows adorned with $v_{i_1}, v_{i_2}, \dots, v_{i_\ell}$ (in order), and $p$ appears in $\Phi$ with coefficient $\lambda$.
\end{itemize}
Moreover, $A^H \cong e_{1_G} \Lambda e_{1_G}$. 
\end{thm}
\indent While the method for obtaining $Q$ and $\Phi$ was somewhat technical to describe, in practice, the calculations are quite easy.

\begin{example} \label{casedexample}
We consider case {(e)} from \cite[Table 1]{ckwzi}. Let $A = \Bbbk\langle u,v \rangle/\langle u^2 - v^2 \rangle = \scrD(\sfw,0)$, where $V = A_1 = \sspan\{ u,v\}$ and $\sfw = u^2 - v^2$. Let $G = D_n$ be a dihedral group of order $2n$, with presentation
\begin{align*}
D_n = \langle g,h \mid g^2 = h^2 = (gh)^n = 1 \rangle.
\end{align*}
Then $A$ and $T_\Bbbk(V)$ are $G$-graded by setting $\deg_G u = g$ and $\deg_G v = h$, which gives rise to a left action of $H = (\Bbbk G)^*$ on $A$ and $T_\Bbbk(V)$. In particular, by Corollary \ref{hdetdual}, the $H$-action has trivial homological determinant since $\deg_G \sfw = 1_G$. \\
\indent We first calculate the McKay quiver $Q$ of the action of $H$ on $A$. By Theorem \ref{dualgroupthm}, the vertices of $Q$ are labelled by the group elements. The arrows are of the following form, where here $x$ is some arbitrary element of $G$:
\begin{align*}
x \xrightarrow{\phantom{u}u\phantom{u}} g^{-1} x = gx, \qquad x \xrightarrow{\phantom{u}v\phantom{u}} h^{-1} x = hx,
\end{align*}
In the above, we have adorned the first arrow with $u$ since $\deg_G u = g$, and similarly for the second arrow. To write these arrows more explicitly, it will be convenient to write the $2n$ elements of $G$ in the following form:
\begin{align*}
G = \{ (hg)^i \mid 0 \leqslant i \leqslant n-1 \} \cup \{ g(hg)^i \mid 0 \leqslant i \leqslant n-1 \}.
\end{align*}
The arrows then take the following form:
\begin{gather*}
(hg)^i \xrightarrow{\phantom{u}u\phantom{u}} g \cdot (hg)^i = g(hg)^i, \qquad (hg)^i \xrightarrow{\phantom{u}v\phantom{u}} h \cdot (hg)^i = g(hg)^{i-1}, \\
g (hg)^i \xrightarrow{\phantom{u}u\phantom{u}} g \cdot g(hg)^i = (hg)^i, \qquad g (hg)^i \xrightarrow{\phantom{u}v\phantom{u}} h \cdot g (hg)^i = (hg)^{i+1}.
\end{gather*}
For example, below we show the McKay quiver when $n=3$, and on the right we provide a relabelling of the vertices and arrows:

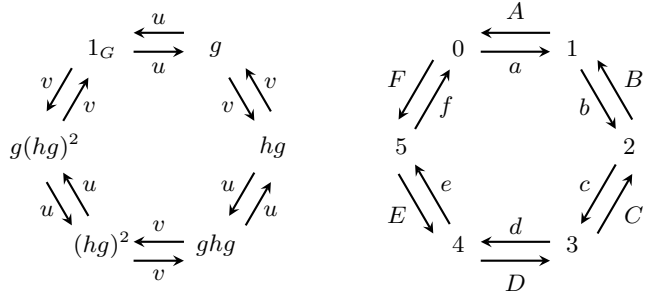
\begin{center}
\begin{tikzpicture}[->,>=stealth,thick,scale=1]

\node[regular polygon,regular polygon sides=6,minimum size=1.05cm] (0) at (4*360/6:  1.6cm) {\phantom{0}};
\node[regular polygon,regular polygon sides=6,minimum size=1.05cm]  (1) at (3*360/6:  1.6cm) {\phantom{0}};
\node[regular polygon,regular polygon sides=6,minimum size=1.05cm]  (2) at (2*360/6:  1.6cm) {\phantom{0}};
\node[regular polygon,regular polygon sides=6,minimum size=1.05cm]  (3) at (1*360/6:  1.6cm) {\phantom{0}};
\node[regular polygon,regular polygon sides=6,minimum size=1.05cm]  (4) at (0*360/6:  1.6cm) {\phantom{0}};
\node[regular polygon,regular polygon sides=6,minimum size=1.05cm]  (5) at (-1*360/6:  1.6cm) {\phantom{0}};

\node at (4*360/6:  1.5cm) {$(hg)^2$};
\node at (3*360/6:  1.5cm) {$g(hg)^2$};
\node at (2*360/6:  1.5cm) {$1_G$};
\node at (1*360/6:  1.5cm) {$g$};
\node at (0*360/6:  1.5cm) {$hg$};
\node at (-1*360/6:  1.5cm) {$ghg$};



\draw (0.105) to (1.-45) ;
\draw (1.45) to (2.255) ;
\draw (2.-15) to (3.195);
\draw (3.-75) to (4.135) ;
\draw (4.225) to (5.75) ;
\draw (5.165) to (0.15) ;

\draw (1.285) to (0.135);
\draw (2.225) to (1.75);
\draw (3.165) to (2.15) ;
\draw (4.105) to (3.-45);
\draw (5.45) to (4.255);
\draw (0.-15) to (5.195);

\node at (30+4*360/6:  1.05cm) {$v$}; 
\node at (30+3*360/6:  1.05cm) {$u$};
\node at (30+2*360/6:  1.05cm) {$v$};
\node at (30+1*360/6:  1.05cm) {$u$}; 
\node at (30+0*360/6:  1.05cm) {$v$};
\node at (30-1*360/6:  1.05cm) {$u$};

\node at (30+4*360/6:  1.7cm) {$v$}; 
\node at (30+3*360/6:  1.7cm) {$u$};
\node at (30+2*360/6:  1.7cm) {$v$};
\node at (30+1*360/6:  1.7cm) {$u$}; 
\node at (30+0*360/6:  1.7cm) {$v$};
\node at (30-1*360/6:  1.7cm) {$u$};

\node at (30+4*360/6:  1.85cm) {\phantom{777}}; 
\node at (30+1*360/6:  1.85cm) {\phantom{777}}; 

\end{tikzpicture}
\hspace{10pt}
\begin{tikzpicture}[->,>=stealth,thick,scale=1]

\node[regular polygon,regular polygon sides=6,minimum size=0.8cm] (0) at (4*360/6:  1.6cm) {\phantom{0}};
\node[regular polygon,regular polygon sides=6,minimum size=0.8cm]  (1) at (3*360/6:  1.6cm) {\phantom{0}};
\node[regular polygon,regular polygon sides=6,minimum size=0.8cm]  (2) at (2*360/6:  1.6cm) {\phantom{0}};
\node[regular polygon,regular polygon sides=6,minimum size=0.8cm]  (3) at (1*360/6:  1.6cm) {\phantom{0}};
\node[regular polygon,regular polygon sides=6,minimum size=0.8cm]  (4) at (0*360/6:  1.6cm) {\phantom{0}};
\node[regular polygon,regular polygon sides=6,minimum size=0.8cm]  (5) at (-1*360/6:  1.6cm) {\phantom{0}};

\node at (4*360/6:  1.5cm) {$4$};
\node at (3*360/6:  1.5cm) {$5$};
\node at (2*360/6:  1.5cm) {$0$};
\node at (1*360/6:  1.5cm) {$1$};
\node at (0*360/6:  1.5cm) {$2$};
\node at (-1*360/6:  1.5cm) {$3$};

\draw (0.100) to (1.-40) ;
\draw (1.40) to (2.260) ;
\draw (2.-20) to (3.200);
\draw (3.-80) to (4.140) ;
\draw (4.220) to (5.80) ;
\draw (5.160) to (0.20) ;

\draw (1.280) to (0.140);
\draw (2.220) to (1.80);
\draw (3.160) to (2.20) ;
\draw (4.100) to (3.-40);
\draw (5.40) to (4.260);
\draw (0.-20) to (5.200);



\node at (30+4*360/6:  1.05cm) {$d$}; 
\node at (30+3*360/6:  1.05cm) {$e$};
\node at (30+2*360/6:  1.05cm) {$f$};
\node at (30+1*360/6:  1.05cm) {$a$}; 
\node at (30+0*360/6:  1.05cm) {$b$};
\node at (30-1*360/6:  1.05cm) {$c$};

\node at (30+4*360/6:  1.8cm) {$D$}; 
\node at (30+3*360/6:  1.8cm) {$E$};
\node at (30+2*360/6:  1.8cm) {$F$};
\node at (30+1*360/6:  1.8cm) {$A$}; 
\node at (30+0*360/6:  1.8cm) {$B$};
\node at (30-1*360/6:  1.8cm) {$C$};

\node at (30+4*360/6:  1.85cm) {\phantom{777}}; 
\node at (30+1*360/6:  1.85cm) {\phantom{777}}; 

\end{tikzpicture}
\end{center}
\indent We now determine $\Phi$ using the recipe from Theorem \ref{dualgroupthm}. Since $\sfw \in V^{\otimes 2}$, the paths in $\Phi$ have length two, and since the homological determinant is trivial, the only ones appearing with a nonzero coefficient are loops. Now, $\sfw = u^2 - v^2$, so $e_{\tau(i)} \Phi e_{i} = e_i \Phi e_i$ consists of those paths from vertex $i$ to itself which correspond to the terms $u^2$ and $v^2$, with respective coefficient $+1$ and $-1$. For example, $e_0 \Phi e_0 = aA - Ff$, since $a$ and $A$ both correspond to the element $u$, while $f$ and $F$ both correspond to $v$. Continuing in this way, we find that
\begin{align*}
\Phi = aA - Ff - bB + Aa + cC - Bb - dD + Cc + eE - Dd - fF + Ee,
\end{align*}
which is a twisted superpotential where the twist is the identity. Therefore $A \hash H$ is isomorphic to $\Lambda = \scrD(\Phi,0)$ by Theorem \ref{dualgroupthm}. The relations are obtained by formal differentiation with respect to paths of length 0, i.e.\ the vertices. Therefore the relations in $\Lambda$ are simply
\begin{align*}
aA = Ff, \quad bB = Aa, \quad cC = Bb, \quad dD = Cc, \quad eE = Dd, \quad fF = Ee;
\end{align*}
that is, the loops of length two at a given vertex are equal. In particular, $\Lambda$ is the preprojective algebra of an $\widetilde{\mathbb{A}}_{5}$ quiver. \\
\indent By a similar argument, for $n \geqslant 3$, the McKay quiver of the pair $(A,H)$ is the double of a type $\widetilde{\mathbb{A}}_{2n-1}$ quiver, and the relations in $\Lambda$ say that the loops of length two at a given vertex are equal. That is, $\Lambda$ is the preprojective algebra of an $\widetilde{\mathbb{A}}_{2n-1}$ quiver.

\end{example}

\begin{example}
As a second example, let 
\begin{align*}
A = \frac{\Bbbk\langle u,v \rangle}{\langle u^2v - vu^2, \hspace{3pt} v^2u - uv^2 \rangle},
\end{align*}
a down-up algebra. By Example \ref{downupisderquot}, we know that $A \cong \scrD(\sfw,1)$, where $\sfw = uv^2u - u^2v^2 + v u^2 v - v^2u^2$. If we let $G = D_4$, the dihedral group of order 8, with presentation
\begin{align*}
G = \langle g,h \mid g^4 = h^2 = 1, hg = g^3h \rangle,
\end{align*}
then $A$ can be $G$-graded by setting $\deg_G u = g$ and $\deg_G v = h$. This gives a left action of $H = (\Bbbk D_4)^*$ on both $A$ and $\Bbbk \langle u,v \rangle = T_\Bbbk(V)$, where $V = \sspan\{u,v\}$. By Corollary \ref{hdetdual}, we know that the action of $H$ on $A$ has nontrivial homological determinant since
\begin{align*}
\deg_G \sfw = \deg_G uv^2u = g h^2 g = g^2 \neq 1_G.
\end{align*}
\indent To determine the McKay quiver $Q$ of the action of $H$ on $A$, we recall that the vertices correspond to elements of $G$, and that the arrows are of the form
\begin{align*}
x \xrightarrow{\phantom{u}u\phantom{u}} g^{-1} x = g^3x, \qquad x \xrightarrow{\phantom{u}v\phantom{u}} h^{-1} x = hx,
\end{align*}
for all $x \in G$. It is then easy to check that the McKay quiver has the form given on the left, and where we provide a relabelling on the right:

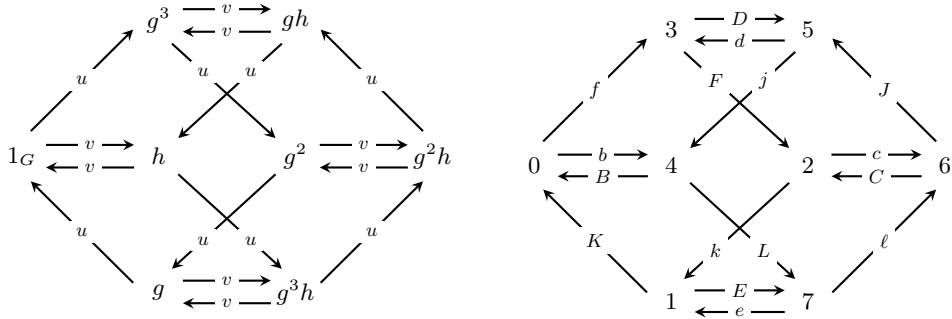
\begin{center}
\begin{tikzpicture}[->,>=stealth,thick,scale=1]

\def \tol {65}

\node[circle,minimum size = 0.68cm] (0) at (0,0) {};
\node[circle,minimum size = 0.68cm] (3) at (1.8,1.8) {};
\node[circle,minimum size = 0.68cm] (4) at (1.8,0) {};
\node[circle,minimum size = 0.68cm] (1) at (1.8,-1.8) {};
\node[circle,minimum size = 0.68cm] (5) at (3.6,1.8) {};
\node[circle,minimum size = 0.68cm] (2) at (3.6,0) {};
\node[circle,minimum size = 0.68cm] (7) at (3.6,-1.8) {};
\node[circle,minimum size = 0.68cm] (6) at (5.4,0) {};

\node[circle,minimum size = 0.68cm] at (0,0) {$1_G$};
\node[circle,minimum size = 0.68cm] at (1.8,1.8) {$g^3$};
\node[circle,minimum size = 0.68cm] at (1.8,0) {$h$};
\node[circle,minimum size = 0.68cm] at (1.8,-1.8) {$g$};
\node[circle,minimum size = 0.68cm] at (3.6,1.8) {$gh$};
\node[circle,minimum size = 0.68cm] at (3.6,0) {$g^2$};
\node[circle,minimum size = 0.68cm] at (3.6,-1.8) {$g^3h$};
\node[circle,minimum size = 0.68cm] at (5.4,0) {$g^2h$};

\draw (0.{90-\tol}) to node[midway,fill=white,scale=0.8]{$v$} (4.{90+\tol});
\draw (4.{270-\tol}) to node[midway,fill=white,scale=0.8]{$v$} (0.{270+\tol});
\draw (2.{90-\tol}) to node[midway,fill=white,scale=0.8]{$v$} (6.{90+\tol});
\draw (6.{270-\tol}) to node[midway,fill=white,scale=0.8]{$v$} (2.{270+\tol});
\draw (3.{90-\tol}) to node[midway,fill=white,scale=0.8]{$v$} (5.{90+\tol});
\draw (5.{270-\tol}) to node[midway,fill=white,scale=0.8]{$v$} (3.{270+\tol});
\draw (1.{90-\tol}) to node[midway,fill=white,scale=0.8]{$v$} (7.{90+\tol});
\draw (7.{270-\tol}) to node[midway,fill=white,scale=0.8]{$v$} (1.{270+\tol});

\draw (0.70) to node[midway,fill=white,scale=0.8]{$u$} (3.200);
\draw (6.110) to node[midway,fill=white,scale=0.8]{$u$} (5.340);
\draw (1.160) to node[midway,fill=white,scale=0.8]{$u$} (0.290);
\draw (7.20) to node[midway,fill=white,scale=0.8]{$u$} (6.250);

\draw (3.300) to node[pos=0.3,fill=white,scale=0.8]{$u$} (2);
\draw (5.240) to node[pos=0.3,fill=white,scale=0.8]{$u$} (4);
\draw (4) to node[pos=0.7,fill=white,scale=0.8]{$u$} (7.120);
\draw (2) to node[pos=0.7,fill=white,scale=0.8]{$u$} (1.60);

\end{tikzpicture}
\hspace{10pt}
\begin{tikzpicture}[->,>=stealth,thick,scale=1]

\def \tol {65}

\node[circle,minimum size = 0.65cm] (0) at (0,0) {};
\node[circle,minimum size = 0.65cm] (3) at (1.8,1.8) {};
\node[circle,minimum size = 0.65cm] (4) at (1.8,0) {};
\node[circle,minimum size = 0.65cm] (1) at (1.8,-1.8) {};
\node[circle,minimum size = 0.65cm] (5) at (3.6,1.8) {};
\node[circle,minimum size = 0.65cm] (2) at (3.6,0) {};
\node[circle,minimum size = 0.65cm] (7) at (3.6,-1.8) {};
\node[circle,minimum size = 0.65cm] (6) at (5.4,0) {};

\node[circle,minimum size = 0.65cm] at (0,0) {$0$};
\node[circle,minimum size = 0.65cm] at (1.8,1.8) {$3$};
\node[circle,minimum size = 0.65cm] at (1.8,0) {$4$};
\node[circle,minimum size = 0.65cm] at (1.8,-1.8) {$1$};
\node[circle,minimum size = 0.65cm] at (3.6,1.8) {$5$};
\node[circle,minimum size = 0.65cm] at (3.6,0) {$2$};
\node[circle,minimum size = 0.65cm] at (3.6,-1.8) {$7$};
\node[circle,minimum size = 0.65cm] at (5.4,0) {$6$};

\draw (0.{90-\tol}) to node[midway,fill=white,scale=0.8]{$b$} (4.{90+\tol});
\draw (4.{270-\tol}) to node[midway,fill=white,scale=0.8]{$B$} (0.{270+\tol});
\draw (2.{90-\tol}) to node[midway,fill=white,scale=0.8]{$c$} (6.{90+\tol});
\draw (6.{270-\tol}) to node[midway,fill=white,scale=0.8]{$C$} (2.{270+\tol});
\draw (3.{90-\tol}) to node[midway,fill=white,scale=0.8]{$D$} (5.{90+\tol});
\draw (5.{270-\tol}) to node[midway,fill=white,scale=0.8]{$d$} (3.{270+\tol});
\draw (1.{90-\tol}) to node[midway,fill=white,scale=0.8]{$E$} (7.{90+\tol});
\draw (7.{270-\tol}) to node[midway,fill=white,scale=0.8]{$e$} (1.{270+\tol});

\draw (0.70) to node[midway,fill=white,scale=0.8]{$f$} (3.200);
\draw (6.110) to node[midway,fill=white,scale=0.8]{$J$} (5.340);
\draw (1.160) to node[midway,fill=white,scale=0.8]{$K$} (0.290);
\draw (7.20) to node[midway,fill=white,scale=0.8]{$\ell$} (6.250);

\draw (3.300) to node[pos=0.3,fill=white,scale=0.8]{$F$} (2);
\draw (5.240) to node[pos=0.3,fill=white,scale=0.8]{$j$} (4);
\draw (4) to node[pos=0.7,fill=white,scale=0.8]{$L$} (7.120);
\draw (2) to node[pos=0.7,fill=white,scale=0.8]{$k$} (1.60);

\end{tikzpicture}
\end{center}

\indent We now determine $\Phi$. Since the homological determinant is nontrivial, the paths appearing in $\Phi$ start and end at different vertices. We already saw that $\deg_G\sfw = g^2$, which means that such a path beginning at a vertex labelled by $x \in G$ must end at the vertex labelled by $(g^2)^{-1} x = g^2 x$. For example, this means that $e_0 \Phi e_2$ is nonzero. Now observe that the path $fDdF$ from vertex $0$ to vertex $2$ in the right hand quiver corresponds to the element $uv^2u$ in the left hand quiver. Repeating this for the other terms in $\sfw$, we can obtain $e_0 \Phi e_2$:
\begin{align*}
\sfw = uv^2u - u^2v^2 + v u^2 v - v^2u^2 \quad \leadsto \quad e_0 \Phi e_2 = fDdF - fFcC + bL\ell C - bBfF.
\end{align*}
Continuing in this way, we find that
\begin{align*}
\Phi &= fDdF - fFcC + bL\ell C - bBfF 
      + KbBf - KfDd + E\ell Jd - EeKf \\
     &+ kEeK - kKbB + cJjB - cCkK 
      + FcCk - FkEe + DjLe - DdFk \\
     &+ LeE\ell - L\ell Cc + BfFc - BbL\ell 
      + jBbL - jLeE + dFkE - dDjL \\
     &+ JdDj - JjBb + CkKb - CcJj 
      + \ell CcJ - \ell JdD + eKfD - eE\ell J.
\end{align*}
This is a twisted superpotential, where the twist $\tau$ acts on the arrows as follows:
\begin{align*}
b \longleftrightarrow -c, \hspace{8pt} B \longleftrightarrow -C, \hspace{8pt} d \longleftrightarrow -e, \hspace{8pt} D \longleftrightarrow -E, \hspace{8pt} f \longleftrightarrow -k, \hspace{8pt} F \longleftrightarrow -K, \hspace{8pt} j \longleftrightarrow -\ell, \hspace{8pt} J \longleftrightarrow -L.
\end{align*}
By Theorem \ref{dualgroupthm}, we find that $A \hash H$ is isomorphic to $\Lambda = \scrD(\Phi,1)$. To obtain the relations in $\Lambda$, we formally differentiate $\Phi$ on the left with respect to each of the paths of length 1, i.e.\ the arrows. We therefore find that $A \hash H$ is isomorphic to the path algebra of the above quiver with the following sixteen relations:
\begin{align*}
\begin{array}{cccccc}
bBf = fDd, & bLl = fFc, & BbL = LeE, & BfF = L\ell C, & cCk = kEe, & cJj = kKb, \\
CcJ = JdD, & CkK = JjB, & dDj = jBb, & dFk = jLe, & DdF = FcC, & DjL = FkE, \\
& eE\ell = \ell Cc, & eKf = \ell Jd, & EeK = KbB, & E\ell J = KfD. &
\end{array}
\end{align*}

\end{example}

\subsection{Calculating $\Lambda$ when $H = \Bbbk G$ and $G$ is abelian}
Suppose that $G \leqslant \GL(r,\Bbbk)$ is a finite abelian group acting homogeneously on $A = \scrD(\sfw,\ell-m)$, where $\sfw \in V^{\otimes \ell}$ and $V = \sspan\{ v_1, \dots, v_r \}$. In particular, we can diagonalise $G$ via a change of basis matrix $P \in \GL(r,\Bbbk)$ to obtain $G' = P^{-1} G P$, although this also affects $\sfw$ and $A$. In particular, the element $\sfw$ is sent to $\sfw' = P^{-1} \sfw$, and the algebra $A' \coloneqq \scrD(\sfw',\ell-m)$ is isomorphic to $A$. \\
\indent Therefore, we may as well assume that $G \leqslant \GL(r,\Bbbk)$ is a finite abelian group where every element of $G$ is diagonal, and that $G$ acts on $A = \scrD(\sfw,\ell-m)$, where $\sfw \in V^{\otimes \ell}$ and $V = \sspan\{ v_1, \dots, v_r\}$. Now, $G$ has $|G| = n + 1$ irreducible representations $\{V_0, V_1, \dots, V_n \}$, and for each $j$ we have $\Bbbk v_j \cong V_{i_j}$ for some representation $V_{i_j}$. Since $G$ is abelian, the set of irreducible representations forms a group $\widehat{G}$ under the tensor product, and this group is isomorphic to $G$. Letting $\varphi : \widehat{G} \to G$ be an isomorphism, we can define a $G$-grading of $A$ and $T_\Bbbk(V)$ by setting 
\begin{align*}
\deg_G v_j = \varphi(V_{i_j}).
\end{align*}
This gives an action of $H = (\Bbbk G)^*$ on $A$ and $T_\Bbbk(V)$, and there is an isomorphism $A \hash G \cong A \hash H$. We can therefore determine the algebra $\Lambda$ for $A \hash G$ using Theorem \ref{dualgroupthm}

\begin{example} \label{bgexamples}
We consider cases {(b)} and {(g)} from \cite[Table 1]{ckwzi}. Fix $q \in \Bbbk^\times$, and let $A = \Bbbk_q[u,v] = \Bbbk \langle u,v \rangle/ \langle vu-quv \rangle$ for some $q \in \Bbbk^\times$, which is $2$-Koszul and AS regular. Also let 
\begin{align*}
G = C_n = \left \langle \hspace{4pt} \underbrace{\hspace{-5pt}\begin{pmatrix} \omega & 0 \\ 0 & \omega^{-1} \end{pmatrix} \hspace{-5pt}}_{=g} \hspace{4pt} \right \rangle
\end{align*}
be a faithful representation of the cyclic group of order $n$, where $\omega$ is a primitive $n$th root of unity. Here $A = \scrD(\sfw,0)$, where $V = \sspan\{u,v\}$ and $\sfw = vu - quv \in V^{\otimes 2}$. The irreducible representations of $G$ are
\begin{align*}
V_i = \sspan\{ v_i \}, \quad \text{where} \quad g \cdot v_i = \omega^i v_i, 
\end{align*}
for $0 \leqslant i < n$; in particular, $V \cong V_1 \oplus V_{n-1}$. There is an isomorphism
\begin{align*}
\varphi : \widehat{G} \to G, \qquad V_i \mapsto g^i,
\end{align*}
so $A$ is $G$-graded by setting $\deg_G u = g$ and $\deg_G v = g^{-1}$. This gives an action of $H = (\Bbbk G)^*$ on $A$, and $A \hash G \cong A \hash H$. \\
\indent The McKay quiver of the action of $H$ on $A$ has vertices labelled by the elements of $G$, and the arrows are of the form 
\begin{align*}
g^i \xrightarrow{\phantom{u}u\phantom{u}} g^{-1} g^i = g^{i-1}, \qquad g^i \xrightarrow{\phantom{u}v\phantom{u}} (g^{-1})^{-1} g^i = g^{i+1},
\end{align*}
for $0 \leqslant i < n$. Therefore the McKay quiver is as shown, where we provide a relabelling on the right: 
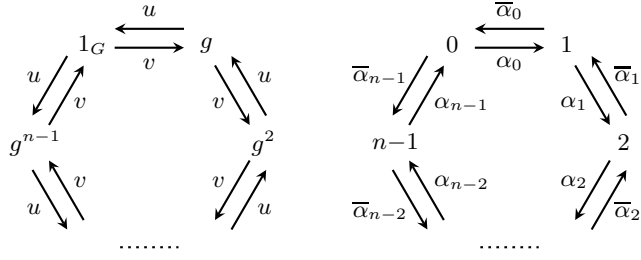
\begin{center}
\begin{tikzpicture}[->,>=stealth,thick,scale=1]

\node[regular polygon,regular polygon sides=6,minimum size=0.8cm] (0) at (4*360/6:  1.6cm) {\phantom{0}};
\node[regular polygon,regular polygon sides=6,minimum size=0.8cm]  (1) at (3*360/6:  1.6cm) {\phantom{0}};
\node[regular polygon,regular polygon sides=6,minimum size=0.8cm]  (2) at (2*360/6:  1.6cm) {\phantom{0}};
\node[regular polygon,regular polygon sides=6,minimum size=0.8cm]  (3) at (1*360/6:  1.6cm) {\phantom{0}};
\node[regular polygon,regular polygon sides=6,minimum size=0.8cm]  (4) at (0*360/6:  1.6cm) {\phantom{0}};
\node[regular polygon,regular polygon sides=6,minimum size=0.8cm]  (5) at (-1*360/6:  1.6cm) {\phantom{0}};

\node at (3*360/6:  1.5cm) {$g^{n-1}$};
\node at (2*360/6:  1.5cm) {$1_G$};
\node at (1*360/6:  1.5cm) {$g$};
\node at (0*360/6:  1.5cm) {$g^2$};

\draw (0.100) to (1.-40) ;
\draw (1.40) to (2.260) ;
\draw (2.-20) to (3.200);
\draw (3.-80) to (4.140) ;
\draw (4.220) to (5.80) ;

\draw (1.280) to (0.140);
\draw (2.220) to (1.80);
\draw (3.160) to (2.20) ;
\draw (4.100) to (3.-40);
\draw (5.40) to (4.260);

\draw[-,dash pattern={on 1pt off 2pt}] (0) to (5);

\node at (30+3*360/6:  1.05cm) {$v$};
\node at (30+2*360/6:  1.05cm) {$v$};
\node at (30+1*360/6:  1.05cm) {$v$}; 
\node at (30+0*360/6:  1.05cm) {$v$};
\node at (30-1*360/6:  1.05cm) {$v$};

\node at (30+3*360/6:  1.75cm) {$u$};
\node at (30+2*360/6:  1.75cm) {$u$};
\node at (30+1*360/6:  1.75cm) {$u$}; 
\node at (30+0*360/6:  1.75cm) {$u$};
\node at (30-1*360/6:  1.75cm) {$u$};

\node at (30+4*360/6:  1.85cm) {\phantom{777}}; 
\node at (30+1*360/6:  1.85cm) {\phantom{777}}; 

\end{tikzpicture}
\hspace{10pt}
\begin{tikzpicture}[->,>=stealth,thick,scale=1]

\node[regular polygon,regular polygon sides=6,minimum size=0.8cm] (0) at (4*360/6:  1.6cm) {\phantom{0}};
\node[regular polygon,regular polygon sides=6,minimum size=0.8cm]  (1) at (3*360/6:  1.6cm) {\phantom{0}};
\node[regular polygon,regular polygon sides=6,minimum size=0.8cm]  (2) at (2*360/6:  1.6cm) {\phantom{0}};
\node[regular polygon,regular polygon sides=6,minimum size=0.8cm]  (3) at (1*360/6:  1.6cm) {\phantom{0}};
\node[regular polygon,regular polygon sides=6,minimum size=0.8cm]  (4) at (0*360/6:  1.6cm) {\phantom{0}};
\node[regular polygon,regular polygon sides=6,minimum size=0.8cm]  (5) at (-1*360/6:  1.6cm) {\phantom{0}};

\node at (3*360/6:  1.5cm) {$n{-}1$};
\node at (2*360/6:  1.5cm) {$0$};
\node at (1*360/6:  1.5cm) {$1$};
\node at (0*360/6:  1.5cm) {$2$};

\draw (0.100) to (1.-40) ;
\draw (1.40) to (2.260) ;
\draw (2.-20) to (3.200);
\draw (3.-80) to (4.140) ;
\draw (4.220) to (5.80) ;

\draw (1.280) to (0.140);
\draw (2.220) to (1.80);
\draw (3.160) to (2.20) ;
\draw (4.100) to (3.-40);
\draw (5.40) to (4.260);

\draw[-,dash pattern={on 1pt off 2pt}] (0) to (5);

\node at (-0.63,-0.54) {$\alpha_{n{-}2}$};
\node at (-0.63,0.50) {$\alpha_{n{-}1}$};
\node at (30+1*360/6:  1.05cm) {$\alpha_0$}; 
\node at (0.85,0.52) {$\alpha_1$};
\node at (0.85,-0.52) {$\alpha_2$};

\node at (-1.7,0.89) {$\overline{\alpha}_{n{-}1}$};
\node at (30+1*360/6:  1.76cm) {$\overline{\alpha}_{0}$}; 
\node at (1.56,0.90) {$\overline{\alpha}_{1}$};
\node at (1.56,-0.90) {$\overline{\alpha}_{2}$};
\node at (-1.7,-0.91) {$\overline{\alpha}_{n-2}$};

\node at (30+4*360/6:  1.85cm) {\phantom{777}}; 
\node at (30+1*360/6:  1.85cm) {\phantom{777}}; 

\end{tikzpicture}
\end{center}

To determine $\Phi$, first note that the homological determinant of the action of $H$ on $A$ is trivial by \cite{ckwzbin}, but this also follows from the fact that $\deg_G \sfw = \id_G$. Therefore $\Phi$ consists of paths of length $\ell = 2$ which start and end at the same vertex. To determine $e_i \Phi e_i$, we note that $\alpha_i \overline{\alpha}_i$ corresponds to the element $vu$, while $\overline{\alpha}_{i-1} \alpha_{i-1}$ corresponds to the element $uv$, and since $\sfw = vu - quv$, we must have $e_i \Phi e_i = \alpha_i \overline{\alpha}_i - q \overline{\alpha}_{i-1} \alpha_{i-1}$. Therefore
\begin{align*}
\Phi = \sum_{i=0}^{n-1} \big( \alpha_i \overline{\alpha}_i - q \overline{\alpha}_{i-1} \alpha_{i-1} \big),
\end{align*}
where subscripts are read modulo $n$, if necessary. It follows that $A \hash H$ is isomorphic to $\Lambda = \scrD(\Phi,0)$, where the relations in $\Lambda$ are simply
\begin{align*}
\alpha_i \overline{\alpha}_i = q \overline{\alpha}_{i-1} \alpha_{i-1},
\end{align*}
for $0 \leqslant i < n$. In particular, when $q=1$, $\Lambda$ is the preprojective algebra of an $\widetilde{\mathbb{A}}_{n-1}$ quiver.
\end{example}

\begin{example} \label{L1example}
We consider case {(c)} from \cite[Table 1]{ckwzi}. Let $A = \Bbbk_{-1}[u,v] = \scrD(\sfw,0)$, where $\sfw = uv+vu$, and let the group
\begin{align*}
G = C_2 = \left \langle \hspace{4pt} \underbrace{\hspace{-5pt}\begin{pmatrix} 0 & 1 \\ 1 & 0 \end{pmatrix} \hspace{-5pt}}_{=g} \hspace{4pt} \right \rangle
\end{align*}
act on $A$ naturally. If we set $P = \frac{1}{\sqrt{2}} \begin{psmallmatrix} 1 & 1 \\ 1 & -1 \end{psmallmatrix}$ then
\begin{align*}
P^{-1}G P = \left \langle \begin{pmatrix} 1 & 0 \\ 0 & -1 \end{pmatrix} \right \rangle, \qquad P^{-1} \sfw = u^2 - v^2,
\end{align*}
and $\Bbbk \langle u,v \rangle/\langle u^2-v^2 \rangle = \scrD(P^{-1}\sfw,0) \cong A$. \\
\indent Relabelling, we may as well assume $A = \Bbbk \langle u,v \rangle/\langle u^2-v^2 \rangle$, $\sfw = u^2-v^2$, and that $G$ is generated by $g = \begin{psmallmatrix} 1 & 0 \\ 0 & -1 \end{psmallmatrix}$. Now, $G$ has two representations $V_0$ and $V_1$, where $V_0$ is trivial, and we have $V = \sspan\{ u,v \} \cong V_0 \oplus V_1$. There is an isomorphism
\begin{align*}
\varphi: \widehat{G} \to G, \qquad V_i \mapsto g^i,
\end{align*}
so $A$ is $G$-graded by setting $\deg_G u = 1_G$ and $\deg_G v = g$, and there is an isomorphism $A \hash G \cong A \hash (\Bbbk G)^*$. The McKay quiver of $A \hash (\Bbbk G)^*$ is then easily found using Theorem \ref{dualgroupthm}:

\begin{center}
\begin{tikzpicture}[->,>=stealth,thick,scale=1]

\node[minimum size=0.5cm] (0) at (0,0) {$\phantom{g}$};
\node at (0,0) {$1_G$};

\node [circle,minimum size=0.55cm](a) at (0,0) {};
\node [circle,minimum size=0.7cm](b) at ([{shift=(0:-0.4)}]a){};
\coordinate  (d) at (intersection 2 of a and b);
\coordinate  (c) at (intersection 1 of a and b);
 \tikzAngleOfLine(b)(d){\AngleStart}
 \tikzAngleOfLine(b)(c){\AngleEnd}
\draw[thick,<-]%
   let \p1 = ($ (b) - (d) $), \n2 = {veclen(\x1,\y1)}
   in   
     
     (d) arc (\AngleStart:\AngleEnd:\n2); 

\node[minimum size=0.5cm] (1) at (1.7,0) {$g$};

\node [circle,minimum size=0.55cm](aa) at (1.7,0) {};
\node [circle,minimum size=0.7cm](bb) at ([{shift=(0:0.4)}]aa){};
\coordinate  (cc) at (intersection 2 of aa and bb);
\coordinate  (dd) at (intersection 1 of aa and bb);
 \tikzAngleOfLine(bb)(dd){\AngleStart}
 \tikzAngleOfLine(bb)(cc){\AngleEnd}
\draw[thick,<-]%
   let \p1 = ($ (bb) - (dd) $), \n2 = {veclen(\x1,\y1)}
   in   
     
     (dd) arc (\AngleStart-360:\AngleEnd:\n2); 

\def \tol {32};
\draw (0.{0+\tol}) to node[fill=white]{$v$} (1.{180-\tol});
\draw (1.{180+\tol}) to node[fill=white]{$v$} (0.{0-\tol});

\node [fill=white] at (-1,0) {$u$};
\node [fill=white] at (1.7+1,0) {$u$};

\end{tikzpicture}
\hspace{10pt}
\begin{tikzpicture}[->,>=stealth,thick,scale=1]

\node[minimum size=0.5cm] (0) at (0,0) {$0$};

\node [circle,minimum size=0.55cm](a) at (0,0) {};
\node [circle,minimum size=0.7cm](b) at ([{shift=(0:-0.4)}]a){};
\coordinate  (d) at (intersection 2 of a and b);
\coordinate  (c) at (intersection 1 of a and b);
 \tikzAngleOfLine(b)(d){\AngleStart}
 \tikzAngleOfLine(b)(c){\AngleEnd}
\draw[thick,<-]%
   let \p1 = ($ (b) - (d) $), \n2 = {veclen(\x1,\y1)}
   in   
     
     (d) arc (\AngleStart:\AngleEnd:\n2); 

\node[minimum size=0.5cm] (1) at (1.7,0) {$1$};

\node [circle,minimum size=0.55cm](aa) at (1.7,0) {};
\node [circle,minimum size=0.7cm](bb) at ([{shift=(0:0.4)}]aa){};
\coordinate  (cc) at (intersection 2 of aa and bb);
\coordinate  (dd) at (intersection 1 of aa and bb);
 \tikzAngleOfLine(bb)(dd){\AngleStart}
 \tikzAngleOfLine(bb)(cc){\AngleEnd}
\draw[thick,<-]%
   let \p1 = ($ (bb) - (dd) $), \n2 = {veclen(\x1,\y1)}
   in   
     
     (dd) arc (\AngleStart-360:\AngleEnd:\n2); 

\def \tol {32};
\draw (0.{0+\tol}) to node[fill=white]{$b$} (1.{180-\tol});
\draw (1.{180+\tol}) to node[fill=white]{$B$} (0.{0-\tol});

\node [fill=white] at (-1,0) {$a$};
\node [fill=white] at (1.7+1,0) {$c$};

\end{tikzpicture}
\end{center}
Noting that $\sfw = u^2 - v^2$ and that the homological determinant is trivial, we can then immediately read of the superpotential $\Phi$:
\begin{align*}
\Phi = a^2 - bB + c^2 - Bb.
\end{align*}
It follows that the algebras $A \hash G$ and $\Lambda = \scrD(\Phi,0)$ are isomorphic. The relations in $\Lambda$ are
\begin{align*}
a^2 = bB, \quad c^2 = Bb.
\end{align*}
\end{example}

\begin{example} \label{KJexample}
As a final example, we consider case {(h)} from \cite[Table 1]{ckwzi}. Let $A = \Bbbk_{J}[u,v]$, where $\sfw = vu-uv-u^2$, and let the group
\begin{align*}
G = C_2 = \left \langle \hspace{4pt} \underbrace{\hspace{-5pt}\begin{pmatrix} -1 & 0 \\ 0 & -1 \end{pmatrix} \hspace{-5pt}}_{=g} \hspace{4pt} \right \rangle
\end{align*}
act on $A$ naturally. It is straightforward to see that this is equivalent to $G$-grading $A$ by setting $\deg_G u = g = \deg_G v$, giving an action of the Hopf algebra $(\Bbbk C_2)^*$ on $A$. The McKay quiver (and a relabelling) are as follows:

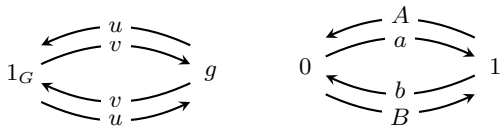
\begin{center}
\def \wiggle {25}
\begin{tikzpicture}[->,>=stealth,thick,scale=1]
\node[circle,minimum size=0.4cm] (0) at (0,0) {\phantom{0}};
\node[circle,minimum size=0.4cm] (1) at (2.5,0) {\phantom{0}};

\node[circle,minimum size=0.4cm] (0a) at (0,0.25) {\phantom{0}};
\node[circle,minimum size=0.4cm] (1a) at (2.5,0.25) {\phantom{0}};

\node[circle,minimum size=0.4cm] (0b) at (0,-0.25) {\phantom{0}};
\node[circle,minimum size=0.4cm] (1b) at (2.5,-0.25) {\phantom{0}};

\node at (0,0) {$1_G$};
\node at (2.5,0) {$g$};

\draw (0) to [out=\wiggle,in={180-\wiggle}] node[fill=white]{$v$} (1);
\draw (1a) to [out={180-\wiggle},in=\wiggle] node[fill=white]{$u$} (0a);

\draw (1) to [out={180+\wiggle},in={-\wiggle}] node[fill=white]{$v$} (0);
\draw (0b) to [out={-\wiggle},in={180+\wiggle}] node[fill=white]{$u$} (1b);

\end{tikzpicture}
\hspace{10pt}
\def \wiggle {25}
\begin{tikzpicture}[->,>=stealth,thick,scale=1]
\node[circle,minimum size=0.4cm] (0) at (0,0) {\phantom{0}};
\node[circle,minimum size=0.4cm] (1) at (2.5,0) {\phantom{0}};

\node[circle,minimum size=0.4cm] (0a) at (0,0.25) {\phantom{0}};
\node[circle,minimum size=0.4cm] (1a) at (2.5,0.25) {\phantom{0}};

\node[circle,minimum size=0.4cm] (0b) at (0,-0.25) {\phantom{0}};
\node[circle,minimum size=0.4cm] (1b) at (2.5,-0.25) {\phantom{0}};

\node at (0,0) {$0$};
\node at (2.5,0) {$1$};

\draw (0) to [out=\wiggle,in={180-\wiggle}] (1);
\draw (1a) to [out={180-\wiggle},in=\wiggle] (0a);

\draw (1) to [out={180+\wiggle},in={-\wiggle}] (0);
\draw (0b) to [out={-\wiggle},in={180+\wiggle}] (1b);

\draw node[fill=white] at (1.25, 0.68) {$A$};
\draw node[fill=white] at (1.25, 0.32) {$a$};

\draw node[fill=white] at (1.25, -0.32) {$b$};
\draw node[fill=white] at (1.25, -0.68) {$B$};

\end{tikzpicture}
\end{center}
It then quickly follows that the twisted superpotential defining $\Lambda$ is
\begin{align*}
\Phi = aA - Bb - BA + bB - Aa - AB.
\end{align*}
\end{example}

\subsection{Quantum Kleinian singularities} In \cite{ckwzbin}, the authors classified all pairs $(A,H)$ satisfying Hypothesis \ref{mainhypothesis}, where moreover $A$ is two-dimensional and the action of $H$ on $A$ has trivial homological determinant. These were further studied in \cite{ckwzi}, where the invariant rings $A^H$ were called \emph{quantum Kleinian singularities} due to their similarities with Kleinian singularities. Table \ref{qkstable} gives some of the details of the classification; full details can be found in \cite{ckwzi}. We remark that cases {(b)} and {(g)} were considered in Example \ref{bgexamples}, case {(c)} was considered in Example \ref{L1example}, case {(e)} was considered in Example \ref{casedexample}, and case {(h)} was considered in Example \ref{KJexample}. \\
\indent One of the main results of \cite{ckwzi} showed that the Auslander map is an isomorphism for every quantum Kleinian singularity. Additionally, presentations for the invariant rings $A^H$ were given, and it was shown in a number of cases that $A^H$ was isomorphic to the invariant ring of a commutative Kleinian singularity. We now show how these results follow for most of the cases using the quiver perspective (case {(a)} is well-known and therefore omitted). In what follows, given a pair $(A,H)$, we let $\Lambda$ denote an algebra obtained using Theorem \ref{pathalgthm}.


\begin{figure}[h]
\begin{tabular}{c|c|c|c|c}
Case & $A$ & $H$ & McKay quiver & Invariant ring $A^H$ \\ \hline\hline &&&& \\[-9pt]
{(a)} & $\Bbbk[u,v]$ & $\Bbbk G$, $G \leqslant \SL(2,\Bbbk)$ & $\widetilde{\mathbb{A}} \widetilde{\mathbb{D}} \widetilde{\mathbb{E}}$ & Commutative Kleinian singularity \\\hline &&&& \\[-9pt]
{(b)}, {(g)} & $\Bbbk_q[u,v]$ & $\Bbbk C_n$ & $\widetilde{\mathbb{A}}_{n-1}$ & $q$-deformed Type $\mathbb{A}_{n-1}$ singularity \\\hline &&&& \\[-9pt]
{(c)} & $\Bbbk_{-1}[u,v]$ & $\Bbbk C_2$ & $\widetilde{\mathbb{L}}_{1}$ & Noncommutative singularity \\\hline &&&& \\[-9pt]

\multirow{5}{*}{{(d)}\vspace{25pt}} & \multirow{5}{*}{$\Bbbk_{-1}[u,v]$\vspace{25pt}} & $\Bbbk D_n$, $n$ even  & $\widetilde{\mathbb{D}}_{\frac{n+4}{2}}$ & Type $\mathbb{D}_{\frac{n+4}{2}}$ singularity \\[-10pt] &&&& \\ \cline{3-5} &&&& \\[-9pt]

&  & $\Bbbk D_n$, $n$ odd  & $\widetilde{\mathbb{DL}}_{\frac{n+1}{2}}$ & Noncommutative singularity \\[-9pt] &&&& \\ \hline &&&& \\[-9pt]
{(e)} & $\frac{\Bbbk \langle u,v \rangle}{\langle u^2-v^2 \rangle}$ & $(\Bbbk D_n)^*$ & $\widetilde{\mathbb{A}}_{2n-1}$ & Type $\mathbb{A}_{2n-1}$ singularity \\[3pt]\hline &&&& \\[-9pt]
{(f)} & $\Bbbk_{-1}[u,v]$ & $\mathcal{D}(G)^*$ & $\widetilde{\mathbb{D}}_n$ or $\widetilde{\mathbb{E}}_n$ & Type $\mathbb{D}_{n}$ or $\mathbb{E}_{n}$ singularity \\\hline &&&& \\[-9pt]
{(h)} & $\Bbbk_{J}[u,v]$ & $\Bbbk C_2$ & $\widetilde{\mathbb{A}}_1$ & Noncommutative singularity 
\end{tabular}
\caption{The quantum Kleinian singularities.} \label{qkstable}
\end{figure}

\begin{thm}[{\cite[Theorem 4.1, Theorem 5.2]{ckwzi}}] \label{ckwzproof}
If $(A,H)$ is a pair from Table \ref{qkstable}, then the corresponding Auslander map is an isomorphism. Moreover, in cases \emph{{(d)}} ($n$ even), \emph{{(e)}}, and \emph{{(f)}}, the invariant ring $A^H$ is a commutative Kleinian singularity of the type corresponding to its McKay quiver.
\end{thm}
\begin{proof}
First consider case {(e)}. By Example \ref{casedexample}, the corresponding algebra $\Lambda$ is the preprojective algebra of an $\widetilde{\mathbb{A}}_{2n-1}$ quiver which, by Example \ref{KSexample}, is the same as the algebra obtained for a Type $\mathbb{A}_{2n-1}$ Kleinian singularity. Since, by \cite{purity}, the Auslander map is known to be an isomorphism for Kleinian singularities, it is also an isomorphism for case {(e)} by Theorem \ref{samelambda}. Moreover, $A^H$ is a Type $\mathbb{A}_{2n-1}$ Kleinian singularity, also by Theorem \ref{samelambda}. \\
\indent A similar argument applies to case {(d)} (when $n$ is even) and case {(f)}. By Lemma \ref{meshalgebralemma}, in each case the algebra $\Lambda$ corresponding to the pair $(A,H)$ is a preprojective algebra of an extended Dynkin quiver. The result now follows for these cases by the same argument as the previous paragraph. \\
\indent We now consider cases {(b)} and {(g)}. By Example \ref{bgexamples}, the algebra $\Lambda$ is the preprojective algebra of an $\widetilde{\mathbb{A}}_{2n-1}$ quiver. By Example \ref{bgexamples}, we have a presentation for $\Lambda$, and the underlying quiver is of Type $\widetilde{\mathbb{A}}_{n-1}$. The algebra $\Lambda/\langle e_0 \rangle$ is obtained from $\Lambda$ by deleting vertex $0$, and modifying the relations by deleting any path in a relation which passes through vertex $0$. In particular, $\Lambda/\langle e_0 \rangle$ satisfies the hypotheses of Lemma \ref{meshalgebralemma}, where $Q$ a Type $\mathbb{A}_{n-1}$ quiver, and so $\Lambda/\langle e_0 \rangle \cong \Pi(Q)$. This algebra is finite-dimensional by \cite[2.4 Corollary]{malkin}, so the Auslander map is an isomorphism by Corollary \ref{austhmcor}. \\
\indent For cases {(c)} and {(h)}, Examples \ref{L1example} and \ref{KJexample} allow us to see that $\Lambda/ \langle e_0 \rangle$ are two-dimensional and one-dimensional, respectively. Corollary \ref{austhmcor} then shows that the Auslander map is an isomorphism in these cases. \\
\indent The only remaining case is case {(d)} when $n$ is odd. By adapting the proof of Lemma \ref{meshalgebralemma}, one can show that $\Lambda$ is isomorphic to the preprojective algebra of a Type $\widetilde{\mathbb{DL}}_{\frac{n+1}{2}}$ quiver, in the sense of \cite[1.2]{malkin}, in which these are called quivers of Type $T$. (A little care is required here, since we have not defined the preprojective algebra when a quiver has loops. Alternatively, \cite[Theorem 7.2.11]{simon} gives a complete proof of this isomorphism.) The algebra $\Lambda/ \langle e_0 \rangle$ is then the preprojective algebra of an $\mathbb{L}_{\frac{n+1}{2}}$ quiver, which is finite-dimensional by \cite[2.4 Corollary]{malkin}. The claim now follows by Corollary \ref{austhmcor}.
\end{proof}

\bibliographystyle{amsalpha}
\bibliography{bibliography}

\end{document}